\documentclass{article}

\usepackage[margin=0.75in]{geometry}
\usepackage{amsmath}
\usepackage{amssymb}
\usepackage{amsthm}
\usepackage{verbatim}
\usepackage[dvipsnames]{xcolor}
\usepackage{xspace}
\usepackage[shortlabels]{enumitem}
\usepackage{graphicx}
\usepackage{mathtools}
\usepackage{stmaryrd}
\usepackage{makecell}
\usepackage{tikz}
\usepackage{multicol}

\usetikzlibrary{arrows,automata}
\definecolor{dark-green}{RGB}{0,136,0}
\newcommand{\bluecomment}[1]{\textcolor{blue}{#1}}

\newcommand{\keyterm}[1]{\textbf{#1}}

\newcommand{\Q}{\mathbb{Q}}
\newcommand{\N}{\mathbb{N}}
\newcommand{\R}{\mathbb{R}}
\newcommand{\NatZero}{{\mathbb{N}^0}}

\newcommand{\bff}{\mathbf{f}}
\newcommand{\bfl}{\mathbf{\ell}}
\newcommand{\bfu}{\mathbf{u}}

\newcommand{\bfx}{\mathbf{x}}

\newcommand{\eps}{\varepsilon}

\newcommand{\A}{\mathcal{A}}
\newcommand{\Tri}{\triangle}
\newcommand{\lang}{\mathcal{L}}

\newcommand{\rev}{R}

\newcommand{\udesub}[1]{\mathbf{u}^{({#1})}}

\newcommand{\emptyword}{\eps}

\newcommand{\makenest}[3]{\newcommand{#1}[1]{\ensuremath{\left#2##1\right#3}}}
\makenest{\abs}||
\makenest{\norm }\lVert\rVert
\makenest{\parens}()
\makenest{\bracket}[]
\makenest{\set}\{\}
\makenest{\clop}[)
\makenest{\opcl}(]
\makenest{\eval}.|
\makenest{\floor}\lfloor\rfloor
\makenest{\ceil}\lceil\rceil

\newtheorem{proposition}{Proposition}
\newtheorem{definition}{Definition}

\newtheorem{lemma}{Lemma}
\newtheorem{corollary}{Corollary}
\newtheorem{conjecture}{Conjecture}
\newtheorem{remark}{Remark}
\newtheorem{theorem}{Theorem}

\author{Thomas Garrity\footnote{Department of Mathematics and Statistics, Williams College, Williamstown, MA 01267, USA. Email: tgarrity@williams.edu} \and Otto Vaughn Osterman\footnote{Department of Mathematics, University of Maryland, College Park, MD 20742, USA. Email: ovo42@umd.edu} }

\title{On the Factor Complexity 
Associated with a Family of Multidimensional Continued Fraction Algorithms}

\date{\today}

\begin{document}

\maketitle

\begin{abstract} 
    We study the complexity of $S$-adic sequences corresponding to a family of 216 multidimensional continued fractions maps, called Triangle Partition maps (TRIP maps), with an emphasis on those with low upper bounds on complexity. Our main result is to prove that the complexity of $S$-adic sequences corresponding to the triangle map (called the $(e,e,e)$-TRIP map in this paper) has upper bound at most $3n$. Our second main result is to prove an upper bound of $2n+1$ on complexity for another TRIP map. We discuss a dynamical phenomenon, which we call ``hidden $\R^2$ behavior,'' that occurs in this map and its relationship to complexity. Combining this with previously known results and a list of counter-examples, we provide a complete list of the TRIP maps which have upper bounds on complexity of at most $3n$, except for one remaining case for which we conjecture such an upper bound to hold.
\end{abstract}

\section{Introduction} \label{SecIntro}

\subsection{Overview}

There is a remarkable correspondence between infinite words with  factor complexity $n+1$ (the number of subwords of length $n$ is exactly $n+1$) and continued fraction expansions of irrational numbers. (See Coven and Hedlund \cite{Coven-Hedlund} and also the work of Morse and Hedlund from the early 1940s \cite{Morse-Hedlund1, Morse-Hedlund2}. For a good expository introduction, see Chapter 19 in Goodson's textbook \cite{Goodson}.)
Given an irrational number $ a \in (0,1) \setminus \Q $, we can produce an infinite word on two letters with letter frequencies $(a,1-a)$ by associating the continued fraction expansion to a sequence of substitutions (for background on this see Chapter 6 in \cite{Fogg} or  in \cite{Thuswaldner}).
It is well-known that such  infinite words  are  \keyterm{Sturmian} (meaning that their factor compexity is  $n+1$, the lowest possible complexity for aperiodic infinite binary strings, as seen in the above references).

Over the years, there have been many generalizations of continued fractions, often called multidimensional continued fractions. Such generalizations often arise from the desire to find best simultaneous Diophantine approximations or to solve the Hermite problem, which asks for methods of writing a real number $\alpha$ as a sequence of integers so that the sequence is eventually periodic if and only if $\alpha$ is a cubic irrational. Just as continued fractions can be interpreted as a dynamical system of the Gauss map on the interval, multidimensional continued fractions can be interpreted as dynamical systems on a higher-dimensional simplex.

For higher-dimensional continued fractions, there is a straightforward technique mimicking the classical case which produces sequences on larger alphabets encoding orbits of the dynamical system. It is natural to wonder about the typical complexities of these sequences, whether there is any general framework for understanding them, and whether the correspondence between Sturmian words and words of low complexity can be generalized.

Arnoux-Rauzy words were among the first such sequences studied; these have factor complexity $2n+1$ and correspond to the application of the Arnoux-Rauzy multidimensional continued fraction algorithm in a set of Lebesgue measure zero, the so-called Rauzy Gasket \cite{AR}. Additionally, in \cite{ARP2} it was shown that generically the algorithm produces sequences of complexity less than or equal to $ \frac{5}{2} n + 1 $. Later, Cassaigne, Labb\'{e} and Leroy \cite{Cassaigne} showed that Cassaigne's multidimensional continued fraction algorithm produces words of factor complexity $2n+1$ almost everywhere. Given these two examples, it might be reasonable to wonder what the ``generic" complexity is for continued fraction algorithms. 

All of these words belong to a more general class of words called \keyterm{$S$-adic sequences}, which are derived via a finite family of substitutions. It is known that any uniformly recurrent infinite word with at most linear factor complexity is $S$-adic \cite{Ferenczi, Leroy1}. However, the converse is not true in general \cite{Fogg}. The \keyterm{$S$-adic conjecture} asks what additional conditions can be placed on $S$-adic sequences to ensure at most linear complexity.   (For remarkable progress on this conjecture, see the work of Espinoza \cite{Espinoza}.)

There is a framework (the triangle partition family, or TRIP family for short) that captures almost all known multidimensional continued fraction algorithms. This framework is generated by 216 different multidimensional continued fraction algorithms, indexed by a triple of permutations in $S_3\times S_3\times S_3$. Though \cite{Cassaigne} did not use the rhetoric of TRIP maps, the Cassaigne algorithm is in this list (it is the TRIP map $(e,23,23)$) \cite{SMALLTRIP}.

Our goal is to understand the complexity for all of the triangle partition maps with an emphasis on those with low upper bounds. Our main technical result is that the triangle map (called the $(e,e,e)$-TRIP map in this rhetoric) has linear complexity bounded between $2n+1$ and $3n$. While other work on this subject has often focused on other properties of such sequences such as balance (as seen for example in \cite{Cassaigne, Cassaigne2}), we do not address these topics here.  Also, while the Cassaigne map was created in large part to find a multi-dimensional continued fraction with low linear complexity, the triangle map has independent interest, as reviewed in the beginning of Section \ref{sec:e-e-e}.

We review the basics of multidimensional continued fractions in Section \ref{TRIP Maps} and the basics of $S$-adic sequences and factor complexity in Section \ref{$S$-adic sequences and Complexity}.
We discuss the link between multidimensional continued fractions and $S$-adic sequences in Section \ref{Correspondence}. In Section \ref{sec:pftech}, we review the common techniques used to determine factor complexity of $S$-adic sequences, which we make use of throughout the rest of the paper.

The new work begins in Section \ref{sec:e-e-e}. This will be the longest and most technical section of this paper, as we give the full proof that the linear complexity for the $S$-adic sequences corresponding to the $(e,e,e)$-TRIP map has lower bound $2n+1$ and upper bound $3n$.

In Section \ref{More general TRIP maps} we return to a discussion of more general TRIP maps and reduce the number of different TRIP maps and classes of $S$-adic sequences to consider from 216 to 21 through equivalence relations.

In Section \ref{Statements on Complexity for all TRIP maps}, we give an overview of each of these 21 classes, specifying those for which are known or for which we conjecture to have an upper bound on complexity of at most $3n$ and providing counterexamples for the others. Three of them are what we call \keyterm{degenerate TRIP maps} as they can be reduced to maps in $\R^2$, and their corresponding $S$-adic sequences are Sturmian sequences. We find through computational experiments that two others are candidates for an upper bound of $2n+1$ and two are candidates for an upper bound of $3n$ (one of which is the triangle map).

In Section \ref{sec:e-13-e}, we discuss a new phenomenon for some multidimensional continued fraction algorithms, which we call \keyterm{hidden $\R^2$ maps}, whereby they can be reduced to continued fraction maps in two dimensions on a positively invariant subset, with respect to Lebesgue measure, of $\triangle$ but not on all of $\triangle$. We discuss the $(e,13,e)$ TRIP map as an example of this and show how the complexity of its corresponding $S$-adic sequences relate to this phenomenon, with a full proof that the complexity has upper bound $2n+1$.

In Section \ref{sec:e-23-e}, we explore the other candidate for an upper bound of $3n$, the $(e,23,e)$ TRIP map, and give an outline of a possible proof for this upper bound similar to that in Section \ref{sec:e-e-e}.

We close with some natural questions in Section \ref{sec:conclusion}.

\subsection{Terse Summary}

In Section \ref{TRIP Maps}, we define the $216$ different $\R^3$ Triangle Partition Maps, which provide a framework to understand almost all known multidimensional continued fraction algorithms \cite{SMALLTRIP, Cubic}.  For example, the Cassaigne algorithm is the $(e,23,23)$ map, the M\"{o}nkemeyer map is the $(e, 132, 23) $ map and the Triangle map is $(e,e,e)$, though none of the papers which defined this maps used this rhetoric.

The two  papers \cite{SMALLTRIP, Cubic} were not concerned with questions of factor complexity at all.   It appears that all $216$ Triangle Partition Maps could be relevant for number-theoretic reasons.    But if we are only concerned with questions of factor complexity, we show in Section \ref{More general TRIP maps} that we can reduce our study to just $21$ different maps, using what we call conjugacy and twinning.  (All of the equivalence, for all $216$ Triangle Partition Maps, are in the appendix.)  The main goal of this paper is in studying the factor complexity of the Triangle Map $(e,e,e)$ in Section \ref{sec:e-e-e}.

Below is a summary of our results on factor complexity:

$$\begin{array}{c|c|c}
\mbox{Type} & \mbox{Theorem/Conjecture} & \mbox{Bound/Conjectured Bound} \\
\hline

(e,e,e) & \mbox{Theorem} \; \ref{Main theorem for $eee$}   & 2n+1 \leq p_{\lang}(n) \leq 3n \\
\hline
\begin{array}{c}(e,12,e)  \\
(e,12,13) \\ (e, 132, e) \end{array}
& \mbox{Theorem}\;  \ref{th:degenerate}  & p_{\lang}(n) = \begin{cases} n+1, & n \neq 1 \\ 3, & n = 1. \end{cases}   \\
\hline
(e,13,e) &  \mbox{Theorem}\;  \ref{conj:e-13-e}  & \begin{array}{l} \mbox{One of the following is true:}\\

    p_{\lang}(n) = 2n+1   \\
p_{\lang}(n) = \min\{2n+1, n+c\} \; \mbox{ for some $ c \in \N$} \\
      p_{\lang}(n) = \min\{2n+1, n+c_1, c_2\} \;  \mbox{ for some $ c_1,c_2 \in \N $}
       
    \end{array}
     \\

\hline
(e,23,e) & \mbox{Conjecture}\;  \ref{conj:e-23-e}& p_{\lang}(n) \leq 3n \\
\hline
\begin{array}{c} (e,123,e)   \\ (e,e,12) \\   (e,12,12) \\ (e,13,12)   \\ (e,23,12)  \\   (e,123,12) \\ (e,e,13)  \\ (e,23,13) \\ (e, 123, 13) \\  (e,e,23)  \\ (e,123,23)\\ (e,e,123)  \\ (e,123,123) \\ (e,e,132) \end{array}&  \mbox{Theorem} \; \ref{th:large-complexity} & p_{\lang}(n) >3n \; \mbox{for some $n\geq 1$}  \\
\hline

   (e,23,23)  &  \mbox{Theorem} \;  \ref{Cassaigne} &  p_{\lang}(n) = 2n+1  \\
        
\end{array}$$

\subsection{Acknowledgements}

We would like to thank Mary Stelow for invaluable contributions to this work.  Also, the discover of hidden $\R^2$ TRIP maps, discussed in Section \ref{sec:e-13-e-dyn}, was done in collaboration with  D. Alvarez, A. Bradford, D. Dong, K. Herbst, A. Koltun-Fromm, B. Mintz, and M. Stelow \cite{Hidden}, for which we thank them. We would also like to thank L. Pedersen for useful comments.  Finally, we would like to thank the referee for an amazing number of useful comments.

This work was completed at the SMALL 2019 REU program at Williams College, supported by NSF grant DMS1659037.

This material is based upon work supported by the National Science Foundation Graduate Research
Fellowship Program under Grant No. DGE 1840340.

\section{Triangle Partition Maps} \label{TRIP Maps}

\subsection{$\R^3$ TRIP maps}\label{R3 TRIP Maps}

We will now briefly introduce the family of Triangle Partition Maps, or TRIP maps for short.  As this is background, much of this section is similar to certain  sections in \cite{stern-trip, GarrityT05, SMALLTRIP, Cubic, Triangle, Karpenkov, Schweiger05}.  For a general background on multi-dimensional continued fractions and an overview of why they are important, see Schweiger \cite{contfrac} or Karpenkov \cite{Karpenkov}.

We will concentrate in this paper on $\R^3$ TRIP maps.   (The traditional continued fraction map is a type of $\R^2$ TRIP map, which we describe in Subsection \ref{R2TRIP Maps}.)
The set of $\R^3$ TRIP maps is an attempt to put  all $3$-dimensional continued fraction algorithms into a structured family. 

We will start by defining the $\R^3(e,e,e)$ TRIP map, which is usually called the triangle map. This map is of independent interest in dynamics, as seen in the recent work  of Berth\'{e},  Steiner and Thuswaldner \cite{Berthe-Steiner-Thuswaldner}, of  Bonanno, Del Vigna and Munday \cite{Bonanno-Del Vigna-Munday}, of Bonanno and Del Vigna \cite{Bonanno-Del Vigna}, of  Fougeron and  Skripchenko \cite{Fougeron-Skripchenko} and of Ito \cite{Ito}.
     
   Fix two matrices
     \begin{equation}
    F_0 = \begin{bmatrix} 0 & 0 & 1 \\ 1 & 0 & 0 \\ 0 & 1 & 1 \end{bmatrix}, \qquad F_1 = \begin{bmatrix} 1 & 0 & 1 \\ 0 & 1 & 0 \\ 0 & 0 & 1 \end{bmatrix},
\end{equation}
and let $C$ be a cone in $\R^3$ spanned by three linearly independent vectors $v_1, v_2, v_3$,
$$C = \{a_1 v_1 + a_2 v_2 + a_3 v_3 : a_1, a_2, a_3 \geq 0\}.$$
We  will often identify this cone $C$ with the $3\times 3$ matrix $V=(v_1, v_2, v_3)$, treating the three vectors $v_1, v_2, v_3$  as column vectors in $\R^3.$ 
We start with subdividing the cone $C$ into two subcones: $C_0$ being the cone spanned by the column vectors $v_2, v_3, v_1 + v_3$ and $C_1$ the cone spanned by $v_1, v_2, v_1+v_3.$
Note that 
$$(v_2, v_3, v_1+v_3) = V F_0, \qquad (v_1, v_2, v_1+v_3) = V F_1.$$

The $\R^3(e,e,e)$ TRIP map $ T(e,e,e): C - (C_0 \cap C_1) \rightarrow C $
is the map
\begin{equation*}   
    T(e,e,e) \bfx = \begin{cases} T_0(e,e,e) \bfx, & \bfx \in C_0 \\
    T_1(e,e,e) \bfx, & \bfx \in C_1, \end{cases}  
\end{equation*}
where
\begin{equation*}
    T_0(e,e,e)  =  VF_0^{-1}V^{-1}, \qquad T_1(e,e,e) =  VF_1^{-1}V^{-1}.
\end{equation*}
Note that $T(e,e,e)$ is defined everywhere except on a set of measure zero.

Given the cone $C$, the map $T(e,e,e)$ depends on a few choices.  First, we had to choose an ordering on the vectors $v_1, v_2, v_3$.   Reordering these vectors will give us a new map.  We also had to choose how to order the vectors $v_2, v_3, v_1+v_3$ for $C_0$ and the vectors $v_1, v_2, v_1+v_3$ for $C_1$
As there are $6=|S_3|$ different orderings for each of these three collections of vectors, we have $ 6^3 = 216 $ different maps.

Here are the details. Let us write each permutation  $\sigma \in S_3$ as 
a $ 3 \times 3 $ matrix,
$$\begin{array}{ccc}
    e = \left(\begin{array}{ccc}
        1&0&0\\
        0&1&0\\
        0&0&1
    \end{array}\right), &
    (12) = \left(\begin{array}{ccc}
        0&1&0\\
        1&0&0\\
        0&0&1
    \end{array}\right), &
    (13) =\left( \begin{array}{ccc}
        0&0&1\\
        0&1&0\\
        1&0&0
    \end{array}\right), \\
    (23) =  \left(  \begin{array}{ccc}
        1&0&0\\
        0&0&1\\
        0&1&0
    \end{array}\right), &
    (123) =   \left(   \begin{array}{ccc}
        0&0&1\\
        1&0&0\\
        0&1&0
    \end{array}\right), &
    (132) =   \left(    \begin{array}{ccc}
        0&1&0\\
        0&0&1\\
        1&0&0
    \end{array}\right). 
\end{array}$$
For $\sigma, \tau_0, \tau_1\in S_3$, set
\begin{eqnarray*} 
    F_0(\sigma, \tau_0, \tau_1) &=&  \sigma F_0 \tau_0, \\
    F_1(\sigma, \tau_0, \tau_1) &=&  \sigma F_1 \tau_0, \\
   C_0 (\sigma, \tau_0, \tau_1) &=& V F_0(\sigma, \tau_0, \tau_1), \\
    C_1 (\sigma, \tau_0, \tau_1) &=&V F_1(\sigma, \tau_0, \tau_1),
    \end{eqnarray*}

Then, the $\R^3 (\sigma, \tau_0, \tau_1) $ TRIP map 
$$T(\sigma, \tau_0, \tau_1): C - ( C_0 (\sigma, \tau_0, \tau_1) \cap  C_1 (\sigma, \tau_0, \tau_1) ) \rightarrow C$$
is
\begin{equation*}    
    T(\sigma, \tau_0, \tau_1) \bfx = \begin{cases} T_0(\sigma, \tau_0, \tau_1) \bfx, & \bfx \in  C_0(\sigma, \tau_0, \tau_1)  \\
    T_1(\sigma, \tau_0, \tau_1) \bfx, & \bfx \in  C_1(\sigma, \tau_0, \tau_1), \end{cases}
\end{equation*}
where now
\begin{equation*}
    T_0(\sigma, \tau_0, \tau_1) = V(F_0(\sigma, \tau_0, \tau_1))^{-1}V^{-1} , \qquad T_1(\sigma, \tau_0, \tau_1) = V(F_1(\sigma, \tau_0, \tau_1))^{-1}V^{-1}.
\end{equation*}
Again, this map is defined everywhere except on a set of measure zero.

We now have a list of 216 maps. For ease of notation, we often omit parentheses when referring to specific permutations to define TRIP maps, substitutions, or Farey matrices, and often omit the $T$. For example, we write $(e,13,23)$ instead of $T(e,(13),(23))$.

As mentioned before, $(e,e,e)$ is the triangle map. The Cassaigne map \cite{Cassaigne} can be shown to be $ (e,23,23). $ The classical M\"{o}nkemeyer map \cite{contfrac, PantiG00} was shown in the three dimensional case to be  $(e,132,23)$ in \cite{SMALLTRIP}. This list, via composition, can be used to create a family of multi-dimensional continued fractions, which are called combination TRIP maps.  As seen in \cite{SMALLTRIP}, (almost) all known multidimensional continued fraction algorithms are combination TRIP maps.
 
For almost every point $ \bfx \in C $ and a map $T = T(\sigma,\tau_0,\tau_1)$, we can associate the following sequences:

\begin{definition} \label{defFareySeq}
    The \keyterm{Farey Sequence} associated to the point $\bfx$ and the map $T$ is the unique sequence $\{i_n\}_{n \geq 0}$, where each $i_n$ is either $0$ or $1$, such that if $T^n(x) \in C_i$ then $i_n = i$.
      If the Farey Sequence $\set{i_n}_{n \geq 0}$ contains $0$ infinitely many times, then we associate a \keyterm{Gauss Sequence} $\{k_n\}_{n \geq 0}$, where $\{i_n\}=1^{k_0}01^{k_1}01^{k_2}0\hdots$.
\end{definition}

We will use $ V = I $, corresponding to the standard basis in $\R^3$. (Often, $V$ is chosen to be 
$ \begin{bmatrix} 1 & 1 & 1 \\ 0 & 1 & 1 \\ 0 & 0 & 1 \end{bmatrix}, $ which is a more natural choice if we would  be studying TRIP maps as giving methods of division, which we are not doing here.)

All TRIP maps on the $3$-dimensional cone $C$ naturally project to maps on a $2$-dimensional triangle $\triangle$ in projective space. Setting $ V = I $, we take
$$ \triangle := C \cap \set{x+y+z=1} = \set{(x,y,z) \in \R^3: x \geq 0, \; y \geq 0, \; z \geq 0, \; x+y+z = 1}. $$
We subdivide $\triangle$ into two subtriangles
$$ \triangle_0(\sigma,\tau_0,\tau_1) = C_0(\sigma,\tau_0,\tau_1) \cap \triangle, \qquad \triangle_1(\sigma,\tau_0,\tau_1) = C_1(\sigma,\tau_0,\tau_1) \cap \triangle, $$
and more generally,
$$ \triangle_{i_0 \hdots i_{m-1}}(\sigma, \tau_0, \tau_1) = C_{i_0 \hdots i_{m-1}}(\sigma, \tau_0, \tau_1) \cap \triangle, $$
where we define 
\begin{equation*}
    C_{i_0 i_1 \hdots i_{m-1}}(\sigma,\tau_0,\tau_1) = V F_{i_0}(\sigma,\tau_0,\tau_1) F_{i_1}(\sigma,\tau_0,\tau_1) \hdots F_{i_{m-1}}(\sigma,\tau_0,\tau_1).
\end{equation*}
The TRIP map
$$ T(\sigma,\tau_0,\tau_1): \triangle - \parens{\triangle_0(\sigma, \tau_0, \tau_1) \cap \triangle_1(\sigma, \tau_0, \tau_1)} \rightarrow \triangle $$
is then
\begin{equation*}    
    T(\sigma, \tau_0, \tau_1) \bfx = \begin{cases} c T_0(\sigma, \tau_0, \tau_1) \bfx, & \bfx \in  C_0(\sigma, \tau_0, \tau_1)  \\
    c T_1(\sigma, \tau_0, \tau_1) \bfx, & \bfx \in  C_1(\sigma, \tau_0, \tau_1), \end{cases}
\end{equation*}
where the scalar multiple $ c = c(\sigma,\tau_0,\tau_1,\bfx) $ is chosen so that the point lands in $\triangle$.

Consider, for example, the TRIP map $T(e,e,e)$. We have
\begin{equation*}
    \triangle_0(e,e,e) = \set{(x,y,z) \in \triangle: x \leq z}, \qquad \triangle_1(e,e,e) = \set{(x,y,z) \in \triangle: x \geq z},
\end{equation*}
and
\begin{equation*}
    T(e,e,e)(x,y,z) = \begin{cases}
        \parens{\frac{y}{y+z}, \frac{z-x}{y+z}, \frac{x}{y+z}}, & (x,y,z) \in \triangle_0 \\
        \parens{\frac{x-z}{x+y}, \frac{y}{x+y}, \frac{z}{x+y}}, & (x,y,z) \in \triangle_1.
    \end{cases}
\end{equation*}

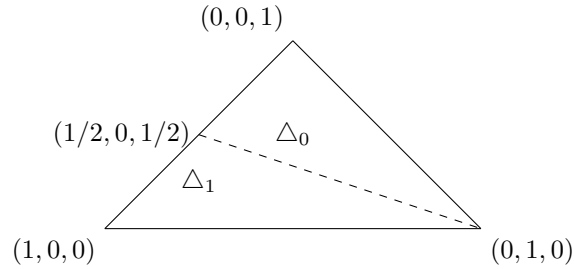
\begin{center}
\begin{tikzpicture}[scale=5]
\draw(0,0)--(1,0);
\draw(0,0)--(1/2,1/2);
\draw(1,0)--(1/2,1/2);
\draw[dashed] (1/4,1/4)--(1,0);
\node[] at (1/2,1/4){$\triangle_0$};
\node[] at (1/4,1/8){$\triangle_1$};

\node[below left]at(0,0){$(1,0,0)$};
\node[below right]at(1,0){$(0,1,0)$};

\node[above left] at (1/2,1/2){$(0,0,1)$};
\node[ left]at(1/4,1/4){$(1/2,0, 1/2)$};
\end{tikzpicture}
\end{center}

There is also the fast-multiplicative Gauss version, where we set
$$\triangle_k^G(\sigma, \tau_0, \tau_1)  = V F_1^k(\sigma, \tau_0, \tau_1) F_0(\sigma, \tau_0, \tau_1)  \cap \{(x,y,z) : x+y+z=1\}.$$
For the $(e,e,e)$ case, again when $V=I$, we have 
$$ F_1^k(e,e,e) F_0(e,e,e) =  \begin{bmatrix} 0 & k & k+1 \\ 1 & 0 & 0 \\ 0 & 1 & 1 \end{bmatrix}, $$
and $\triangle_k^G(e,e,e) $ will be the triangle with vertices
\begin{equation*}
    (0, \; 1, \; 0), \qquad \parens{\frac{k}{k+1}, \; 0, \; \frac{1}{k+1}}, \qquad \parens{\frac{k+1}{k+2}, \; 0, \; \frac{1}{k+2}}.
\end{equation*}

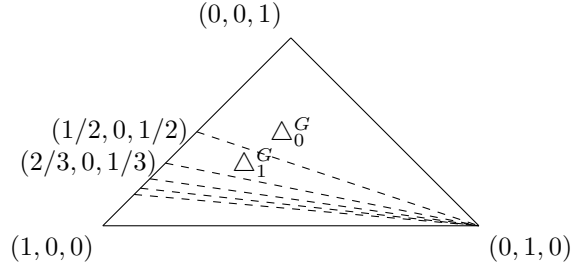
\begin{center}
\begin{tikzpicture}[scale=5]
\draw(0,0)--(1,0);
\draw(0,0)--(1/2,1/2);
\draw(1,0)--(1/2,1/2);
\draw[dashed] (1/4,1/4)--(1,0);
\draw[dashed] (1/6,1/6)--(1,0);
\draw[dashed] (1/8,1/8)--(1,0);
\draw[dashed] (1/10,1/10)--(1,0);
\draw[dashed] (1/12,1/12)--(1,0);

\node[] at (1/2,1/4){$\triangle_0^G$};
\node[] at (2/5,1/6){$\triangle_1^G$};

\node[below left]at(0,0){$(1,0,0)$};
\node[below right]at(1,0){$(0,1,0)$};

\node[above left] at (1/2,1/2){$(0,0,1)$};
\node[ left]at(1/4,1/4){$(1/2,0, 1/2)$};
\node[ left]at(1/6,1/6){$(2/3,0, 1/3)$};
\end{tikzpicture}
\end{center}

The Gauss map $T^G(\sigma, \tau_0, \tau_1) $ is the map on $\triangle$ defined to map each vector $\bfx$ in the interior of a subtriangle $\triangle_k^G(\sigma, \tau_0, \tau_1) $ to a constant multiple of
$$ V(VF_1^k(\sigma, \tau_0, \tau_1) F_0(\sigma, \tau_0, \tau_1) )^{-1} \bfx, $$
where the constant is chosen so that it lands in $\triangle$.  
In particular, for $(x,y,z) \in \triangle_k^G(e,e,e)$,
$$ T^G(e,e,e)(x,y,z) = \parens{\frac{y}{y+z}, \frac{(k+1)z - x}{y+z}, \frac{x-kz}{y+z}}. $$

\subsection{$\R^2$ TRIP Maps}\label{R2TRIP Maps}

The $\R^3$ TRIP maps  have  two-dimensional analogs.
Fix two matrices
     \begin{equation}
    F_0 = \begin{bmatrix}  0 & 1 \\ 1 & 1 \end{bmatrix}, \qquad F_1 = \begin{bmatrix} 1 & 1  \\ 0  & 1  \end{bmatrix},
\end{equation}
and let $C$ be the cone in $\R^2$ spanned by the two linearly independent vectors 
 \begin{equation}
    v_0 = \begin{bmatrix}  0  \\  1  \end{bmatrix}, \qquad v_1 = \begin{bmatrix} 1 \\ 1  \end{bmatrix}.
\end{equation}
Write the two permutations  $\sigma \in S_2$ as 
  \begin{equation}
    e = \begin{bmatrix}  1 & 0 \\  0 & 1 \end{bmatrix}, \qquad (12) = \begin{bmatrix} 0 & 1  \\ 1  & 0  \end{bmatrix}.
\end{equation}
For $\sigma, \tau_0, \tau_1\in S_2$, set
\begin{eqnarray*}
    F_0(\sigma, \tau_0, \tau_1) &=&  \sigma F_0 \tau_0, \\
    F_1(\sigma, \tau_0, \tau_1) &=&  \sigma F_1 \tau_0, \\
    C_0 (\sigma, \tau_0, \tau_1) &=& V F_0(\sigma, \tau_0, \tau_1), \\
    C_1 (\sigma, \tau_0, \tau_1) &=&V F_1(\sigma, \tau_0, \tau_1).
\end{eqnarray*}
Then, the $\R^2$ TRIP map 
$$T(\sigma, \tau_0, \tau_1): C - ( C_0 (\sigma, \tau_0, \tau_1) \cap  C_1 (\sigma, \tau_0, \tau_1) ) \rightarrow C$$
is
\begin{equation*}
    T(\sigma, \tau_0, \tau_1) \bfx = \begin{cases} T_0(\sigma, \tau_0, \tau_1) \bfx, \qquad \bfx \in  \Delta_0(\sigma, \tau_0, \tau_1)  \\
    T_1(\sigma, \tau_0, \tau_1) \bfx, \qquad \bfx \in  \Delta_1(\sigma, \tau_0, \tau_1), \end{cases}  
\end{equation*}
where
\begin{equation*}
    T_0(\sigma, \tau_0, \tau_1) = V(F_0(\sigma, \tau_0, \tau_1))^{-1}V^{-1}, \qquad T_1(\sigma, \tau_0, \tau_1) = V(F_1(\sigma, \tau_0, \tau_1))^{-1}V^{-1}.
\end{equation*}
This yields a family of $8$ maps. The map $T(e,e,e)$ is the traditional Farey map, meaning that the fast version is the traditional Gauss map. The map $T(e,12,e)$ corresponds to the backward continued fraction algorithm (these are also called Hirzebruch-Jung continued fractions),  which is when the plus signs in the tradiational continued fraction expansion are replaced by negative signs.  The other six maps are new.

\section{$S$-adic Words and Complexity} \label{$S$-adic sequences and Complexity}


We consider sequences or \keyterm{infinite words} $ \bfu = u_0 u_1 u_2 \hdots $, where each $u_i$ is from some \keyterm{alphabet} $\A$. 
A \keyterm{finite word} over $\A$ is a sequence of finitely many characters $ w = w_0 w_1 \hdots w_{n-1} $ for some positive interger $ n $. Such $n$ is called the \keyterm{length} of $w$, denoted $ \abs{w} $. We denote the empty word, or the word of length $0$, by $\emptyword$.
We denote the set of all finite words over $\A$ by $ \A^* $ and the set of all infinite words by $ \A^\N $.
For finite words $ v = v_0 v_1 \cdots v_{n-1} $ and $ w = w_0 w_1 \cdots w_{m-1} $ we define the \keyterm{concatenation} of $v$ and $w$ to be $ vw := v_0 v_1 \cdots v_{n-1} w_0 w_1 \cdots w_{m-1} $.
Concatenation of finite words is associative but not commutative.

The following definitions are adopted from Berth\'{e} and Labb\'{e} \cite{ARP}, Berth\'{e} and Rigo  \cite{Berthe-Rigo}, and Fogg \cite{Fogg}:

\begin{definition}
    A \keyterm{factor} or \keyterm{subword} of an infinite word $\bfu$ is a finite word $w$ that appears exactly as some sequence of consecutive characters in $\bfu$. Specifically, this means that if $ \bfu = a_0 a_1 a_2 \cdots $ and $ \abs{w} = n $, then $w$ is a factor if for some non-negative integer $ m $,  we have $ w = a_m a_{m+1} \cdots a_{m+n-1} $.
    Factors of a finite word are defined similarly.
\end{definition}

\begin{definition} \label{def:language}
    A \keyterm{language} is a subset $ \mathcal{L} \subseteq \A^* $ that is factorially closed, that is, if $ w \in \mathcal{L} $ and $v$ is a factor of $w$, then $ v \in \mathcal{L} $.
    We say that a language $\mathcal{L} $ is \keyterm{left extendable} (respectively \keyterm{right extendable}) if for all $ w \in \mathcal{L} $, there is some letter  $ a \in \A $ such that $ aw \in \mathcal{L} $ (respectively $ wa \in \mathcal{L} $).
    If $\mathcal{L}$ is both left extendable and right extendable, then we say $\mathcal{L}$ is \keyterm{extendable}.
\end{definition}

Note that our definition of a language and extendability differs slightly from that in Berth\'{e} and Rigo \cite{Berthe-Rigo}.

\begin{definition} \label{def:comp-language}
    The \keyterm{factor complexity} of a language $\mathcal{L}$ is the function $p_\mathcal{L}$ that maps each positive integer $n$ to the number of distinct words of length $n$ contained in $\mathcal{L}$.
\end{definition}
While there many different types of complexity, overwhelmingly in this paper we when we say ``complexity,'' we will mean factor complexity.

\begin{definition}
    A \keyterm{prefix} of a finite or infinite word $w$ is a word $a$ that appears at the beginning of $w$, that is, $ w = av $ for some finite or infinite word $v$.
    Similarly, a \keyterm{suffix} of a finite word $w$ is a word $b$ that appears at the end of $w$, that is, $ w = vb $ for some finite word $v$.
    We call a prefix or suffix of a word $w$ that is not equal to $w$ a \keyterm{proper prefix} or \keyterm{proper suffix}, respectively.
\end{definition}

\begin{definition}
    An infinite word $\bfu$ is \keyterm{recurrent} if all of its factors occur infinitely often in $\bfu$.
    If for any factor $w$ of $\bfu$ there exists $N$ such that every subword of length $N$ in $\bfu$ contains $w$ as a factor, then $\bfu$ is \keyterm{uniformly recurrent}. 
\end{definition}

\begin{definition}
    The \keyterm{language} of an infinite word $\bfu$, denoted $\mathcal{L}(\bfu)$, is the set of all factors of $\bfu$.
    The \keyterm{complexity function} of $\bfu$, denoted $p_\bfu$, $p$, $p_{\mathcal{L}(\bfu)}$, or even $p_{\mathcal{L}}$,  is simply the complexity of $\mathcal{L}(\bfu)$, so that $p_{\bfu}(n)$ is the number of distinct factors of $\bfu$ of length $n$.
\end{definition}

For example, for the word $\bfu = 123123123\ldots$, the set of factors of length $3$ is $\set{123,231,312}$, so $ p_\bfu(3) = 3 $.

A \keyterm{substitution} is a function $ \sigma: \A \rightarrow \A^* $. A substitution can be extended to an operation on finite words through the operation of concatenation,
\begin{equation}
    \sigma(u_0 u_1 u_2 \hdots u_{n-1}) = \sigma(u_0) \sigma(u_1) \sigma(u_2) \hdots\sigma(u_{n-1})
\end{equation}
and  to infinite words
\begin{equation}
    \sigma(u_0 u_1 u_2 \hdots) = \sigma(u_0) \sigma(u_1) \sigma(u_2) \hdots,
\end{equation}

We give the following definition, as in Definition 1.4.1 of \cite{Thuswaldner}:
\begin{definition} \label{defSadic}
    Let $S$ be a finite set of substitutions on $\A$. An infinite word $ \bfu \in \A^\N $ is \keyterm{$S$-adic} if there is a sequence $ \set{\udesub{m}}_{m=0}^\infty $ of infinite words and a sequence $ \set{\sigma_m}_{m=0}^\infty \subseteq S $ of substitutions such that $ \udesub{0} = \bfu $ and for all $m$, $ \udesub{m} = \sigma_m(\udesub{m+1}) $.
    In this case, $\set{\sigma_m}$ is the \keyterm{coding sequence} of $\bfu$.
\end{definition}

The classical case of $S$-adic words are the Sturmian words, which are infinite words on two letters with complexity $n+1$.  Let $ \mathcal{A} = \set{1,2} $ and $ S = \set{S_0,S_1} $ with:
\begin{itemize}
    \item $ S_0: \quad 1 \mapsto 1, \quad 2 \mapsto 21 $,
    \item $ S_1: \quad 1 \mapsto 12, \quad 2 \mapsto 2 $.
\end{itemize}
It is known that an infinite word on two letters  is Sturmian if and only if it has the above coding sequence containing $S_0$ and $S_1$ infinitely often. See Chapter 2 in Lothaire \cite{Lothaire}, Chapter 6 in  Fogg \cite{Fogg} or Chapter 9 in Allouche and Shallit \cite{Allouche-Shallit}.

The $S$-adic conjecture asks to what extent this can be generalized to a relationship between $S$-adic words and words of at most linear complexity.
In one direction, it was proven by Ferenczi \cite{Ferenczi} that any uniformly recurrent word $\bfu$ with linear complexity has the same language as an $S$-adic word. Leroy \cite{Leroy1} strengthened this result and showed that $\bfu$ itself is $S$-adic.
However, there exists $S$-adic words with only two different substitutions that have superlinear complexity (see, for example, \cite{Fogg}). (For recent work on the converse, see Espinoza \cite{Espinoza}.)

As mentioned in the introduction, recent work has explored classes of $S$-adic words on three letters with low complexity. For example, classes of such infinite words with upper bounds $2n+1$ were found by Arnoux and Rauzy \cite{AR} and by Cassaigne, Labb\'{e}, and Leroy \cite{Cassaigne}, and a class of such infinite words $\bfu$ with $ 2n+1 \leq p_{\mathcal{L}}(n) \leq \frac{5}{2}n $ was found by Berth\'{e} and  Labb\'{e} \cite{ARP}.

\section{Correspondence between Multi-dimensional Continued Fractions and $S$-adic sequences} \label{Correspondence}


Connections between three-dimensional continued fraction algorithms and classes of $S$-adic sequences have been explored by many others in recent years. For an excellent general background, see Fogg \cite{Fogg}, Berth\'{e} and Delecroix \cite{Berthe-Delecroix} and Thuswaldner's recent survey  \cite{Thuswaldner}. We will be following, to some extent, chapter 3 of \cite{Thuswaldner}. In this section, we present a way to associate a TRIP map $T(\sigma,\tau_0,\tau_1)$ with a set $ S(\sigma,\tau_0,\tau_1) = \set{S_0(\sigma,\tau_0,\tau_1), S_1(\sigma,\tau_0,\tau_1)} $ of two substitutions.

For the rest of the paper, we restrict our alphabet to $ \A = \set{1,2,3} $.

\subsection{Abelianizations} \label{sec:abelianization}

We start with seeing how any substitution can be associated to a matrix via abelianization. While all of this works for any finite alphabet, we will restrict attention to the three letter alphabet $ \A = \set{1,2,3} $ for ease of notation.

\begin{definition} \label{def:abel-word}
    The \keyterm{abelianization} of a finite word $w$ is defined as
    \begin{equation}
        \bfl(w) = \begin{bmatrix} \abs{w}_1 \\ \abs{w}_2 \\ \abs{w}_3 \end{bmatrix},
    \end{equation}
    and the \keyterm{frequency vector} is $ \frac{1}{\abs{w}} \bfl(w) $, where $\abs{w}_i$ is the number of times the character $i$ appears in $w$.
    We define the frequency vector of an infinite word $\bfu$ to be $ \bff(\bfu) = \parens{f_1(\bfu), f_2(\bfu), f_3(\bfu)} $ with
    \begin{equation}
        f_i(\bfu) = \lim_{n \rightarrow \infty} \frac{\abs{u_0 u_1 u_2 \hdots u_{n-1}}_i}{n},
    \end{equation}
    provided that this limit exists.
\end{definition}

\begin{definition}
    The abelianization of a substitution $\sigma$ is a $ 3 \times 3 $ matrix $\ell(\sigma)$ defined by
    \begin{equation}
        \ell(\sigma) = \set{a_{ij}}, \qquad a_{ij} = \abs{\sigma(j)}_i, 
    \end{equation}
   where $ i,j \in \{1,2,3\}.$
\end{definition}

For example, the abelianization of the substitution
$$\sigma: \quad 1\rightarrow 2, \quad 2\rightarrow 3, \quad 3 \rightarrow 13$$
is the matrix
$$ \ell(\sigma) = \begin{bmatrix} 0&0&1 \\ 1&0 & 0 \\ 0&1& 1 \end{bmatrix} .$$
Note that two distinct substitutions can have the same abelianization, as seen for the substitution
$$ \widehat{\sigma}: \quad 1 \mapsto 2, \quad 2 \mapsto 3, \quad 3 \mapsto 31, $$
we have $ \ell(\sigma) = \ell(\widehat{\sigma}) $.
An important property is that under abelianization, the action of a substitution becomes matrix multiplication,
\begin{equation} \label{abelprod}
    \ell(\sigma(w)) = \ell(\sigma) \ell(w).
\end{equation}
Note that this means $ \ell(\sigma^2) =  \ell(\sigma)^2.$

While two different substitutions can have the same matrix, we can always go from a $3\times 3$ matrix with non-negative entries to a substitution. For example, given any matrix 
$$M= \begin{bmatrix} a&b&c \\ d&e & f \\ g&h& k \end{bmatrix}, $$
the substitution
$$ \sigma_M: \quad 1 \mapsto 1^a2^d 3^g, \quad 2 \mapsto 1^b2^e3^h, \quad 3 \mapsto 1^c2^f3^k $$
will have abelianization $M$.

\subsection{Substitutions and TRIP Maps} \label{sec:sub-cont-frac}

We aim to define substitutions $S_0(\sigma,\tau_0,\tau_1)$ and $S_1(\sigma,\tau_0,\tau_1)$ whose abelianizations are $F_0(\sigma,\tau_0,\tau_1)$ and $F_1(\sigma,\tau_0,\tau_1)$, respectively.
To do so, we first define the substitutions of the TRIP map $T(e,e,e)$ as
\begin{itemize}
    \item $ S_0(e,e,e): \quad 1 \mapsto 2, \quad 2 \mapsto 3, \quad 3 \mapsto 13 $,
    \item $ S_1(e,e,e): \quad 1 \mapsto 1, \quad 2 \mapsto 2, \quad 3 \mapsto 13 $.
\end{itemize}
Then, for each of the other 215 TRIP maps, we define
\begin{align*}
    S_i (\sigma, \tau_0, \tau_1) = \sigma \circ S_i(e,e,e ) \circ \tau_i
\end{align*}
for $ i \in \set{0,1} $. Here we consider the permutations $\sigma$ and $\tau_i$ on $\A$ as substitutions which map each letter to a word of length $1$.
We define the $ S(\sigma, \tau_0, \tau_1) $-adic sequences using Definition \ref{defSadic} with $ S = \set{S_0(\sigma, \tau_0, \tau_1), S_1(\sigma, \tau_0, \tau_1)} $.
When the choice of the permutations is clear from context, we will often refer to the substitutions $S_i(\sigma,\tau_0,\tau_1)$ simply by $S_i$.

We will routinely refer to the substitutions $ S_k^G := S_1^k \circ S_0 $ as ``Gauss substitutions.''
Any coding sequence over $\set{S_0,S_1}$ containing $S_0$ infinitely often can be expressed as a coding sequence over $ \set{S_k^G}_{k=0}^\infty $ simply by replacing each group of $ k \geq 0 $ occurrences of $S_1$ followed by an $S_0$ by $S_k^G$.

We can associate points in $\triangle$ with $S(\sigma,\tau_0,\tau_1)$-adic sequences $\bfu$ according to their frequency ratios, and the Farey sequences of such points to the coding sequence of $\bfu$. The following proposition (which is a special case of Theorem 5.7 in Berth\'{e} and Vincent Delecroix \cite{Berthe-Delecroix}, which provides a proof)
 determines when this is possible:

\begin{proposition} \label{freqvector}
    For any $ \sigma, \tau_0, \tau_1 \in S_3 $, let $F_0$ and $F_1$ be the corresponding Farey matrices, and $S_0$ and $S_1$ the corresponding substitutions. Let $ (i_0, i_1, i_2, \cdots) $ be the Farey sequence of some point $ (x,y,z) \in \triangle $. If
    \begin{equation} \label{infintconverge}
        \bigcap_{m=0}^\infty \triangle_{i_0,i_1,i_2,\cdots,i_{m-1}}(\sigma,\tau_0,\tau_1) = \set{(x,y,z)},
    \end{equation}
    then any $S(\sigma,\tau_0,\tau_1)$-adic sequence $\bfu$ with coding sequence $ (S_{i_0}, S_{i_1}, S_{i_2}, \cdots) $ will have frequency vector $(x,y,z)$.
\end{proposition}

We give the following definition of primitivity from page 29 of Thuswaldner \cite{Thuswaldner}:

\begin{definition}
    An $S$-adic word $\bfu$ with coding sequence $\set{\sigma_m}_{m=0}^\infty$ is \keyterm{primitive} if for all $ m \in \NatZero $, there exists $ n > m $ such that $ M_m M_{m+1} M_{m+2} \cdots M_{n-1} $ has all entries strictly positive, where each $M_i$ is the abelianization matrix of the substitution $\sigma_i$.
\end{definition}

Lemma 1.5.10 of \cite{Thuswaldner} implies, in particular, that condition (\ref{infintconverge}) is true for primitive $S$-adic words that are also recurrent. 

Our ultimate goal is to find bounds on the complexity of $S$-adic sequences associated with various TRIP maps. Specifically, for any $ \sigma, \tau_0, \tau_1 \in S_3 $, we would like to find an upper bound for $p$ for any $F(\sigma,\tau_0,\tau_1)$-adic sequences. We are particularly interested in cases with low upper bounds such as $2n+1$ or $3n$.

\begin{remark}
    Our definition of $S_i$ substitutions can differ from what one would get by taking the $F_i(\sigma, \tau_0, \tau_1)$ matrix as defined in \cite{Garrity-Mccdonald} and defining an associated substitution to this matrix as was done in the previous subsection. For example, we have
    $$ S_0(13,e,e): \quad 1 \mapsto 2, \quad 2 \mapsto 1, \quad 3 \mapsto 31, $$
    but the substitution we would derive from the matrix $F_0(13,e,e)$ in the previous subsection is
    $$ 1 \mapsto 2, \quad 2 \mapsto 1, \quad 3 \mapsto 13. $$
    The definitions given above allow for equivalence relations between various $S$-adic sequences that correspond to certain equivalence relations between various TRIP maps which lead to the same complexities, as we will see in Section \ref{More general TRIP maps}.
    There is, however, one source of freedom that we do have in choosing the substitutions from the Farey matrices, which we will explain in Section \ref{Statements on  Complexity for all TRIP maps}.
\end{remark}

\subsection{Removing the Requirement of Primitivity}

In most cases throughout the literature, complexity corresponding to multidimensional continued fraction maps is viewed through the lens of infinite $S$-adic words. However, it is possible to define a language directly from a coding sequence without a particular $S$-adic word using the following definition from \cite{Thuswaldner}:

\begin{definition} \label{def:lang-coding-seq}
    The \keyterm{language associated with} a coding sequence $ \mathbf{\sigma} = (\sigma_0, \sigma_1, \sigma_2, \hdots) $ is the smallest possible language (by inclusion) containing $ (\sigma_0 \circ \sigma_1 \circ \hdots \circ \sigma_{n-1})(a) $ for all $ n \in \N $ and $ a \in \mathcal{A} $.
    Equivalently,
    \begin{equation}
        \lang(\sigma) = \set{w: w \text{ is a factor of } (\sigma_0 \circ \sigma_1 \circ \hdots \circ \sigma_{n-1})(c) \text{ for some } n \in \N, \; c \in \A}. 
    \end{equation}
\end{definition}

If the coding sequence $ \sigma = \set{\sigma_n}_{n=0}^\infty $ is primitive, as the vast majority of cases occurring in the literature are, then it makes no difference whether we consider the complexity of $S$-adic words with coding sequence $\sigma$ or the complexity of the language $L(\sigma)$ as defined in Definition \ref{def:lang-coding-seq}. 
Indeed, assuming primitivity, all $S$-adic words $\bfu$ with coding sequence $\sigma$ satisfy $ \lang(\bfu) = \lang(\sigma) $, and such an $S$-adic word always exists as one can be found by taking a limit point of the sequence
$$ u_n := \parens{\sigma_0 \circ \sigma_1 \circ \hdots \circ \sigma_{n-1}}(a) $$
for any $ a \in \A $.

However, many of the $S$-adic languages that we consider in this paper are not primitive, so the distinction between these two perspectives is substantial.
For example, consider the substitutions associated to the $T(e,13,e)$ TRIP map,
\begin{itemize}
    \item $ S_0: \quad 1 \mapsto 13, \quad 2 \mapsto 3, \quad 3 \mapsto 2 $,
    \item $ S_1: \quad 1 \mapsto 1, \quad 2 \mapsto 2, \quad 3 \mapsto 13 $.
\end{itemize}
The $S$-adic word $ \bfu = 1323232323\hdots $ with coding sequence $ (S_0, S_0, S_0, \hdots) $ has complexity $n+2$, but binary sequence of $2$s and $3$s is also $S$-adic with the same coding sequence. Such words have complexity up to $2^n$, but they do not properly reflect the substitutions as each character is its own ``$S$-adic component,'' meaning it can be de-substituted an arbitrarily large number of times independently of the other characters.

We believe a reasonable condition to impose on $S$-adic words is that for all $ n \in \N $, there is a single character $ c \in \mathcal{A} $ and $ m \in \N $ such that $ u_0 u_1 u_2 \hdots u_{n-1} $ is a prefix of $ \parens{\sigma_0 \circ \sigma_1 \circ \hdots \circ \sigma_{m-1}}(c) $. We call such sequences \keyterm{singly-generated}.
It can be shown that the only singly-generated $S$-adic sequence with this coding sequence is $ 1323232323\hdots $, so in a sense the true upper bound is $n+2$.

Now consider the substitutions associated with the TRIP map $T(13,12,132)$,
\begin{itemize}
    \item $ S_0: \quad 1 \mapsto 1, \quad 2 \mapsto 2, \quad 3 \mapsto 31 $,
    \item $ S_1: \quad 1 \mapsto 31, \quad 2 \mapsto 3, \quad 3 \mapsto 2 $.
\end{itemize}
We would expect the complexity properties of $S(13,12,132)$-adic words with coding sequence $ (S_1, S_1, S_1, \hdots) $ to be the same as those of $S(e,13,e)$-adic words with coding sequence $ (S_0, S_0, S_0, \hdots) $, as the corresponding substitutions are exactly the same with the order of characters reversed (see Section \ref{More general TRIP maps} for a more detailed discussion). However, an infinite singly-generated $S$-adic word with this coding sequence does not exist, as can be seen in 
$$1 \xmapsto{S_1} 31 \xmapsto{S_1} 231 \xmapsto{S_1} 3231 \xmapsto{S_1}\cdots,$$
giving us that the first term quickly starts to oscillate between $3$ and $2$.

If we consider instead the complexity of the languages of coding sequences as defined in Definition \ref{def:lang-coding-seq} as opposed to the complexity of $S$-adic sequences, we have $n+2$ for both cases and these issues are resolved.
For this reason, we believe considering complexity from the perspective of languages defined directly from coding sequences is more natural for our work.
Our complexity bounds immediately imply the same upper bounds on complexity for singly-generated $S$-adic words via the following:

\begin{proposition}\label{prop on lang}
    For any singly-generated $S$-adic word $\bfu$ with coding sequence $\sigma$, $ \lang(\bfu) \subseteq \lang(\sigma) $, with equality if $\sigma$ is primitive.
\end{proposition}
Certainly $ \lang(\bfu) \subseteq \lang(\sigma) $ is always the case.  The assumption of primitivity of $\sigma$, which means that eventually all the entries in the corresponding matrices are positive, will imply that no matter what is our initial ``seed'' for  $\bfu$, all terms will eventually appear.


\section{Proof Techniques for Complexity of $S$-Adic Sequences} \label{sec:pftech}

Our goal is to find  bounds for the complexity functions for various $S$-adic sequences. For example, in Theorem \ref{Main theorem for $eee$}, we will see that an $S$-adic sequence corresponding to the TRIP map $T(e,e,e)$ (also called the triangle map) will satisfy the bound $ p(n) \leq 3n $.
In order to find such bounds, we will need to relate the factors of length $n+1$ or $n+2$ to factors of length $n$.

The tools described in this section are the primary methods of determining complexities of $S$-adic words (see Cassaigne and Nicolas's  ``Factor Complexity''  in \cite{Berthe-Rigo} and also \cite{AR,ARP,Berthe-Rigo,Leroy}). However, we reformulate them to describe the complexity of languages in general rather than the language of a specific infinite word.

\begin{definition}
    For $ w \in \mathcal{L} $, we define the \keyterm{left extension set} of $w$ to be
    $$ E^-(w) := \set{a \in \set{1,2,3}: aw \in \mathcal{L}}, $$
    and the \keyterm{right extension set} of $w$ to be
    $$ E^+(w) := \set{b \in \set{1,2,3}: wb \in \mathcal{L}}. $$
    Finally, we set
    $$ E(w) := \set{(a,b) \in \set{1,2,3} \times \set{1,2,3}: awb \in \lang(\bfu)}. $$
    We say $w$ is \keyterm{left special} if $ |E^-(w)| \geq 2 $, \keyterm{right special} if $ |E^+(w)| \geq 2 $, and \keyterm{bispecial} if it is both left special and right special.
\end{definition}

Theorem 4.5.4 in Cassaigne and Nicolas's ``Factor Complexity'' in  \cite{Berthe-Rigo} states

\begin{theorem} \label{prop:l-r-complex}
   Letting  $\lang_n$ denotes the set of words of length $n$ in any language $\lang$, we have 
    \begin{enumerate}
\item $p_\lang(n+1) - p_\lang(n) = \sum_{w \in \lang_n} \parens{ |E^-(w)| - 1} $
\item $ p_\lang(n+1) - p_\lang(n)= \sum_{w \in \lang_n} \parens{ |E^+(w)| - 1}$ 
\item $p_\lang(n+2) - 2p_\lang(n+1) + p_\lang(n) = \sum_{w \in \lang_n} m(w),$
where 
$$m(w) := |E(w)| - |E^-(w)| - |E^+(w)| + 1.$$
\end{enumerate}

\end{theorem}

(The first two formulas  follows directly from the fact that there is a bijection from $\lang_{n+1}$ to pairs $(a,w)$ with $ w \in \lang_n $ and $ a \in E^-(w) $, and similarly for $E^+(w)$.
We may restrict the summations in the above theorem to be only over the left (respectively right) special factors and the factors which have no left (respectively right) extensions since all other terms are zero. The third follows from the first two  and the fact that there is a bijection from $\lang_{n+2}$ to triples $(a,w,b)$ with $ w \in \lang_n $ and $ (a,b) \in E(w) $.)

We may then compute complexity by the following formulas:

\begin{corollary} \label{cor:bi-complex}
    For any language $\lang$,
    \begin{equation} \label{bi-complex}
        p_\lang(n+1) - p_\lang(n) = \abs{\lang_1} - 1 + \sum_{\substack{w \in \lang \\ \abs{w} \leq n-1}} m(w).
    \end{equation}
    In particular,
    \begin{equation} \label{bi-complex-2}
        p_\lang(n+1) - p_\lang(n) \leq \abs{\lang_1} - 1 + \sum_{\substack{w \in \lang \\ \abs{w} \leq n-1 \\ w \text{ bispecial}}} m(w),
    \end{equation}
    where equality holds if $\lang$ is extendable.
\end{corollary}

Adopting the convention of Cassaigne and Nicolas from \cite{Berthe-Rigo}, we call a bispecial factor $w$ \keyterm{neutral} if $ m(w) = 0 $.

\section{Factor Complexity for $(e,e,e)$ TRIP map: The Triangle Map} \label{sec:e-e-e}

We are now done with the preliminaries. In this section, we prove the main technical result of this paper:

\begin{theorem}\label{Main theorem for $eee$}
    The complexity function of any $S(e,e,e)$-adic language $\mathcal{L}$ with rationally independent frequency vectors satisfies $ 2n+1 \leq p_{\mathcal{L}}(n) \leq 3n $ for all $ n \geq 1 $.
\end{theorem}

The $(e,e,e)$ TRIP map has independent interest.  It is usually called the triangle map and was originally created in \cite{Triangle, GarrityT05} for number theoretic reasons, namely in an attempt to solve the Hermite Problem. Messaoudi,  Nogueira, and  Schweiger, in the beginning of  \cite{SchweigerF08} 
where they showed that the fast map is ergodic, discuss why this map is interesting dynamically.  Copying from \cite{Baalbaki-Garrity}, 
we know that further dynamical properties were discovered by  Berth\'e, Steiner and Thuswaldner \cite{Berthe-Steiner-Thuswaldner}  and by Fougeron and  Skripchenko \cite{Fougeron-Skripchenko}.   Bonanno, Del Vigna and  Munday \cite{Bonanno-Del Vigna-Munday} and Bonanno and Del Vigna \cite{Bonanno-Del Vigna} recently used the $\R^3$ slow  triangle map to develop a tree structure of rational pairs in the plane.   In a recent preprint, Ito \cite{Ito} showed that the fast  map is self-dual (in section three of that paper).  

The upper bound is what is new.  The lower bound stems from quite general principles, as shown by R. Tijdeman \cite{Tijdeman}.  Also see the more recent work of Andrieu and Vivion \cite{Andrieu-Vivion}

Finally, the $(e,e,e)$ map (the triangle map) and its higher dimensional analog seems to be the only multidimensional continued fraction algorithms that can be  both used to study partition numbers and also that respect Young conjugation for partitions, as is discussed in the recent papers \cite{BCDGI, Baalbaki-Garrity}

\subsection{Preliminaries on $S$-adic Languages for the Triangle map} \label{sec:e-e-e1}

We begin by noting the following:

\begin{lemma}
    The Farey sequence associated to any point in $\triangle$ with rationally independent entries in the TRIP map $T(e,e,e)$ has infinitely many $0$s.
\end{lemma}

\begin{proof}
    The proof is relatively straightforward.  As discussed in Section \ref{TRIP Maps}, we start with the cone spanned by the basis vectors $e_1,e_2,e_3$.  As we apply various $F_i(e,e,e)$, we are shrinking the initial cone, but the edges are images of the $e_1,e_2,e_3$ by the product of the $F_i(e,e,e)$.  As each of these $F_i(e,e,e)$ are in the special linear group, the cone will be spanned by vectors with integer entries.  This means that after a finite number of applications of the  $F_i(e,e,e)$, we have a cone $\triangle = (v_1,v_2, v_3)$, where each $v_i$ has integer entries. Now let us suppose we simply keep applying the map $F_1(e,e,e)$ over and over again.  Then we get
    $F_1^k(e,e,e) (v_1, v_2, v_3) = (v_1, v_2, kv_1+v_3).$  The only points that will be in  $F_1^k(e,e,e) (v_1, v_2, v_3) $ for all $k$ are those points in the plane spanned by $v_1$ and $v_2.$  But any point in this plane will have entries that are rationally related, as desired.
\end{proof}

This allows us to transform the coding sequence $ \set{S_{i_m}}_{m=0}^\infty $ into a Gauss coding sequence $ \set{G_{k_m}}_{m=0}^\infty $, where
\begin{equation} \label{eee-Gauss}
    G_k := S_1^k \circ S_0: \quad 1 \mapsto 2, \quad 2 \mapsto 1^k 3, \quad 3 \mapsto 1^{k+1}3.
\end{equation}
(The adjective ``Gauss'' is being used in analog of the difference between the traditional Farey map and the traditional Gauss map for continued fractions; hence we are using Gauss to refer to concatenating sequences of the $F_1$ map.)

We have the following proposition:
\begin{proposition}

    Any Gauss coding sequence is primitive.
\end{proposition}

\begin{proof}The matrix whose abelianization gives us $G_k$ is 
$$\left(
\begin{array}{ccc}
 0 & k & k+1 \\
 1 & 0 & 0 \\
 0 & 1 & 1 \\
\end{array}
\right)$$
Composing three Gauss maps, say $G_{k_0}, G_{k_1}$ and $G_{k_2}$ will yield the product of the three matrices
$$\left(\begin{array}{ccc}
 0 & k_0 & k_0+1 \\
 1 & 0 & 0 \\
 0 & 1 & 1 \\
\end{array}
\right)     \left(
\begin{array}{ccc}
 0 & k_1 & k_1+1 \\
 1 & 0 & 0 \\
 0 & 1 & 1 \\
\end{array}
\right) \left(\begin{array}{ccc}
 0 & k_2 & k_2+1 \\
 1 & 0 & 0 \\
 0 & 1 & 1 \\
\end{array} 
\right)$$
which is 
$$\left(
\begin{array}{ccc}
 k_0+1 &k_0 k_2+ k_0 +1 & k_0 (k_2+1)+k_0+1 \\
k_1 & k_1+1 & k_1+1 \\
 1 & k_2+1 & k_2+2 \\
\end{array}
\right).$$
As this matrix has all non-zero entries, we are done.  
\end{proof}
From Proposition \ref{prop on lang} and the above proposition, we have
\begin{corollary}
    Any $S$-adic word $\bfu$ with coding sequence $\sigma$ and rationally independent frequency vectors has $L(\bfu) = L(\sigma)$.
\end{corollary}
Thus, for the rest of this section, we will compute the complexity of such a sequence $\sigma$ by computing the complexity of a corresponding $S$-adic word $\bfu$.

We define a sequence of de-substituted languages via preimages, $ \set{\lang^{(n)}}_{n=0}^\infty $, by $ \lang^{(0)} = \lang $ and $ \lang^{(n+1)} = G_{k_n}^{-1}\parens{\lang^{(n)}} $. Equivalently, $\lang^{(n)}$ is the language of the coding sequence $ \set{G_{k_n}, G_{k_{n+1}}, G_{k_{nS+2}}, \hdots} $.

To control the linear complexity of $\bfu$, as we saw in Section \ref{sec:pftech}, we need to understand the non-neutral bispecial factors of $\bfu$. These are the finite subwords $w$ with both  the left extension set $E^L(w)$
and the right extension set $E^R(w)$ having at least two elements and with
$0\neq m(w) = |E(w)|- |E^L(w)| - |E^R(w)| + 1.$

To prove Theorem \ref{Main theorem for $eee$}, we show that for all $n$, the summation in (\ref{bi-complex-2}) is either $0$ or $1$. To do so, it suffices to find an explicit sequence of bispecial factors
$$ \set{w_0^+, w_0^-, w_1^+, w_1^-, w_2^+, w_2^-, \hdots} $$
such that:
\begin{enumerate}[(a)]
    \item Every non-neutral bispecial element of $\lang$ is in this sequence (although it might also contain neutral bispecial factors)
    \item The lengths of the factors satisfy
    $$ |w_0^+| \leq |w_0^-| < |w_1^+| \leq | w_1^-| < |w_2^+| \leq |w_2^-| < \hdots. $$
    \item For all $n$, $ 0 \leq m(w_n^+) \leq 1 $.
    \item For all $n$, $ m(w_n^+) = -m(w_n^-) $.
\end{enumerate}

To do this, we need to consider not just the bispecial elements of $\lang$ but also the bispecial elements of $\lang^{(n)}$ for all non-negative integers $n$.
Let $\text{Bispecial}(\mathcal{L})$ denote the set of bispecial factors of $\mathcal{L}$. We will construct a map
$$ A: \{ \text{Bispecial}\parens{\lang^{(n)}} \setminus 1^j \} \rightarrow \text{Bispecial}\parens{\lang^{(n+1)}} $$
satisfying the following properties:
\begin{enumerate}
    \item  Lengths decrease, that is, for all $ w \in \{ \text{Bispecial}\parens{\lang^{(m)}} \setminus 1^j \}$,
    $$ |A(w)| \leq |w|. $$
    \item For all $ w \in 
    \{ \text{Bispecial}\parens{\lang^{(n)}} \setminus 1^j \} $, there is some $N$ such that $A^N(w)$ is either the empty word $\emptyword$ or a word of the form $1^j$. We call $N$ the \keyterm{age} of $w$.
\end{enumerate}

Note that our definition of the term \textit{age} is analogous to that in \cite{ARP}.
We will show that for all $N$, $\lang$ has at most two non-neutral bispecial words of age $N$. If there is only one such word, we label it $ w_N = w_N^+ = w_N^- $, which we will show is neutral.
If there are two such words, we label one of them $ w_N^+ $ and the other $ w_N^- $ and show that properties (c) and (d) above are satisfied.
All of this will take work, though almost all of the individual steps are fairly straightforward calculations. The difficulty lies in the amount of such calculations that are required.

\subsection{General Techniques for Relating $\lang$ to $\lang^{(m)}$ and Allowed Words of length $2$} \label{sec:eee-length2}

At the forefront of our analysis is relating elements of $\lang$ to elements of $\lang^{(1)}$. In other words, given some finite word $w$, we want to find some finite word $\hat{v}$ such that $ w \in \lang $ if and only if $ \hat{v} \in \lang^{(1)} $. Recall as discussed in subsection \ref{sec:e-e-e1} that we are given a Gauss coding sequence $G_{k_0}= S_1^{k_0}\circ S_0,  G_{k_1} = S_1^{k_1}\circ S_0, \ldots,$ that $\lang = \lang^{(0)}$ is our initial language and $\lang^{(1)} = G_{k_0}^{-1}\left(\lang^{(0)}\right) $.

It follows from Definition \ref{def:lang-coding-seq} that for $ \abs{w} \geq 2 $, $ w \in \lang $ if and only if there is some $ v \in \lang^{(1)} $ such that $w$ is a factor of $G_{k_0}(v)$.
We want to find $v$ by ``de-substituting'' $w$, which requires decomposing $w$ into substitutions of single characters.
To do so, notice that in view of the Gauss substitutions (\ref{eee-Gauss}), the last character of $G_{k_0}(c)$ for any $ c \in \set{1,2,3} $ is either a $2$ or a $3$, while all characters before the last character, if they exist, are $1$s. Thus, we simply insert a break after every $2$ or $3$ and de-substitute each block to uniquely write $ w = aG_{k_0}(v)b $ where $a$ is a proper suffix of $G_{k_0}(c)$ for some single character $c$ and $b$ is a proper prefix of $G_{k_0}(d)$ for some single character $d$.
Below are some examples:

\begin{itemize}
    \item If $ k_0 = 2 $ and $ w = 132111311131132111 $, then we may write
    $$ w = 13 \; 2 \; 1113 \; 1113 \; 113 \; 2 \; 111 = 13 G_{k_0}(13321) 111, $$
    so $ a = 13 $, $ v = 13321 $, and $ b = 111 $. The word $a$ can be a suffix of either $G_{k_0}(2)$ or $G_{k_0}(3)$, while $b$ must be a prefix of $G_{k_0}(3)$, so we know that $ w \in \lang $ implies $ 2v3 \in \lang^{(1)} $ or $ 3v3 \in \lang^{(1)} $. Conversely, since $w$ is a factor of both $G_{k_0}(2v1)$ and $G_{k_0}(3v1)$, if $ 2v3 \in \lang^{(1)} $ or $ 3v3 \in \lang^{(1)} $, then $ w \in \lang $.
    \item If $ k_0 = 1 $ and $ w = 22113131 $, then we may write
    $$ w = 2 \; 2 \; 113 \; 13 \; 1 = G_{k_0}(1132) 1, $$
    so $ a = \emptyword $, $ v = 1132 $, and $ b = 1 $. The word $b$ can be a prefix of either $G_{k_0}(2)$ or $G_{k_0}(3)$, so $ w \in \lang $ if and only if $ v2 \in \lang^{(1)} $ or $ v3 \in \lang^{(1)} $.
\end{itemize}

\begin{remark}
    Notice that this technique is specific to our particular TRIP substitution. In general, arbitrary choices of substitution will require individual analysis to determine how to, if possible, develop conditions on $\lang^{(1)}$ for when a given word is in $\lang$.
\end{remark}

As an application of these observations, we now determine the set of allowed factors of length $2$ in $\lang$ for different choices of $k_0$ and $k_1$. We know that $\emptyword$ as well as all single characters occur in any $\lang^{(m)}$, but this is certainly not true for words of length $2$, and in fact the set of length $2$ words in $\lang$ depend on the coding sequence. This is the content of the following proposition.

The two-letter words in $\lang$ are also the extension set for  $\emptyword$. The two letter words that occur are said to be \textbf{allowed} and those that do not occur are said to be \textbf{forbidden}.

\begin{proposition} \label{two words} 
    The set of allowed two letter words is determined by $k_0$ and $k_1$ as follows:
    $$\begin{array}{c|c|c} 
    k_0, k_1& \mbox{Allowed $2$-words} &  \mbox{Forbidden $2$-words}\\
    \hline \hline
    k_0=0, k_1=0 & 13,21,31,32,33 & 11,12,22,23 \\
    k_0=0, k_1>0 & 13,21,22, 32,33 & 11,12,31,23 \\
    k_0>0, k_1=0 & 11, 13,21,31,32 &12,22,23, 33 \\
    k_0>0, k_1>0 & 11, 13,21,22, 31,32 & 12,23,33 \\
    \end{array}$$
\end{proposition}

This implies that the empty word $\emptyword$ is always bispecial, but it is non-neutral only if both $ k_0 \geq 1 $ and $ k_1 \geq 1 $ with $ m(\emptyword) = 1 $.

In the rhetoric of \cite{Cassaigne}, we can also describe this result in terms of extension diagrams. These are $ 3 \times 3 $ tables where the rows represent left extensions and the columns represent right extensions. By Proposition \ref{two words}, the extension diagrams of $\eps$ are:

\bigskip
\begin{center}
 $$ \begin{array}{cc||cc}  
  k_0 = 0, k_1 = 0 &
        \begin{tabular}{c|ccc}
            $\emptyword$ & 1 & 2 & 3 \\ \hline
            1 & & & $\times$ \\
            2 & $\times$ & & \\
            3 & $\times$ & $\times$ & $\times$
        \end{tabular}&
     k_0 = 0, k_1 \geq 1 &
        \begin{tabular}{c|ccc}
            $\emptyword$ & 1 & 2 & 3 \\ \hline
            1 & & & $\times$ \\
            2 & $\times$ & $\times$ & \\
            3 & & $\times$ & $\times$
        \end{tabular} \\
     k_0 \geq 1, k_1 = 0 &
        \begin{tabular}{c|ccc}
                $\emptyword$ & 1 & 2 & 3 \\ \hline
                1 & $\times$ & & $\times$ \\
                2 & $\times$ & & \\
                3 & $\times$ & $\times$ &
        \end{tabular}&
     k_0 \geq 1, k_1 \geq 1 &
        \begin{tabular}{c|ccc}
            $\emptyword$ & 1 & 2 & 3 \\ \hline
            1 & $\times$ & & $\times$ \\
            2 & $\times$ & $\times$ & \\
            3 & $\times$ & $\times$ &
        \end{tabular}.
        \end{array}$$
        \end{center}
\bigskip

\begin{proof}[Proof of Proposition \ref{two words}]
    As with almost all of the results in this section, the proof is not deep but is instead simply keeping track of the notation and the maps.
    
    First, in all four  cases we are claiming that the words $12$ and $23$ are forbidden. In (\ref{eee-Gauss}), we see that any occurrence of a $1$ must be followed by a $1$ or by a $3$, but not a $2$, so $12$ is always forbidden. 
    Also, by the above procedure, we may uniquely de-substitute $23$ into $12$ and conclude that $ 23 \in \lang $ if and only if $ 12 \in \lang^{(1)} $, so $23$ must also be forbidden.
    To see that $13$, $21$, and $32$ are always allowed, simply notice that all are factors of $ \parens{G_{k_0} \circ G_{k_1} \circ G_{k_2}}(3) = \parens{1^{k_0}3}^{k_2+1}2^{k_1+1}1^{k_0+1}3 $.
    
    The four remaining cases, $11$, $22$, $31$, and $33$, depend on the particular values of $k_0$ and $k_1$.
    
    We start with $11$. For any finite word $v$, every occurrence of the character $1$ in $G_{k_0}(v)$ must occur in groups of either $k_0$ or $k_0+1$ $1$s, arising from a substitution of the character $2$ or $3$, respectively. Furthermore, these groups of $1$s must be either at the beginning of $G_{k_0}(v)$ or be immediately preceeded by the last character of $G_{k_0}(c)$ for some $c$, which must be either $2$ or $3$. Hence, if $ k_0 = 0 $, then $11$ cannot be a factor of $G_{k_0}(v)$ and is therefore forbidden, but if $ k_0 \geq 1 $, then $11$ is a factor of $G_{k_0}(3)$, so $11$ is allowed.

    By the above procedure, $22$ can be uniquely de-substituted into $11$, so $ 22 \in \lang $ if and only if $ 11 \in \lang^{(1)} $. Hence, $22$ is allowed if and only if $11$ is allowed in $\lang^{(1)}$, or if $ k_1 \geq 1 $.

    Now consider $33$. If $ k_0 \geq 1 $, then any occurrence of a $3$ in $G_{k_0}(v)$ for some finite word $v$ must be immediately preceded by a $1$, so $33$ is forbidden. If $ k_0 = 0 $, then since $33$ is a subword of $ G_{k_0}(32) = 133 $ and $32$ is allowed in $\lang^{(1)}$, $33$ is allowed in $\lang$.

    Finally, we consider the word $31$. By the above procedure, we have that $ 31 \in \lang $ if and only if $\lang^{(1)}$ contains $22$, $23$, $32$, or $33$. We see that $31$ is a subword of $ G_{k_0}(22) = 1^{k_0}31^{k_0}3 $ and $ G_{k_0}(32) = 1^{k_0+1}31^{k_0}3 $ if and only if $ k_0 \geq 1 $, while $31$ is always a subword of $ G_{k_0}(33) = 1^{k_0+1}31^{k_0+1}3 $ (we need not check $G_{k_0}(23)$ since $23$ is forbidden in $\lang^{(1)}$).
    Therefore, it holds that $ 31 \in \lang $ if and only if it is true that both $ v \in \lang^{(1)} $ and $31$ is a factor of $G_{k_0}(v)$ for some $ v \in \set{22,32,33} $. By the cases completed above, we find that this is true for all cases except that where $ k_0 = 0 $ and $ k_1 \geq 1 $.
\end{proof}

\subsection{Antecedents} \label{sec:eee-ante}

In this section, we formalize our previous observations relating $ w \in \lang $ to $ \hat{v} \in \lang^{(1)} $.
Recall in the previous section, for every $ w \in \lang $, we were able to identify words $\hat{v}$ for which $w$ can be de-substituted to. In this section, we establish a canonical method of doing this. 

As an illustration of the ambiguity that may occur, suppose we have $ k_0 = 0 $ and $$ w = 3 \; 3 \; 2 \; 13 = 3 G_{k_0}(213), $$
so $ a = 3 $, $ v = 213 $, and $ b = \emptyword $. The word $a$ can be a suffix of either $G_{k_0}(2)$ or $G_{k_0}(3)$, so $ w \in \lang $ if and only if $ 2v \in \lang^{(1)} $ or $ 3v \in \lang^{(1)} $.
Note that we could have chosen $ a = \emptyword $ and $ v = 2213 $ instead, but the conclusion would no longer necessarily hold. Indeed, $w$ appears as a factor of $G_{k_0}(3213)$, so $ 3213 \in \lang^{(1)} $ would imply $ w \in \lang $ even if $ 2213 \notin \lang^{(1)} $. For this reason, we always take $ a = 1^{k_0}3 $ whenever possible.
With this restriction, the choice of $ a \in \set{\emptyword,3,13,1^23,\hdots,1^{k_0}3} $, $ b \in \set{\emptyword,1,1^2,\hdots,1^{k_0+1}} $, and $v$ will always be uniquely determined by $w$. We call $v$ the \keyterm{antecedant} of $w$, analogous to the terminology in \cite{ARP}. We summarize these results below:

\begin{proposition} \label{prop:eee-ante}
    Consider any finite word $ w = a G_{k_0}(v) b $ for some $ v \in \A^* $, $ a \in \set{\emptyword,3,13,1^23,\hdots,1^{k_0}3} $, and $ b \in \set{\emptyword,1,1^2,\hdots,1^{k_0+1}} $. Assume that $w$ caontains the character $2$ or $3$ at least once and that $ a = 1^{k_0}3 $ if $1^{k_0}3$ is a prefix of $w$. Then:
    \begin{itemize}
        \item If $ a = \emptyword $ and $ b = \emptyword $, then $ w \in \lang $ if and only if $ v \in \lang^{(1)} $.
        \item If $ a = \emptyword $ and $ b \in \set{1,1^2,\hdots,1^{k_0}} $, then $ w \in \lang $ if and only if $ v2 \in \lang^{(1)} $ or $ v3 \in \lang^{(1)} $.
        \item If $ a = \emptyword $ and $ b = 1^{k_0+1} $, then $ w \in \lang $ if and only if $ v3 \in \lang^{(1)} $.
        \item If $ a \neq \emptyword $ and $ b = \emptyword $, then $ w \in \lang $ if and only if $ 2v \in \lang^{(1)} $ or $ 3v \in \lang^{(1)} $.
        \item If $ a \neq \emptyword $ and $ b \in \set{1,1^2,\hdots,1^{k_0}} $, then $ w \in \lang $ if and only if $ 2v2 \in \lang^{(1)} $, $ 2v3 \in \lang^{(1)} $, $ 3v2 \in \lang^{(1)} $, or $ 3v3 \in \lang^{(1)} $.
        \item If $ a \neq \emptyword $ and $ b = 1^{k_0+1} $, then $ w \in \lang $ if and only if $ 2v3 \in \lang^{(1)} $ or $ 3v3 \in \lang^{(1)} $.
    \end{itemize}
\end{proposition}

We make frequent use of Proposition \ref{prop:eee-ante} throughout.

\subsection{Relating Extensions to Extensions of Antecedents} \label{sec:eee-extensions}

In this section, we aim to relate the extension set of an element $ w \in \lang $ to that of its antecedent element $ v \in \lang^{(1)} $.
We are only concerned with the bispecial elements of $\lang$, so we begin with the following:

\begin{proposition} \label{prop:bispecial-a-b}
    Let $ w \in \lang $ contain at least one $2$ or $3$, and $a$, $v$, and $b$ be as in Proposition \ref{prop:eee-ante}. If $w$ is left special, then $ a \in \set{\emptyword, 1^{k_0}3} $, and if $w$ is right special, then $ b \in \set{\emptyword, 1^{k_0}} $.
\end{proposition}

\begin{proof}
    All $1$s in $G_{k_0}(v)$ must occur in groups of $k_0$ or $k_0+1$, and all such groups of $1$s must be immediately followed by a $3$. The following are immediate consequences of this:
    \begin{itemize}
        \item If $ a = 1^j 3 $ for $ 0 \leq j < k_0 $, then $ E^{-}(w) = \set{1} $.
        \item If $ b = 1^j $ for $ 1 \leq j < k_0 $, then $ E^{+}(w) = \set{1} $, since the last character in $G_{k_0}(v)$ is not a $1$.
        \item If $ b = 1^{k_0+1} $, then $ E^{+}(w) = \set{3} $.
    \end{itemize}
    In any of the above cases, $w$ cannot be bispecial. The only cases that remain are $ a \in \set{\emptyword, 1^{k_0}3} $ and $ b \in \set{\emptyword, 1^{k_0}} $.
\end{proof}

We now aim to determine $E(w)$ for a bispecial $w$ only by $a$, $b$, and $E(v)$ (here by $E(w)$ we mean the extension set in $\lang$ and by $E(v)$ we mean the extension set in $\lang^{(1)}$). To do so, for any $ a \in \set{\emptyword, 1^{k_0}3} $, we define the left extension function $ \alpha_a^L: \set{1,2,3} \rightarrow \set{1,2,3,\emptyword} $, and for any $ b \in \set{\emptyword, 1^{k_0}} $, the right extension function $ \alpha_{b,k_0}^R: \set{1,2,3} \rightarrow \set{1,2,3,\emptyword} $, depending on $k_0$, as follows:
$$ \begin{array}{c|c|c}
   i &  \alpha_a^L(i),a = \emptyword & \alpha_a^L(i), a = 1^{k_0}3 \\ \hline
    1 & 2 & \mbox{undefined} \\
    2 & 3 & 3 \\
    3 & 3 & 1
\end{array} \qquad \begin{array}{c|c|c|c}
   i & \begin{array}{c} \alpha_{b,k_0}^R(i)\\ k_0 = 0\\ b = \emptyword \end{array} &\begin{array}{c} \alpha_{b,k_0}^R(i)\\ k_0 \geq 1\\ b = \emptyword \end{array}   &  \begin{array}{c} \alpha_{b,k_0}^R(i)\\ k_0 \geq 1\\ b = 1^{k_0} \end{array}  \\ \hline
    1 & 2 & 2 & \mbox{undefined} \\
    2 & 3 & 1 & 3 \\
    3 & 1 & 1 & 1
\end{array} $$
Note that  functions $\alpha_{1^{k_0}3}$ and  $\alpha_{1^{k_0},k_0}$ (for $ k_0 \geq 1 $)  are undefined at $1$.

We interpret $\alpha_a^L$ as a map from left extensions of the antecedent factor $v$ to left extensions of $w$, and $\alpha_{b,k_0}^R$ as a map from right extensions of $v$ to right extensions of $w$. We now show that the extension set $E(w)$ is precisely the image of $E(v)$ under the map $ \alpha_a^L \times \alpha_{b,k_0}^R $:

\begin{proposition} \label{prop:extension}
    Let $ a \in \set{\emptyword, 1^{k_0}3} $, $ v \in \lang^{(1)} $, and $ b \in \set{\emptyword, 1^{k_0}} $ be such that $ w := aG_{k_0}(v)b $ contains at least one $2$ or $3$. Then,
    \begin{equation} \label{e-e-e-extension}
        E(w) = \set{\parens{\alpha_a^L(c), \alpha_{b,k_0}^R(d)}: (c,d) \in E(v), \, \alpha^L(c) \neq \emptyword, \, \alpha^R(c) \neq \emptyword}.
    \end{equation}
\end{proposition}

\begin{proof}[Proof of Proposition \ref{prop:extension}]
    We first prove the inclusion
    \begin{equation} \label{extension-pf-1}
        \set{\parens{\alpha_a^L(c), \alpha_{b,k_0}^R(d)}: (c,d) \in E(v), \, \alpha^L(c) \neq \emptyword, \, \alpha^R(c) \neq \emptyword} \subseteq E(w).
    \end{equation}
    We begin by noticing the following:
    \begin{eqnarray*}
        G_{k_0}(1v1) &=& 2 G_{k_0}(v) 2 \\
        G_{k_0}(1v2) &=& 2 G_{k_0}(v) 1^{k_0} 3 \\
        G_{k_0}(1v3) &=& 2 G_{k_0}(v) 1^{k_0+1} 3 \\
        G_{k_0}(2v1) &=& 1^{k_0} 3 G_{k_0}(v) 2 \\
        G_{k_0}(2v2) &=& 1^{k_0} 3 G_{k_0}(v)  1^{k_0} 3 \\
        G_{k_0}(2v3) &=& 1^{k_0} 3 G_{k_0}(v)  1^{k_0+1} 3 \\
        G_{k_0}(3v1) &=& 1^{k_0+1} 3 G_{k_0}(v) 2 \\
        G_{k_0}(3v2) &=& 1^{k_0+1} 3 G_{k_0}(v)  1^{k_0} 3 \\
        G_{k_0}(3v3) &=& 1^{k_0+1} 3 G_{k_0}(v)  1^{k_0+1} 3.
    \end{eqnarray*}
    It can be verified that in all cases where $ \alpha_a^L(c) \neq \emptyword $ and $ \alpha_{b,k_0}^R(d) \neq \emptyword $, $w$ is a subword of $G_{k_0}(cvd)$.
    Furthermore, in all cases except those where $ a = 1^{k_0}3 $ and $ c = 2 $, it can be verified that $ G_{k_0}(cvd) = fwg $ for nonempty words $f$ and $g$ where $\alpha_a^L(c)$ is the last character of $f$ and $\alpha_{b,k_0}^R(d)$ is the last character of $g$. This means that $ cvd \in \lang^{(1)} $ implies $ \parens{\alpha_a^L(c), \alpha_{b,k_0}^R(d)} \in \lang $.
    
    For the case where $ a = 1^{k_0}3 $ and $ c = 2 $, notice that since $12$ is forbidden, $ 2vd \in \lang^{(1)} $ implies $ 22vd \in \lang^{(1)} $ or $ 32vd \in \lang^{(1)} $. It can be verified by computations similar to those above that $ 31^{k_0}3G_{k_0}(v)G_{k_0}(d) $ is a subword of both $G_{k_0}(22vd)$ and $G_{k_0}(32vd)$, and hence that $ 2vd \in \lang^{(1)} $ implies $ \parens{3, \alpha_{b,k_0}^R(d)} \in E(w) $.
    This completes the proof of (\ref{extension-pf-1}).
    
    It can be proven that
    $$ E(w) \subseteq \set{\parens{\alpha_a^L(c), \alpha_{b,k_0}^R(d)}: (c,d) \in E(v), \, \alpha^L(c) \neq \emptyword, \, \alpha^R(c) \neq \emptyword} $$
    by going through all cases and applying Proposition \ref{prop:eee-ante}.
    To show how these computations work, we give two examples.
    Suppose that $ a = 1^{k_0}3 $, $ b = \emptyword $, and $ (3,2) \in E(w) $. We have $ 3w2 = 31^{k_0}3G_{k_0}(v)2 \in \lang $, so by applying Proposition \ref{prop:eee-ante} to $3w2$ (here the $v$ in Proposition \ref{prop:eee-ante} is actually $2v1$, with $a=3$ and $b=\emptyword$), we know that $ 22v1 \in \lang^{(1)} $ or $ 32v1 \in \lang^{(1)} $, so $ (2,1) \in E(v) $.
    Indeed, $ \alpha_a^L(2) = 3 $ and $ \alpha_{b,k_0}^R(1) = 2 $, so
    $$ (3,2) \in \set{\parens{\alpha_a^L(c), \alpha_{b,k_0}^R(d)}: (c,d) \in E(v), \, \alpha^L(c) \neq \emptyword, \, \alpha^R(c) \neq \emptyword}. $$
    For another example, suppose that $ k_0 \geq 1 $, $ a = \emptyword $, $ b = 1^{k_0} $, and $ (3,1) \in E(w) $. We have $ 3w1 = 3G_{k_0}(v)1^{k_0+1} \in \lang $, so by Proposition \ref{prop:eee-ante}, $ 2v3 \in E(v) $ or $ 3v3 \in E(v) $. Indeed, $ \alpha_a^L(2) = \alpha_a^L(3) = 3 $ and $ \alpha_{b,k_0}^R(3) = 1 $, so necessarily
    $$ (3,1) \in \set{\parens{\alpha_a^L(c), \alpha_{b,k_0}^R(d)}: (c,d) \in E(v), \, \alpha^L(c) \neq \emptyword, \, \alpha^R(c) \neq \emptyword}. $$
\end{proof}

The following is a direct consequence of Propositions \ref{prop:bispecial-a-b} and \ref{prop:extension}:

\begin{corollary} \label{cor:extension-descendent}
    If $ v \in \lang^{(1)} $ is the antecedent of $ w \in \lang $, then $ \abs{E^-(w)} \leq \abs{E^-(v)} $ and $ \abs{E^+(w)} \leq \abs{E^+(v)} $.  In particular, if $w$ is bispecial, then $v$ is bispecial.
\end{corollary}

We also have the following, which allows us to narrow down our consideration of bispecial factors, setting the stage for an inductive characterization of them:

\begin{corollary} \label{cor:possible-ext-image}
    For any bispecial $ w \in \lang $ containing at least one $2$ or $3$, there exists $ a \in \set{\emptyword, 1^{k_0}3} $, a bispecial factor $ v \in \lang^{(1)} $, and $ b \in \set{0, 1^{k_0}} $ such that $ w = aG_{k_0}(v)b $.
\end{corollary}

\subsection{The age of a bispecial word via the map $A$}

Corollary \ref{cor:possible-ext-image} allows us to make the following definition:
\begin{definition}
    We define the function
    $$ A: \set{\mbox{Bispecial words in $\lang$}} \setminus \set{1^j: 0 \leq j \leq k_0+1} \rightarrow \set{\mbox{Bispecial factors in $\lang^{(1)}$}} $$
    by $ A(w) = v $, where $v$ is the antecedent of $w$.
    We call $w$ the \keyterm{extended image} of $v$.
\end{definition}

We use the function $A$ to relate properties of a bispecial word $ w \in \lang $ to properties of the simpler antecedent factor $ A(w) \in \lang^{(1)} $ via an inductive process.
As part of this, we will need to be able to iterate the function $A$ until we arrive at a base case. In doing this, for a bispecial $ w \in \lang $, $A^m(w)$ will be a bispecial word in $\lang^{(m)}$, since $A$ maps a bispecial word in $\lang^{(k)}$ to its antecedent word in $\lang^{(k+1)}$.
The following allows us to perform such induction:

\begin{proposition} \label{bispecialage}
    For all $w$ in $\mbox{Bispecial} (\bfu)$, there exists $ m \in \NatZero $ so that with
    $$w_m := A^m(w)=A \circ A \circ \hdots \circ A(w) $$
    we have
    $$ w_m = 1^j $$
    for some $ j \in \NatZero $.
\end{proposition}

Note that if $w=1^j$, then its age is zero.
Thus the proposition is stating that the age of every bispecial word in $\lang$ exists and is finite.

\begin{proof}
    We have the following sequence, letting $w= w_0,$
        
    \begin{eqnarray*}
    w_0   &=& a_0 G_{k_0}(w_1) b_0  \\ 
    w_1   &=& a_1 G_{k_1}(w_2) b_1  \\ 
    w_2   &=& a_2 G_{k_2}(w_3) b_2  \\ 
    &\vdots&
    \end{eqnarray*}
    which continues until some $w_m = 1^j$, where $j=0,1,2,\ldots.$  Our proposition is that this sequence must eventually stop.
    
    To prove this, it suffices to prove that for all $i$, if $ w_{i+2} \neq 1^j $, then $ \abs{w_i} > \abs{w_{i+2}} $.
    Indeed, $ \parens{G_{k_i} \circ G_{k_{i+1}}}(w_{i+2}) $ is a subword of $w_i$. We may compute the composition of two substitutions as
    $$ G_{k_i} \circ G_{k_{i+1}}: 1 \mapsto 1^{k_0}3, 2 \mapsto 2^{k_1} 1^{k_0+1} 3, 3 \mapsto 2^{k_1+1} 1^{k_0+1} 3. $$
    Therefore, $ \parens{G_{k_i} \circ G_{k_{i+1}}}(1) $ has length at least $1$, and $ \parens{G_{k_i} \circ G_{k_{i+1}}}(2) $ and $ \parens{G_{k_i} \circ G_{k_{i+1}}}(3) $ have length at least $2$, so if $w_{i+2}$ contains at least one occurrence of the character $2$ or $3$, then we have $ \abs{w_i} \geq \abs{\parens{G_{k_i} \circ G_{k_{i+1}}}(w_{i+2})} > \abs{w_{i+2}} $, and the proof is complete.
\end{proof}

The value of $m$ in Proposition \ref{bispecialage} must be unique since $A(w_m)$ is not defined when $ w_m = 1^j $ for some $j$. We call such $m$ the \keyterm{age} of a bispecial factor $w$.
Also, if $ v = A^k(w) $ for some $k$, then we say $w$ is a \keyterm{descendant} of $v$, or a \keyterm{descendant of order $k$}.

\subsection{All age zero bispecial words} \label{sec:eee-age-0}

We will soon recursively create a list of bispecial words $\set{w_m^+}$ and $\set{w_m^-}$, each of age $m$ that will contain all non-neutral bispecial words in $\lang$. This means in
particular that $A^m(w_m^+)$ and $A^m(w_m^-)$ have age $0$ as words in $\lang^{(m)}$ and $A^{m-1}(w_m^+)$ and $A^{m-1}(w_m^-)$ have age $1$ as words in $\lang^{(m-1)}.$

Given any bispecial $ w_m \in \lang $ of age $m$ such that $A^m(w_m) = 1^j$,  we will understand the extension set $E(w_m)$ (and thus the multiplicity $m(w_m)= \abs{E(w_m)} - \abs{E^-(w_m)} - \abs{E^+(w_m)} + 1$) if we can determine the extension sets of some of its lower-age antecedents. Since there are only finitely many factors of age zero, this sets the stage for an inductive argument.
Initially, we might hope to compute the multiplicity $m(w_m)$ of $w_m$ by knowing the multiplicity of $A^m(w_m) = 1^j$, but actually we will also need to know the extension set $A^{m-1}(w_m)$. Thus we need to  calculate the extension sets of both age $0$ and age $1$ factors before proceeding further.

In section \ref{sec:eee-length2}, we found the extension set $E(\emptyword).$
For the other possible age $0$ bispecial words, we have the following:

\begin{proposition} \label{age0}
The possible extension sets and the multiplicities for the word $1^j$ is given in the table
$$\begin{array}{c|c| c} 
&  E(1^j) & m(1^j)\\
\hline \hline
j=k_0>0 , k_1=0 & (1,3), (2,1), (3,1), (3,3)  &0 \\
j=k_0>0, k_1>0 &    (1,3), (2,1), (3,3) &    -1  \\
1\leq j < k_0&  (1,1), (1,3), (2,1), (3,1) & 0 \\

\end{array}$$
\end{proposition}

The corresponding extension diagrams are
$$ \begin{tabular}{c|ccc}
       $\begin{array}{c} j=k_0>0, \\ k_1 = 0 \end{array}$     & 1 & 2 & 3 \\ \hline
            1 & & & $\times$ \\
            2 & $\times$ & & \\
            3 & $\times$ &  & $\times$
        \end{tabular}, \qquad \begin{tabular}{c|ccc}
             $\begin{array}{c} j=k_0>0, \\ k_1 > 0 \end{array}$& 1 & 2 & 3 \\ \hline
            1 & & & $\times$ \\
            2 & $\times$ &  & \\
            3 & & & $\times$
        \end{tabular}, \qquad \begin{tabular}{c|ccc}
           $ 1\leq j < k_0 $& 1 & 2 & 3 \\ \hline
                1 & $\times$ & & $\times$ \\
                2 & $\times$ & & \\
                3 & $\times$ &  &
        \end{tabular} $$

Note that $m(1^j)$ is nonzero (equaling $-1$ only when $j=k_0>0, k_1>0$), which is mirrored in the fact that $m(\emptyword)$ is nonzero (equaling now $1$)  only when $ k_1>0.$  Thus the two types of age $0$ bispecial words will have multiplicities of opposite sign.  We will eventually show that this happens for all ages.

\begin{proof}
    As to be expected, the proof is simply going through the various possibilities. 

    First consider the factor $1^{k_0}$ when $ k_0 > 0 $. We have the following:
    \begin{itemize}
        \item By Proposition \ref{two words}, $12$ is always forbidden, so $ (c,2) \notin E(1^{k_0}) $ for any $ c \in \set{1,2,3} $.
        \item Since $ 1^{k_0+2} \notin \lang $,  then $ (1,1) \notin E(1^{k_0}) $.
        \item Since $ 3 \in \lang^{(1)} $ and $ G_{k_0}(3) = 1^{k_0+1}3 \in \lang $,  we must have $ (1,3) \in E(1^{k_0}) $.
        \item By Proposition \ref{two words}, we know that $ 32 \in \lang^{(1)} $. Since $ G_{k_0}(32) = 1^{k_0+1}31^{k_0}3$, we get that $ (3,3) \in E(1^{k_0}).$
        \item By Proposition \ref{two words}, we know that $ 13 \in \lang^{(1)} $. Since $ G_{k_0}(13) = 21^{k_0+1}3$, we get that $ (2,1) \in E(1^{k_0}).$
        \item $ (2,3) \in E(1^{k_0}) $ means that $ 2 1^{k_0}3 \in \lang $, which by de-substitution requires $ 12 \in \lang^{(1)} $. As $12$ is always forbidden, this cannot happen, hence $(2,3) \notin E(1^{k_0})$
        \item If $ k_1 = 0 $, then $ 33 \in \lang^{(1)} $,  giving us that $ G_{k_0}(33) = 1^{k_0+1}31^{k_0+1}3 \in  \lang $, which  which in turn means $ (3,1) \in E(1^{k_0}) $. Conversely, if $ (3,1) \in E(1^{k_0}) $, then $ 31^{k_0}3 \in \lang $, which means that $23$ or $33$ must be in $\lang^{(1)}$. However, by Proposition \ref{two words}, neither  $ 23 $  nor  $ 33$  can be in $\lang^{(1)}$ when  $ k_1 = 0 $.  Thus  $ (3,1) \in E(1^{k_0}) $ if and only if $ k_1 = 0 $.
    \end{itemize}
    
    We have the following regarding the factor $1^j$ for $ 1 \leq j \leq k_0 - 1 $:
    \begin{itemize}
        \item By the same arguments as above, $ (c,2) \notin E(1^j) $ for any $c$, and $ (1,3) \in E(1^j) $.
        \item Since all $1$s in $w$ must occur in groups of $k_0$ or $k_0+1$, $ (2,3), (3,3) \notin E(1^j) $.
        \item Since $ j + 2 \leq k_0 + 1 $, $1^{j+2}$ is a subword of $ G_{k_0}(3) = 1^{k_0+1}3 $, giving us that  $ (1,1) \in E(1^j) $.
        \item Since $ 13 \in \lang^{(1)} $  by Proposition \ref{two words}, we have that $21^{j+1}$ is a subword of $ G_{k_0}(13) = 21^{k_0+1}3 $, giving us that $ (2,1) \in E(1^j) $.
        \item Since $ 21 \in \lang^{(1)} $  by Proposition \ref{two words}, we have that $31^{j+1}$ (as a subword of  $G_{k_0}(21)$) must be in $\lang$, giving us that $(3,1) \in E(1^j) $,
    \end{itemize}

\end{proof}

\subsection{All age one bispecial words} \label{sec:eee-age-1}

We now find all bispecial words of age $1$.

\begin{proposition}\label{eeeageonefactors}
    The following are the only possible bispecial words in $\lang$ of age $1$:
    \begin{itemize}
        \item $2^j$ for $ 1 \leq j \leq k_1 $, with left extension set $\set{2,3}$
        \item $ 1^{k_0}3 $, with left extension set $\set{1,3}$
        \item $ 1^{k_0}31^{k_0} $, if $ k_0 \geq 1 $, $ k_1 = 0 $, and $ k_2 \geq 1 $, with left extension set $\set{1,3}$
        \item $ 1^{k_0}32^{k_1} $, if $ k_1 \geq 1 $, with left extension set $\set{1,3}$
    \end{itemize}
    There are no additional bispecial words of age $1$. Furthermore:
    \begin{itemize}
        \item If $ k_1 \geq 1 $ and $ k_2 \geq 1 $, then $ m(1^{k_0}3) = 1 $, $ m(1^{k_0}32^{k_1}) = -1 $, and all other bispecial factors of age $1$ are neutral.
        \item If $ k_1 = 0 $ or $ k_2 = 0 $, then all bispecial factors of age $1$ are neutral.
    \end{itemize}
\end{proposition}

\begin{proof}
    We know that any nonneutral bispecial word $w$ of age $1$ must be of the form
    $$w= a G_{k_0}(v)b$$
    where $a$ is either the empty word $\emptyword$ or $1^{k_0}3$, $b$ is either the empty word $\emptyword$ or $1^{k_0}$ and $v$ is a bispecial word of age $0.$ Also, as a bispecial word in $\lang^{(1)}$ of age $0$, $v$ must be either the empty word $\emptyword$ or equal to $1^j$ for $ 1 \leq j \leq k_0 $. Furthermore, $w$ must contain at least one $2$ or $3$, since otherwise $w$ would have age $0$. This limits the possible age $1$ bispecial words in $\lang^{(1)}$ to the following:
    \begin{itemize}
        \item $2^j$ for $ 1 \leq j \leq k_1 $
        \item $ 2^j 1^{k_0} $ for $ 1 \leq j \leq k_1 $ when $ k_0 \geq 1 $
        \item $ 1^{k_0}32^j $ for $ 0 \leq j \leq k_1 $
        \item $ 1^{k_0}32^j 1^{k_0} $ for $ 0 \leq j \leq k_1 $ when $ k_0 \geq 1 $
    \end{itemize}
    
    We first show that the words $ 2^j 1^{k_0} $ and $ 1^{k_0}32^j 1^{k_0} $ are not bispecial for $ 1 \leq j \leq k_1 $. Indeed, $12$ is always forbidden, so $ 2 \notin E^+(2^j 1^{k_0}) $. Also, $ 2^j 1^{k_0}3 \in \lang $ requires $ 1^j 2 \in \lang^{(1)} $, which cannot happen since $12$ is forbidden, so $ 3 \notin E^+(2^j 1^{k_0}) $. This means $ E^+(2^j 1^{k_0}) = \set{1} $, so $ 2^j 1^{k_0} $ is not right special. This implies $ 1^{k_0}3 2^j 1^{k_0} $ is also not right special.
    
    We now show that $ 1^{k_0}3 2^j $ is not bispecial when $ 1 \leq j < k_1 $. Indeed, $ 3 \notin E^+(1^{k_0}3 2^j) $ since $23$ is always forbidden. Also, $ 1^{k_0}32^j 1 \in \lang $ requires one of the words $ 21^j 2 $, $ 21^j 3 $, $ 31^j 2 $, or $ 31^j 3 $ to be in $\lang^{(1)}$, which cannot happen since all $1$s must occur in groups of $k_1$ or $k_1+1$. Therefore, $ E^+(1^{k_0}3 2^j) = \set{2} $, so this word is not right special.
    
    The only potential bispecial words of age $1$ that remain are $2^j$ for $ 1 \leq j \leq k_1 $, $1^{k_0}3$, $1^{k_0}31^{k_0}$, and $1^{k_0}32^{k_1}$.
    The computation of extension diagrams in the following proposition completes the proof.
\end{proof}

For ease of reading the following proposition, we will summarize at the end of this subsection its results in a table that   will record the size of the left, right, and total extension sets, and the corresponding multiplicities.

\begin{proposition} \label{age one} 
    We have the following extension diagrams for the words of age 1 remaining from the proof of Proposition \ref{eeeageonefactors}:
        \begin{itemize}
        \small
        \item $ k_1 \geq 1 $: \qquad
        \begin{tabular}{c|ccc}
                $2^j$ & 1 & 2 & 3 \\ \hline
                1 & & & \\
                2 & $\times$ & $\times$ & \\
                3 & & $\times$ &
            \end{tabular} for $ 1 \leq j \leq k_1 - 1$; \qquad
            \begin{tabular}{c|ccc}
                $2^{k_1}$ & 1 & 2 & 3 \\ \hline
                1 & & & \\
                2 & $\times$ & & \\
                3 & $\times$ & $\times$ &
            \end{tabular}
        \item $ k_0 = 0, k_1 = 0, k_2 = 0 $: \qquad
            \begin{tabular}{c|ccc}
                $3$ & 1 & 2 & 3 \\ \hline
                1 & $\times$ & $\times$ & $\times$ \\
                2 & & & \\
                3 & & $\times$ &
            \end{tabular}
        \item $ k_0 = 0, k_1 = 0, k_2 \geq 1 $: \qquad
            \begin{tabular}{c|ccc}
                $3$ & 1 & 2 & 3 \\ \hline
                1 & $\times$ & & $\times$ \\
                2 & & & \\
                3 & & $\times$ & $\times$
            \end{tabular}
        \item $ k_0 = 0, k_1 \geq 1, k_2 = 0 $: \qquad
            \begin{tabular}{c|ccc}
                $3$ & 1 & 2 & 3 \\ \hline
                1 & & $\times$ & $\times$ \\
                2 & & & \\
                3 & & $\times$ &
            \end{tabular}; \qquad
            \begin{tabular}{c|ccc}
                $32^{k_1}$ & 1 & 2 & 3 \\ \hline
                1 & $\times$ & $\times$ & \\
                2 & & & \\
                3 & & $\times$ &
            \end{tabular}
        \item $ k_0 = 0, k_1 \geq 1, k_2 \geq 1 $: \qquad
            \begin{tabular}{c|ccc}
                $3$ & 1 & 2 & 3 \\ \hline
                1 & & $\times$ & $\times$ \\
                2 & & & \\
                3 & & $\times$ & $\times$
            \end{tabular}; \qquad
            \begin{tabular}{c|ccc}
                $32^{k_1}$ & 1 & 2 & 3 \\ \hline
                1 & $\times$ & & \\
                2 & & & \\
                3 & & $\times$ &
            \end{tabular}
        \item $ k_0 \geq 1, k_1 = 0, k_2 = 0 $: \qquad
            \begin{tabular}{c|ccc}
                $1^{k_0}3$ & 1 & 2 & 3 \\ \hline
                1 & $\times$ & $\times$ & \\
                2 & & & \\
                3 & & $\times$ &
            \end{tabular}; \qquad
            \begin{tabular}{c|ccc}
                $1^{k_0}31^{k_0}$ & 1 & 2 & 3 \\ \hline
                1 & $\times$ & & $\times$ \\
                2 & & & \\
                3 & & &
            \end{tabular}
        \item $ k_0  \geq 1, k_1 = 0, k_2 \geq 1 $: \qquad
            \begin{tabular}{c|ccc}
                $1^{k_0}3$ & 1 & 2 & 3 \\ \hline
                1 & $\times$ & & \\
                2 & & & \\
                3 & $\times$ & $\times$ &
            \end{tabular}; \qquad
            \begin{tabular}{c|ccc}
                $1^{k_0}31^{k_0}$ & 1 & 2 & 3 \\ \hline
                1 & $\times$ & & $\times$ \\
                2 & & & \\
                3 & & & $\times$
            \end{tabular}
        \item $ k_0  \geq 1, k_1 \geq 1, k_2 = 0 $: \qquad
            \begin{tabular}{c|ccc}
                $1^{k_0}3$ & 1 & 2 & 3 \\ \hline
                1 & $\times$ & $\times$ & \\
                2 & & & \\
                3 & & $\times$ &
            \end{tabular}; \qquad
            \begin{tabular}{c|ccc}
                $1^{k_0}31^{k_0}$ & 1 & 2 & 3 \\ \hline
                1 & & & $\times$ \\
                2 & & & \\
                3 & & &
            \end{tabular}; \qquad
            \begin{tabular}{c|ccc}
                $1^{k_0}32^{k_1}$ & 1 & 2 & 3 \\ \hline
                1 & $\times$ & $\times$ & \\
                2 & & & \\
                3 & & $\times$ &
            \end{tabular}
        \item $ k_0  \geq 1, k_1 \geq 1, k_2 \geq 1 $: \qquad
            \begin{tabular}{c|ccc}
                $1^{k_0}3$ & 1 & 2 & 3 \\ \hline
                1 & $\times$ & $\times$ & \\
                2 & & & \\
                3 & $\times$ & $\times$ &
            \end{tabular}; \qquad
            \begin{tabular}{c|ccc}
                $1^{k_0}31^{k_0}$ & 1 & 2 & 3 \\ \hline
                1 & & & $\times$ \\
                2 & & & \\
                3 & & & $\times$
            \end{tabular}; \qquad
            \begin{tabular}{c|ccc}
                $1^{k_0}32^{k_1}$ & 1 & 2 & 3 \\ \hline
                1 & $\times$ & & \\
                2 & & & \\
                3 & & $\times$ &
            \end{tabular}
    \end{itemize}

\end{proposition}

\begin{proof}
    There are 18 extension diagrams that must be checked.  Each is its own straightforward calculation.  To give a flavor, we will show how to get two of these extension sets, namely the third and the last of the above extension diagrams. 
    
    First, assume that $k_0=k_1=k_2=0$.  We have to find the extension set of $3$.
    From our knowledge of 2 letter words, we may exclude a left extension of $2$, so the list 
    $$131, 132, 133, 331, 332, 333$$
    are the only possible three letter words that give rise to the extension set of the middle $3$.  We have to  explicitly show that
    $131, 132, 133, 332$ do occur while $331$ and $333$ cannot occur.
    
    Since $ k_1 = k_2 = 0 $, $ 33 \in \lang^{(1)} $, and $131$ is a subword of $ G_{k_0}(33) = 1313 $, so $ 131 \in \lang $.
    
    Since $ 31 \in \lang^{(1)} $ and $ G_{k_0}(31) = 132 $, $ 132 \in \lang $.
    
    Since $ 32 \in \lang^{(1)} $ and $ G_{k_0}(32) = 133 $, $ 133 \in \lang $.
    
    If $331$ is in $\lang$, then it must be derived from the substitution of either $223$ or $323$, so $ 223 \in \lang^{(1)} $ or $ 323 \in \lang^{(1)} $, but $23$ is forbidden, so this is impossible.
    
    Similarly, if $333$ is in $\lang$, then either $ 222 \in \lang^{(1)} $ or $ 322 \in \lang^{(1)} $, but since $ k_2 = 0 $, $ 22 \notin \lang^{(1)} $ which yields a contradiction.
    
    All that is left is to show that $ 332 \in \lang $. We know that $ 13 \in \lang^{(3)} $ always, and we can compute
    $$G_{k_0}(G_{k_1}(G_{k_2}(13) ))=G_{k_0}(G_{k_1}(213) )=  G_{k_0}(3213) = 133213. $$
    As a subword of this, $ 332 \in \lang $.
    
    For our second case, let us assume that $k_0, k_1, k_2$ are all positive.  We find the extension set of 
    $1^{k_0} 3 2^{k_1}.$  From our knowledge of 2 letter words, we may exclude a right extension of $3$, so the list 
    $$    11^{k_0}32^{k_1} 1, 11^{k_0}32^{k_1} 2, 21^{k_0}32^{k_1} 1, 21^{k_0}32^{k_1} 2, 31^{k_0}32^{k_1} 1, 31^{k_0}32^{k_1} 2    $$
    are the only possible words that give rise to the extension set of the middle $1^{k_0}32^{k_1} $.  
    We must show that $  11^{k_0}32^{k_1} 1$ and $ 31^{k_0}32^{k_1} 2  $ are in $\lang$, while $ 11^{k_0}32^{k_1} 2,   21^{k_0}32^{k_1} 1, 21^{k_0}32^{k_1} 2, 31^{k_0}32^{k_1} 1  $ are not.
    
    The word $32$ is always present, and hence is in $\lang^{(2)}$,  We have that $11^{k_0}32^{k_1} 1$ is a subword of 
    $$G_{k_0}(G_{k_1}(32) )=G_{k_0}(  1^{k_1+1}3    1^{k_1} 3      )=  2^{k_1+1}1^{k_0+1}3  2^{k_1} 1^{k_0+1}3$$
    and hence is in $\lang$.
    
    Similarly, $ 3 \in \lang^{(3)} $, so as a subword of
    $$G_{k_0}(G_{k_1}(G_{k_2}(3) ))=G_{k_0}(G_{k_1}(  1^{k_2+1}3  ) )=  G_{k_0}(    2^{k_2+1}  1^{k_1+1}3               ) = ( 1^{k_0}3)^{k_2+1}  2^{k_1+1} 1^{k_0+1}3, $$
    $ 31^{k_0}32 \in \lang $.
    
    The only way that $11^{k_0}32^{k_1} 2$ can be in $\lang$ is for it to be derived from a substitution of $31^{k_1+1}3$ in $\lang^{(1)}$, which in turn must be derived from a substitution $G_{k_1}$ of $23$ or $33$ in $\lang^{(2)}$.  As $k_2>0$, both $23$ and $33$ are forbidden, so $ 11^{k_0}32^{k_1}2 \notin \lang $.
    
    The word $31^{k_0}32^{k_1} 1$ can be in $\lang$ only if one of the words $2 2 1^{k_1}2, 2 2 1^{k_1}3, 2 3 1^{k_1}2$  or $3 2 1^{k_1}3$ are in $\lang^{(1)}$.
    As $12$ is always forbidden, we now have to show that $2 2 1^{k_1}3$  and  $3 2 1^{k_1}3$  cannot be in $\lang^{(1)}$. Both contain $21^{k_1}3$, which would require a forbidden $12$ in $\lang^{(2)}$, and thus cannot be in $\lang^{(1)}$.
    
    Now consider $ 21^{k_0}32^{k_1} 1 $ and $ 21^{k_0}32^{k_1} 2$.  Both start with $2 1^{k_0}3 $, which must be derived from a substitution of a forbidden $12$, so neither $2 1^{k_0}3  1$ nor  $2 1^{k_0}3  2$ are in $\lang$.

As mentioned before the last proposition, here  we summarize the results from the extension diagrams in terms of tables that record the size of the left, right, and total extension sets.

    $$ \begin{array}{c|c|c|c|c}
        k_0, k_1, k_2     & |E(1^{k_0} 3)| & |E^-(1^{k_0} 3)|  & |E^+(1^{k_0} 3)|  & m(1^{k_0}3) \\ \hline
        k_0=0,k_1=0, k_2=0 &    4    &   2   &  3   &    0       \\
        \hline
         k_0=0,k_1=0, k_2>0 &   4     &  2    &  3   &     0      \\
        \hline
         k_0=0,k_1>0, k_2=0 &   3     &   2   &  2   &    0       \\
        \hline
         k_0=0,k_1>0, k_2>0 &      4  &   2   &  2   &    1       \\
        \hline
         k_0>0,k_1=0, k_2=0 &   3     &   2   &  2   &    0       \\
        \hline
         k_0>0,k_1=0, k_2>0 &   3     &    2  & 2    &     0      \\
        \hline
         k_0>0,k_1>0, k_2=0 &   3     &  2    &  2   &    0       \\
        \hline
         k_0>0,k_1>0, k_2>0 &    4    &   2   &  2   &     1      \\
        \hline
                    \end{array} $$
                    
    $$ \begin{array}{c|c|c|c|c}
        k_0, k_1, k_2     & |E(1^{k_0} 32^{k_1})| & |E^-(1^{k_0} 32^{k_1})|  & |E^+(1^{k_0} 32^{k_1})|  & m(1^{k_0}32^{k_1}) \\ \hline
                 k_0=0,k_1=0, k_2>0 &    3   &  2    &  2  &     0     \\
        \hline
         k_0=0,k_1>0, k_2>0 &   2     &   2   &  2  &     -1     \\
        \hline
         k_0>0,k_1=0, k_2>0 &   3    &  2    & 2    &    0       \\
        \hline
         k_0>0,k_1>0, k_2>0 &  2      &  2   &  2  &      -1    \\
        \hline
                    \end{array} $$

   $$ \begin{array}{c|c|c|c|c}
        k_0, k_1, k_2  & |E(2^j)| & |E^-(2^j)| & |E^+(2^j)| & m(2^j) \\ \hline 
        k_1 \ge 1 & 3 & 2 & 2 & 0 
    \end{array} , \qquad
    1 \le j \le k_1 - 1 
    $$

    $$\begin{array}{c|c|c|c|c}
         k_0, k_1, k_2 & |E(2^{k_1})| & |E^-(2^{k_1})| & |E^+(2^{k_1})| & m( 2^{k_1}) \\ \hline 
         k_1 \ge 1 & 3 & 2 & 2 & 0
    \end{array}$$

    $$
    \begin{array}{c|c|c|c|c}
         k_0, k_1, k_2 & |E(1^{k_0} 3 1^{k_0})| & |E^- (1^{k_0} 3 1^{k_0})| & |E^+ (1^{k_0} 3 1^{k_0})| & m(1^{k_0} 3 1^{k_0}) \\ \hline 
         k_0 \ge 1, k_1 = 0, k_2 = 0 & 2 & 1 & 2& 0 \\ \hline 
         k_0 \ge 1, k_1 =0, k_2 \ge 1 & 3 & 2 & 2 & 0 \\ \hline
         k_0 \ge 1, k_1 \ge 1, k_2 = 0 & 1 & 1 & 1 & 0 \\ \hline
         k_0 \ge 1, k_1 \ge 1, k_2 \ge 1 & 2 & 2 & 1 & 0 
    \end{array}
    $$
\end{proof}

\subsection{Characterizing Extended Images of Bispecial Words} \label{sec:eee-ext-img}

First to recap the notation and the goals.  We have an $S$-adic
word ${\bf u}$ described via a Gauss coding sequence $\{ G_{k_0}, G_{k_1}, G_{k_2}, \ldots   \}$.  Our language $\lang$ is all subwords of ${\bf u}$.   We need to understand all of the bispecial factors in ${\bf u}$.  We do this in part by looking at the bispecial factors in $\lang^{1}$, which is the language of the Gauss coding sequence of $\{ G_{k_1}, G_{k_2}, G_{k_3}, \ldots   \}$, etc. 

In this subsection, we aim to characterize all possible extended images $ w \in \lang $ of a bispecial word $ v \in \lang^{(1)} $ based on its extension set $E(v)$ and the value of $k_0$.

\begin{proposition} \label{prop:ext-img-list}
    Suppose $ v \in \lang^{(1)} $ has age at least $1$. Then, for all extended images $ w = aG_{k_0}(v)b $:
    \begin{itemize}
 \item If $ E^-(v) = \set{1,3} $, then $ a = \emptyword $ and $ E^-(w) = \set{2,3} $.
        \item If $ E^-(v) = \set{2,3} $, then $ a = 1^{k_0}3 $ and $ E^-(w) = \set{1,3} $.
    \end{itemize}
    Furthermore:
    \begin{itemize}
        \item If $ E^+(v) = \set{1,2} $ or $\set{1,3}$, then $ w = aG_{k_0}(v) $ is the unique extended image.
        \item If $ E^+(v) = \set{2,3} $, then $ w = aG_{k_0}(v)1^{k_0} $ is the unique extended image.
        \item If $ E^+(v) = \set{1,2,3} $ and $ k_0 = 0 $, then $ w = aG_{k_0}(v) $ is the unique extended image and $ E^+(w) = \set{1,2,3} $.
        \item If $ E^+(v) = \set{1,2,3} $ and $ k_0 \geq 1 $, then there are two extended images: $ w_+ = aG_{k_0}(v) $ with $ E^+(w_+) = \set{1,2} $, and $ w_- = aG_{k_0}(v)1^{k_0} $ with $ E^+(w_-) = \set{1,3} $.
    \end{itemize}
\end{proposition}

\begin{proof}
    These results are all direct applications of Proposition \ref{prop:extension}.
\end{proof}

\begin{lemma} \label{lem:eee-left-ext}
    Let $ w \in \lang $ be bispecial. If the age of $w$ is odd, then $ E^-(w) = \set{1,3} $, while if the age of $w$ is even and nonzero, then $ E^-(w) = \set{2,3} $.
\end{lemma}

\begin{proof}
    By the results of Subsection \ref{sec:eee-age-1}, $ E^-(w) = \set{1,3} $ for all bispecial words of age $1$. The lemma can be proven inductively by applying Proposition \ref{prop:ext-img-list}, using this as a base case.
\end{proof}

Owing to Lemma \ref{lem:eee-left-ext} and Proposition \ref{bispecialage}, we know all the bispecial factors of age $\ge 1$ have left extension set $\set{1,3}$ or left extension set $\set{2,3}$, so we can restrict our attention to these.

We proceed with our goal of characterizing all non-neutral bispecial $ w \in \lang $ and the values of $m(w)$.
We begin with the case where the extended image is unique.

\begin{proposition}
    Suppose that $ v \in \lang^{(1)} $ is bispecial with age at least $1$. If $ \abs{E^+(v)} = 2 $ or $ k_0 = 0 $, then there is only one extended image $ w = aG_{k_0}(v)b $ of $v$ and the function $\alpha_{b,k_0}^R$ is a bijection from $E^+(v)$ to $E^+(w)$. Consequently, $ \abs{E^+(v)} = \abs{E^+(w)} $ and $ m(v) = m(w) $.
\end{proposition}
For foreshadowing purposes, we will see  in  subsection \ref{Construction of a sequence containing all non-neutral bispecial words} that this allows us to define various $A^{-1}(w)$.

\begin{proof}
    First notice that Lemma \ref{lem:eee-left-ext} and Proposition \ref{prop:ext-img-list} guarantee that either $ E^-(v) = \set{1,3} $ or $ E^-(v) = \set{2,3} $ and that $\alpha_a^L$ is a bijection from $E^-(v)$ to $E^-(w)$.
    It can then be verified by applying Proposition \ref{prop:extension} on all cases that if $ \abs{E^+(v)} = 2 $ or $ k_0 = 0 $, only one choice of $b$ makes $w$ right special, and for this $b$, $\alpha_{b,k_0}^R$ is injective when restricted to $E^R(v)$, meaning it is a bijection from $E^+(v)$ to $E^+(w)$.
    Hence, $E(w)$ can be derived from $E(v)$ by applying a permutation to both the left extensions and the right extensions. Since permutations preserve $\abs{E}$, $\abs{E^-}$, and $\abs{E^+}$, this means $ \abs{E^+(v)} = \abs{E^+(w)} $ and $ m(v) = m(w) $.
\end{proof}

In the case where $ \abs{E^+(v)} = 2 $ for some bispecial $ v \in \lang^{(k)} $, for any $k$, the above results imply that $v$ has a unique descendant $ w \in \lang $ of order $k$ and that $ m(v) = m(w) $. Therefore, no further investigation is required on bispecial words with only two right extensions.
In fact, based on the results of Subsection \ref{sec:eee-age-1}, we may conclude the following:

\begin{corollary} \label{cor:eee-originate}
    Every non-neutral bispecial word with age at least $1$ is a descendant of $1^{k_0}3$ or $1^{k_0}32^{k_1}$.
\end{corollary}

\begin{proof}
    Proposition \ref{eeeageonefactors} gives an exclusive list of the possible bispecial words of age $1$. However, by Proposition \ref{age one}, for all of the words $w$ in this list other than $1^{k_0}3$ and $1^{k_0}32^{k_1}$, $ \abs{E^-(w)} = \abs{E^-(w)} = 2 $ and $ m(w) = 0 $ whenever $w$ is bispecial.
\end{proof}

There is only one case where a bispecial word $ v \in \lang^{(1)} $ can have two extended images $w_+$ and $w_-$: that where $ \abs{E^+(v)} = 3 $ and $ k_0 \geq 1 $. In this case, to determine the values $m(w_+)$ and $m(w_-)$ for the two extended images, it is not sufficient to know the value of $m(v)$, so we need the entire extension set $E(v)$ to apply Proposition \ref{prop:extension}.
This requires characterizing the possible extension sets of all bispecial factors with three right extensions.

\begin{proposition} \label{prop:eee-3-right}
    Let $ w \in \lang $ be bispecial with age at least $1$ and $ \abs{E^+(w)} = 3 $. The extension set of $w$ takes one of the following six forms, for $ i \in \set{1,2} $:
    \begin{equation*}
        \begin{array}{c|ccc}
            & 1 & 2 & 3 \\ \hline
            i & \times & \times & \times \\
            3 & & \times &
        \end{array} \qquad \begin{array}{c|ccc}
            & 1 & 2 & 3 \\ \hline
            i & \times & \times & \times \\
            3 & & & \times
        \end{array} \qquad \begin{array}{c|ccc}
            & 1 & 2 & 3 \\ \hline
            i & \times & \times & \times \\
            3 & \times & &
        \end{array}
    \end{equation*}
    \begin{equation*}
        \begin{array}{c|ccc}
            & 1 & 2 & 3 \\ \hline
            i & \times & & \times \\
            3 & & \times & \times
        \end{array} \qquad \begin{array}{c|ccc}
            & 1 & 2 & 3 \\ \hline
            i & \times & \times & \\
            3 & \times & & \times
        \end{array} \qquad \begin{array}{c|ccc}
            & 1 & 2 & 3 \\ \hline
            i & & \times & \times \\
            3 & \times & \times &
        \end{array}
    \end{equation*}
\end{proposition}

\begin{proof}
    By the results of Subsection \ref{sec:eee-age-1}, the only possible age $1$ bispecial word with three right extensions is $3$, and only when $ k_0 = k_1 = 0 $. In this case, the extension diagram of $3$ is either the first or the fourth in this proposition, depending on the value of $k_2$.
    
    Suppose that $w$ has age at least $2$. By Proposition \ref{prop:ext-img-list}, the only way we can have $ \abs{E^+(w)} = 3 $ is if $ k_0 = 0 $ and $ \abs{E^+(v)} = 3 $ for the antecedent $v$. In this case, by Proposition \ref{prop:extension}, the extension diagram of $w$ can be derived from that of $v$ by applying the permutation $(123)$ to the columns and changing the value of $i$.
    By induction, if $w$ has age $m$, then the extension diagram of $w$ can be derived from that of the age $1$ bispecial word $ A^{m-1}(w) \in \lang^{(m-1)} $ by applying the permutation $(123)^{m-1}$ to the columns of that of $A^{m-1}(w)$ and possibly changing the value of $i$.
    Hence, the set of all possible such extension diagrams can be generated by cycling through the permutation $(123)$ on the columns of the first and fourth extension diagrams. This produces exactly the six possible extension diagrams given in the proposition.
\end{proof}

In the case where the extended image of a bispecial word is not unique, both extended images only have two right extensions, so the following is all that remains to characterize all non-neutral bispecial words $w$ and the values of $m(w)$:

\begin{proposition} \label{prop:eee-non-unique}
    Suppose that $ v \in \lang^{(1)} $ is bispecial with $ \abs{E^+(v)} = 3 $ and $ k_0 \geq 0 $. Let $ w_+ = aG_{k_0}(v) $ and $ w_- = aG_{k_0}(v)1^{k_0} $ be the two possible bispecial extended images. One of the following is true:
    \begin{itemize}
        \item $ m(w_+) = m(w_-) = 0 $
        \item $ m(w_+) = 1 $ and $ m(w_-) = -1 $
    \end{itemize}
\end{proposition}

\begin{proof}
    The extension diagram of $v$ must take one of the six forms in Proposition \ref{prop:eee-3-right}. The statement can be proven by applying Proposition \ref{prop:extension} to determine $E(w_+)$ and $E(w_-)$ in each case.
    Specifically, we have $ m(w_+) = 1 $ and $ m(w_-) = -1 $ if $v$ has the fifth extension diagram in Proposition \ref{prop:eee-3-right}, and $ m(w_+) = m(w_-) = 0 $ otherwise.
\end{proof}

Note that it is not always true that $v$ has two bispecial extended images when $ \abs{E^R(v)} = 3 $ and $ k_0 = 0 $. Particularly, if $v$ has the third extension diagram in Proposition \ref{prop:eee-3-right}, then $ aG_{k_0}(v)1^{k_0} $ is not bispecial. However, Proposition \ref{prop:eee-non-unique} is still valid.

\subsection{Construction of a sequence containing all non-neutral bispecial words}\label{Construction of a sequence containing all non-neutral bispecial words}

\begin{proposition}
    There is a sequence  
    $$w_0^+, w_0^-, w_1^+, w_1^-, \ldots , w_m^+, w_m^-, \ldots .$$
    of words in $\lang$ containing all non-neutral bispecial words such that:
    \begin{enumerate}
        \item For all $m$,
        $$ 0 \leq m(w_m^+) = -m(w_m^-) \leq 1 $$
        \item $$ \abs{w_0^+} \leq \abs{w_0^-} < \abs{w_1^+} \leq \abs{w_1^-} < \abs{w_2^+} \leq \abs{w_2^-} < \cdots$$
    \end{enumerate}
\end{proposition}

In this section we will define the sequence and prove the first part.
In the process, we will also show the following:
\begin{itemize}
    \item If $ w_m^+ = w_m^- $, then $ \abs{E^+(w_m^+)} = \abs{E^+(w_m^-)} = 3 $.
    \item If $ w_m^+ \neq w_m^- $, then $ \abs{E^+(w_m^+)} = \abs{E^+(w_m^-)} = 2 $.
\end{itemize}
We will prove the increase in the lengths of the elements of the sequence in the next section.

The various $w_k^+$ and $w_k^-$ are subwords of the initial language $\lang$.  There are analogous subwords  in each  of the languages $\lang^{(k)}$.  We will write these subwords as 
 $w_k^+(\lang^{(k)})$ and $w_k^-(\lang^{(k)})$ when it is important to specify the language.

Start by setting 
$$w_0^+ = \emptyword, \qquad w_0^-= 1^{k_0}.$$
These are the same when $k_0=0.$
Set
$$w_1^+ = 1^{k_0}3, \qquad w_1^-=1^{k_0}32^{k_1}.$$
These are equal precisely when  $k_1 = 0.$
We know that $ \set{w_0^+, w_0^-, w_1^+, w_1^-} $ includes all non-neutral bispecial words in $\lang$ with age $0$ or $1$, and by Corollary \ref{cor:eee-originate}, all non-neutral bispecial words with age $n$ are descendents of $w_1^+\parens{\lang^{(n-1)}}$ or $w_1^-\parens{\lang^{(n-1)}}$.

Now assume that we know how to construct $w_0^+, w_0^-, w_1^+, w_1^-, \ldots , w_m^+, w_m^-$ and that $ w_m^+ \neq w_m^- $ implies $ \abs{E^+(w_m^+)} = \abs{E^-(w_m^-)} = 2 $.  We must show how to construct the next two terms $w_{m+1}^+, w_{m+1}^-$.

We first look at the corresponding sequence for $\lang^{(1)}$,
$$ w_0^+\parens{\lang^{(1)}}, w_0^-\parens{\lang^{(1)}}, w_1^+\parens{\lang^{(1)}}, w_1^-\parens{\lang^{(1)}}, \hdots, w_m^+\parens{\lang^{(1)}}, w_m^-\parens{\lang^{(1)}}, \hdots. $$
We will now find $w_{m+1}^+, w_{m+1}^-$ from knowing $w_m^+\parens{\lang^{(1)}}, w_m^-\parens{\lang^{(1)}}.$

There are the following cases.
If $w_m^+\parens{\lang^{(1)}} \neq w_m^-\parens{\lang^{(1)}},$ then we know that $\abs{E^+(w_m^+\parens{\lang^{(1)}}} = \abs{E^+(w_m^-\parens{\lang^{(1)}}} = 2$, which means that both $A^{-1}(w_m^+)$  and $A^{-1}(w_m^-)$ are unique. Thus we set
$$w_{m+1}^+ = A^{-1}(w_m^+), \; w_{m+1}^- = A^{-1}(w_m^-).$$

Next suppose $ w_m^+\parens{\lang^{(1)}} = w_m^-\parens{\lang^{(1)}} $. If $A^{-1}(w_m^+\parens{\lang^{(1)}})$ is unique, then we again  define $w_{m+1}^+ = A^{-1}(w_m^+)$  and $w_{m+1}^- = A^{-1}(w_m^-)$. If $A^{-1}(w_m^+\parens{\lang^{(1)}})$ is not unique, then by the results of Subsection \ref{sec:eee-ext-img}, we know $ E^R\parens{w_m^+\parens{\lang^{(1)}}} = \{1,2,3\} $, $k_0=0$, and $A^{-1}\parens{w_m^+\parens{\lang^{(1)}}}$ has two elements, differing only by a suffix of $1^{k_0}.$
We set $w_{m+1}^+$ and $ w_{m+1}^-$ to be these two elements, with 
$$w_{m+1}^-= w_{m+1}^+ 1^{k_0}.$$
All our work so far has shown that in all cases, for all $m$,
$$0 \leq m(w_m^+) = -m(w_m^-) \leq 1.$$

Given any bispecial $ w \in \lang $, we know that by repeatedly applying the map $A$ to $w$, we will eventually get to $\emptyword$ or some $1^k$.  All of the above work produces all possible non-neutral descendents of $\emptyword$ and of $1^k$. In particular, $w$ must be in this sequence.

\subsection{ Showing $|w_m^+| \leq |w_m^-| < |w_{m+1} ^+| $}\label{Bounds on $w_m$}

This section has a different feel than the earlier sections. Our goal is to prove the following:

\begin{proposition} \label{eeeseqsbounding}
    We have 
    $$ \abs{w_0^+(\lang)} \leq \abs{w_0^-(\lang)} < \abs{w_1^+(\lang)} \leq \abs{w_1^-(\lang)} < \abs{w_2^+(\lang)} \leq \abs{w_2^-(\lang)} < \cdots. $$
\end{proposition}
We will do this through using linear algebra on the abelianizations of these finite words, as well as the results of the extension diagrams.

Before giving the proof, which is at the end of this section after the proof of Lemma \ref{biprefix}, 
 we need to  set some notation. As defined in subsection \ref{sec:abelianization}, for a factor $w$, we write the ``abelianization" of $w$ as
$$ \ell(w) = \begin{bmatrix} \abs{w}_1 \\ \abs{w}_2 \\ \abs{w}_3 \end{bmatrix}, $$
where $\abs{w}_i$ is the number of $i$'s in $w$.

Now we define $a_m^+, a_m^-, b_m^+, b_m^-$ so that
\begin{equation*}
    w_m^+(\lang) = a_m^+ w_{m-1}^+(\lang^{(1)}) b_m^+, \qquad w_m^-(\lang) = a_m^- w_{m-1}^-(\lang^{(1)}) b_m^-.
\end{equation*} 
The incidence matrix of the substitution $G_{k_0}$ is
\begin{equation}
    M_{k_0} := \begin{bmatrix} 0 & k_0 & 1 + k_0 \\ 1 & 0 & 0 \\ 0 & 1 & 1 \end{bmatrix}.
\end{equation}
Therefore, we have
\begin{equation}
    \bfl(w_m^+(\lang)) = M_{k_0} \bfl(w_{m-1}^+(\lang^{(1)})) + \bfl(a_m^+) + \bfl(b_m^+).
\end{equation}
Since $a_m^+$ and $b_m^+$ may take only one of two possible values, by defining
\begin{equation} \label{eq:eeeABpdef}
    A_m^+(\lang) := \begin{cases} 1, & a_m^+ = 1^{k_0} 3 \\ 0, & a_m^+ = \emptyword, \end{cases} \qquad
    B_m^+(\lang) := \begin{cases} 1, & b_m^+ = 1^{k_0} \\ 0, & b_m^+ = \emptyword, \end{cases}
\end{equation}
we have
\begin{equation} \label{eee-lenrecurp}
    \bfl(w_m^+(\lang)) = M_{k_0} \bfl(w_{m-1}^+(\lang^{(1)})) + A_m^+(\lang) \begin{bmatrix} k_0 \\ 0 \\ 1 \end{bmatrix} + B_m^+(\lang) \begin{bmatrix} k_0 \\ 0 \\ 0 \end{bmatrix}.
\end{equation}
We may similarly define
\begin{equation} \label{eq:eeeABndef}
    A_m^-(\lang) := \begin{cases} 1, & a_m^- = 1^{k_0} 3 \\ 0, & a_m^- = \emptyword, \end{cases} \qquad
    B_m^-(\lang) := \begin{cases} 1, & b_m^- = 1^{k_0} \\ 0, & b_m^- = \emptyword. \end{cases}
\end{equation}
Then,
\begin{equation} \label{eee-lenrecurn}
    \bfl(w_m^-(\lang)) =MG_{k_0} \bfl(w_{m-1}^-(\lang^{(1)})) + A_m^-(\lang) \begin{bmatrix} k_0 \\ 0 \\ 1 \end{bmatrix} + B_m^-(\lang) \begin{bmatrix} k_0 \\ 0 \\ 0 \end{bmatrix}.
\end{equation}
We use the relations (\ref{eee-lenrecurp}) and (\ref{eee-lenrecurn}) to prove Proposition \ref{eeeseqsbounding}.
Our previous work immediately allows us to relate $A_m^+, A_m^-$ to $A_{m+1}^+, A_{m+1}^-$:
\begin{lemma}
    \label{biprefix}
    For all $ m \geq 1 $,
    \begin{equation} \label{aevenodd}
        a_m^+ = a_m^- = \begin{cases} \emptyword, & m \text{ is even}, \\ 1^{k_0} 3, & m \text{ is odd}. \end{cases}
    \end{equation}
\end{lemma}

\begin{proof}

If the age $m$ is even, then by 
    Lemma \ref{lem:eee-left-ext} we have that 
    $$E^{-1} = \{2,3\}$$
    which by Proposition \ref{prop:ext-img-list} gives us that 
    $$a_m^+ = a_m^- = \emptyword.$$
    Similarly, if $m$ is odd,  by 
    Lemma \ref{lem:eee-left-ext} we have that 
    $E^{-1} = \{1,3\}$
    which by Proposition \ref{prop:ext-img-list} gives us that 
    $$a_m^+ = a_m^- = 1^{k_0} 3.$$
\end{proof}

\begin{proof}[Proof of Proposition \ref{eeeseqsbounding}]
    We define
    \begin{equation}
        R_m^+(\lang) := \begin{bmatrix} 1 & 1 & 1 \\ 0 & 1 & 1 \\ 0 & 0 & 1 \end{bmatrix} \bfl(v_m^+(\lang)), \qquad R_m^-(\lang) := \begin{bmatrix} 1 & 1 & 1 \\ 0 & 1 & 1 \\ 0 & 0 & 1 \end{bmatrix} \bfl(v_m^-(\lang)).
    \end{equation}
    Then, $ \abs{w_m^+} $ and $ \abs{w_m^-} $ are the first components of $R_m^+$ and $R_m^-$, respectively. We use the partial ordering as follows:
    \begin{align*} (x_1,x_2,x_3) \leq (y_1,y_2,y_3) \iff  x_i \leq y_i \text{ for all } i \in \set{1,2,3}. \end{align*}
    By applying this conjugacy to (\ref{eee-lenrecurp}) and (\ref{eee-lenrecurn}), we have
    \begin{equation} \label{eee-Rlenineq1}
        R_{m+1}^+(\lang) = \begin{bmatrix} 1 & k_0 & 1 \\ 1 & 0 & 0 \\ 0 & 1 & 0 \end{bmatrix} R_m^+(\lang^{(1)}) + \begin{bmatrix} k_0 + 1 \\ 1 \\ 1 \end{bmatrix} A_{m+1}^+(\lang) + \begin{bmatrix} k_0 \\ 0 \\ 0 \end{bmatrix} B_{m+1}^+(\lang)
    \end{equation}
    and
    \begin{equation} \label{eee-Rlenineq2}
        R_{m+1}^-(\lang) = \begin{bmatrix} 1 & k_0 & 1 \\ 1 & 0 & 0 \\ 0 & 1 & 0 \end{bmatrix} R_m^-(\lang^{(1)}) + \begin{bmatrix} k_0 + 1 \\ 1 \\ 1 \end{bmatrix} A_{m+1}^-(\lang) + \begin{bmatrix} k_0 \\ 0 \\ 0 \end{bmatrix} B_{m+1}^-(\lang).
    \end{equation}

    We complete the proof in two parts by proving that for all  languages $\lang$ and all $ m \in \NatZero $,
    \begin{enumerate}[(a)]
        \item $ R_m^-(\lang) \geq R_m^+(\lang) $,
        \item $ R_{m+1}^+(\lang) > R_m^-(\lang) $, with strict inequality in the first entry.
    \end{enumerate}
    This will imply the respective inequalities on $\abs{w_m^+}$ and $\abs{w_m^-}$.
    
    To prove (a), first notice that by Lemma \ref{biprefix}, $ A_m^+ = A_m^- = 1 $ for all odd $m$, and $ A_m^+ = A_m^- = 0 $ for all even $m$.
    Therefore, in subtracting (\ref{eee-Rlenineq1}) from (\ref{eee-Rlenineq2}), the terms involving $A_{m+1}^-$ and $A_{m+1}^+$ cancel, so considering all possible values for $B_{m+1}^+$ and $B_{m+1}^-$, we have
    \begin{equation} \label{eee-Rlenineq3}
        R_{m+1}^-(\lang) - R_{m+1}^+(\lang) \geq \begin{bmatrix} 1 & k_0 & 1 \\ 1 & 0 & 0 \\ 0 & 1 & 0 \end{bmatrix} \parens{R_m^-(\lang^{(1)}) - R_m^+(\lang^{(1)})} - \begin{bmatrix} k_0 \\ 0 \\ 0 \end{bmatrix},
    \end{equation}
    with equality in the worst case scenario where $ B_{m+1}^- = 0 $ and $ B_{m+1}^+ = 1 $.
    To prove (a), we show by induction on $m$ that for all $ m \in \NatZero $ and all  sequences, one of the following holds:
    \begin{enumerate}[(i)]
        \item $ w_m^+(\lang) = w_m^-(\lang) $.
        \item $ R_m^-(\lang) - R_m^+(\lang) = \begin{bmatrix} 1 \\ 0 \\ 0 \end{bmatrix} $, $ E^R(w_m^+(\lang)) = \set{1,2} $, and $ E^R(w_m^-(\lang)) = \set{1,3} $.
        \item $ R_m^-(\lang) - R_m^+(\lang) \geq \begin{bmatrix} 1 \\ 1 \\ 0 \end{bmatrix} $.
    \end{enumerate}
    Trivially, by considering the extension diagrams of $w_0^+$, $w_0^-$, $w_1^+$, and $w_1^-$, either (i) or (ii) holds for $m=0$, and either (i) or (iii) holds for $m=1$.

    Suppose one of the above three statements is true for $ w_m^+(\lang^{(1)}) $ and $ w_m^-(\lang^{(1)}) $.
    First consider the case where (i) holds. If $A^{-1}(w_m^+(\lang^{(1)}))$ is unique, then (i) holds again for the age $m+1$ bispecial words in $\lang$. Otherwise, $ b_{m+1}^+ = \emptyword $ and $ b_{m+1}^- = 1^{k_0} $, so by applying the relevant case in Proposition \ref{prop:ext-img-list}, (ii) holds for $m+1$.
    If (ii) holds for $m$, then by investigating the cases in Proposition \ref{prop:ext-img-list}, we must have $ b_{m+1}^+ = \emptyword $, so $ B_{m+1}^+ = 0 $, and instead of (\ref{eee-Rlenineq3}), we have the stronger condition
    $$ R_{m+1}^-(\lang) - R_{m+1}^+(\lang) \geq \begin{bmatrix} 1 & k_0 & 1 \\ 1 & 0 & 0 \\ 0 & 1 & 0 \end{bmatrix} \parens{R_m^-(\lang^{(1)}) - R_m^+(\lang^{(1)})} \geq \begin{bmatrix} 1 & k_0 & 1 \\ 1 & 0 & 0 \\ 0 & 1 & 0 \end{bmatrix} \begin{bmatrix} 1 \\ 0 \\ 0 \end{bmatrix} = \begin{bmatrix} 1 \\ 1 \\ 0 \end{bmatrix}, $$
    so (iii) holds for $m+1$.
    If (iii) holds for $m$, then applying (\ref{eee-Rlenineq3}), we have
    $$ R_{m+1}^-(\lang) - R_{m+1}^+(\lang) \geq \begin{bmatrix} 1 & k_0 & 1 \\ 1 & 0 & 0 \\ 0 & 1 & 0 \end{bmatrix} \begin{bmatrix} 1 \\ 1 \\ 0 \end{bmatrix} - \begin{bmatrix} k_0 \\ 0 \\ 0 \end{bmatrix} = \begin{bmatrix} 1 \\ 1 \\ 1 \end{bmatrix}, $$
    so (iii) holds again for $m+1$. This completes the proof of (a).
    
    To prove (b), we first compute directly the first few terms of the sequences $w_m^+(\lang)$ and $w_m^-(\lang)$:
    \begin{itemize}
        \item $ w_0^-(\lang) = 1^{k_0} $
        \item $ w_1^+(\lang) = 1^{k_0}3 $
        \item $ w_1^-(\lang) = 1^{k_0}32^{k_1} $
        \item $ w_2^+(\lang) = 2^{k_1} 1^{k_0+1} 3 $
    \end{itemize}
    This proves (b) for $m=0$ and $m=1$.
    To complete the proof, we show that for $ m \geq 1 $ and all languages $\lang$,
    \begin{equation} \label{eq:eee-Rlenineq4}
        R_{2m+1}^+(\lang) - R_{2m}^-(\lang) \geq \begin{bmatrix} 2 \\ 2 \\ 1 \end{bmatrix}, \qquad R_{2m}^+(\lang) - R_{2m-1}^-(\lang) \geq \begin{bmatrix} 1 \\ 0 \\ 0 \end{bmatrix}.
    \end{equation}
    In a similar manner as to how the expression (\ref{eee-Rlenineq3}) is derived, we have
    \begin{equation} \label{eee-Rlenineq5}
        R_{m+2}^+(\lang) - R_{m+1}^-(\lang) \geq \begin{bmatrix} 1 & k_0 & 1 \\ 1 & 0 & 0 \\ 0 & 1 & 0 \end{bmatrix} \left(R_{m+1}^+(\lang^{(1)}) - R_m^-(\lang^{(1)})\right) + \left(A_{m+2}^+(\lang) - A_{m+1}^+(\lang)\right)\begin{bmatrix} k_0 + 1 \\ 1 \\ 1 \end{bmatrix} - \begin{bmatrix} k_0 \\ 0 \\ 0 \end{bmatrix}
    \end{equation}
    for $ m \geq 1 $, with equality in the worst case scenario where $ B_{m+1}^+ = 0 $ and $ B_m^- = 1 $.
    By Proposition \ref{biprefix}, $ A_m^+ = A_m^- $ are both equal to $0$ when $m$ is even and both equal to $1$ when $m$ is odd. Therefore,
    \begin{equation} \label{eee-abelbcond1}
        R_{2m+2}^+(\lang) - R_{2m+1}^-(\lang) \geq \begin{bmatrix} 1 & k_0 & 1 \\ 1 & 0 & 0 \\ 0 & 1 & 0 \end{bmatrix} \parens{R_{2m+1}^+(\lang^{(1)}) - R_{2m}^-(\lang^{(1)})} - \begin{bmatrix} 2k_0 + 1 \\ 1 \\ 1 \end{bmatrix}
    \end{equation}
    and
    \begin{equation} \label{eee-abelbcond2}
        R_{2m+1}^+(\lang) - R_{2m}^-(\lang) \geq \begin{bmatrix} 1 & k_0 & 1 \\ 1 & 0 & 0 \\ 0 & 1 & 0 \end{bmatrix} \parens{R_{2m}^+(\lang^{(1)}) - R_{2m-1}^-(\lang^{(1)})} + \begin{bmatrix} 1 \\ 1 \\ 1 \end{bmatrix}.
    \end{equation}
    The above values of $w_2^+(\lang)$ and $w_1^-(\lang)$ imply
    $$ \bfl(w_2^+(\lang)) - \bfl(w_1^-(\lang)) = \begin{bmatrix} 1 \\ 0 \\ 0 \end{bmatrix}, \qquad R_2^+(\lang) - R_1^-(\lang) = \begin{bmatrix} 1 \\ 0 \\ 0 \end{bmatrix}, $$
    which satisfies the second inequality of (\ref{eq:eee-Rlenineq4}).
    With this as a base case, both inequalities of (\ref{eq:eee-Rlenineq4}) can be proven by induction through alternating applications of (\ref{eee-abelbcond1}) and (\ref{eee-abelbcond2}).
\end{proof}

\subsection{Proving Theorem \ref{Main theorem for $eee$}, the $3n$ bound }\label{proof of main theorem}

We are finally ready to prove that the factor complexity $p_{\mathcal{L}}(n)$ is bounded above by $3n$ (Theorem \ref{Main theorem for $eee$}), the whole goal of this long section.  (As mentioned at the beginning of Section \ref{sec:e-e-e}, the lower bound of $p_{\mathcal{L}}(n)\geq 2n+1$ 
 stems from quite general principles, as shown by R. Tijdeman \cite{Tijdeman}.)

We need one more lemma:
\begin{lemma}
We have 
$$ \sum_{\substack{w \in \lang \\ \abs{w} \leq n-1 \\ w \text{ non-neutral}}} m(w) \in \set{0,1} $$
    
\end{lemma}

\begin{proof}
    Our sequence $w_0^+, w_0^-, w_1^+, w_1^-, \ldots , w_n^+, w_n^-, \ldots $ was created precisely to prove this. The terms in the sequence have non-decreasing lengths, and each time some $w_n^+$ has multiplicity $ m(w_n^+) \neq 0 $, we know not only that  $m(w_n^+)=1$ but also that for the very next term $w_n^-$ we have $m(w_n^-)=-1$.
This means that the above sum will always alternate between $0$ and $1$, giving us our result.
\end{proof}

Now for the grand finale:

\begin{proof}[Proof of Theorem \ref{Main theorem for $eee$}]
 Refer back to Corollary  \ref{cor:bi-complex}, which states that for any language  
    $\lang$,
    \begin{equation} 
        p_\lang(n+1) - p_\lang(n) = \abs{\lang_1} - 1 + \sum_{\substack{w \in \lang \\ \abs{w} \leq n-1}} m(w).
    \end{equation}
    In particular,
    \begin{equation} \label{bi-complex-2}
        p_\lang(n+1) - p_\lang(n) \leq \abs{\lang_1} - 1 + \sum_{\substack{w \in \lang \\ \abs{w} \leq n-1 \\ w \text{ bispecial}}} m(w),
    \end{equation}
    where equality holds if $\lang$ is extendable.

    For our $(eee)$ language, this means that 
    $$p_{\mathcal{L}}(n+1) \leq p_{\mathcal{L}}(n) + |\lang_1| - 1 + \sum_{\substack{w \in \lang \\ \abs{w} \leq n-1}} m(w).$$
    By the above lemma, this means that for all $n$ we have
    $$p_{\mathcal{L}}(n+1) \leq p_{\mathcal{L}} + |\lang_1|.$$
    In our language, we know that all three subwords of length one, namely $1, 2, 3,$ can appear. Thus 
    $$|\lang_1| = 3,$$
    giving us 
$$p_{\mathcal{L}}(n+1) \leq p_{\mathcal{L}}(n) +3.$$
We now simply induct over $n$. As $p_{\mathcal{L}}(0) =0$ , the base case simply is stating that $p_{\mathcal{L}}(1) \leq 3$, which of course we know to be true since $p_{\mathcal{L}}(1) =|\lang_1| = 3. $ 
Then if $p_{\mathcal{L}}(n) \leq 3n,$ we have 
$$p_{\mathcal{L}}(n+1) \leq p_{\mathcal{L}}(n) + 3 \leq 3n + 3 = 3(n+1),$$
as desired.

\end{proof}

\section{Equivalence Classes of TRIP Maps} \label{More general TRIP maps}

The work in finding the complexity bounds for the $T(e,e,e)$ TRIP map was detailed and technical. It is hard for us to imagine that such work would be worthwhile to imitate for the other 215 TRIP maps. However, following the procedure in Section 8.2 of \cite{Garrity-Mccdonald}, certain TRIP maps are equivalent via transformations which preserve the complexities of the associated $S$-adic sequences, reducing the 216 TRIP maps to 21 equivalence classes.

\subsection{Reduction to 36 Maps through Conjugacy} \label{sec:36}

Each permutation $ \rho \in S^3 $ is a bijection on $\A$, so given an infinite sequence $ \bfu = u_0 u_1 u_2 \hdots $ in $\A$, we may apply $\rho$ as a substitution to yield
\begin{align*}
    \rho(\bfu) := \rho(u_0) \rho(u_1) \rho(u_2) \cdots,
\end{align*}
and similarly for finite words.
For any language $\mathcal{L}$, we can define the language $ \rho(\mathcal{L}) := \set{\rho(w): w \in \mathcal{L}} $, whose complexity will be identical to that of $\mathcal{L}$.
This motivates the following proposition:

\begin{proposition}
    If $\mathcal{L}$ is an $S(\rho,\tau_0,\tau_1)$-adic language with coding sequence $ \set{S_{i_n}(\rho,\tau_0,\tau_1)} $, then $\rho(\mathcal{L})$ is $ S(e, \tau_0 \rho, \tau_1 \rho) $-adic with coding sequence $ \set{S_{i_n}(e, \tau_0 \rho, \tau_1 \rho)} $.
\end{proposition}

This reduces the number of classes of $S$-adic languages to consider from 216 to 36.

\subsection{Further Reduction to 21 Maps through Twinning} \label{sec:21}

In this section, we show how under certain conditions the 36 distinct classes of $S$-adic sequences can be reduced further through twinning, or interchanging the roles of the $S_0$ and $S_1$ substitutions as well as the $F_0$ and $F_1$ Farey matrices. As mentioned earlier, ``twinning'' was originally described in \cite{Garrity-Mccdonald}.

The dynamics of the TRIP maps and the substitutions will in many ways remain unchanged, and factors of the $S$-adic sequences will be related to factors of their twins by the reversal operation:

\begin{definition}
   For any finite word $ v = v_0 v_1 v_2 \cdots v_{n-1} $, we define the \keyterm{reversal} of $v$ to be $ \rev(v) = v_{n-1} \cdots v_2 v_1 v_0 $. Similarly, we define the reversal of a substitution $\sigma$, $\rev(\sigma)$, by $ \rev(\sigma)(i) = \rev(\sigma(i)) $, and the reversal of a language $\mathcal{L}$ by $ \rev(\mathcal{L}) := \set{R(w): w \in \mathcal{L}} $.
\end{definition}

Clearly, the complexities of $\mathcal{L}$ and $R(\mathcal{L})$ are identical.

The \keyterm{twin} TRIP map to $T(\sigma,\tau_0,\tau_1)$ is the map $ T(\sigma (13), (12) \tau_1, (12) \tau_0) $.
This TRIP map has the property that if $\set{i_n}_{n=0}^\infty$ is the Farey coding sequence of a point $ x \in \triangle $ for the map $T(\sigma,\tau_0,\tau_1)$, then $ \set{1 - i_n}_{n=0}^\infty $ is the coding sequence of $x$ for the twin map $ T(\sigma (13), (12) \tau_1, (12) \tau_0) $.
The associated substitutions have the property that $ S_0(\sigma,\tau_0,\tau_1) = R\parens{S_1(\sigma (13), (12) \tau_1, (12) \tau_0)} $ and $ S_1(\sigma,\tau_0,\tau_1) = R\parens{S_0(\sigma (13), (12) \tau_0, (12) \tau_1)} $.
This gives us the following:

\begin{proposition} \label{prop:twin}
    Let $\mathcal{L}$ be an $S(\sigma,\tau_0,\tau_1)$-adic language with coding sequence $ \set{S_{i_n}(\sigma,\tau_0,\tau_1)} $. Then, $R(\mathcal{L})$ is an $ S(\sigma (13), (12) \tau_1, (12) \tau_0) $-adic language with coding sequence $ \set{S_{1-i_n}(\sigma (13), (12) \tau_1, (12) \tau_0)} $.
\end{proposition}

Of the 36 TRIP maps $T(e,\tau_0,\tau_1)$, six have twins that are their own conjugates. Therefore, Proposition \ref{prop:twin} allows us to reduce the number of classes of $S$-adic languages from 36 to 21. We use the following representatives of the equivalence classes:
\begin{multicols}{5}
    \begin{itemize}
        \item $(e,e,e)$
        \item $(e,12,e)$
        \item $(e,13,e)$
        \item $(e,23,e)$
        \item $(e,123,e)$
        \item $(e,132,e)^*$
        \item $(e,e,12)$
        \item $(e,12,12)$
        \item $(e,13,12)^*$
        \item $(e,23,12)$
        \item $(e,123,12)$
        \item $(e,e,13)$
        \item $(e,12,13)^*$
        \item $(e,23,13)$
        \item $(e,123,13)$
        \item $(e,e,23)$
        \item $(e,23,23)^*$
        \item $(e,123,23)$
        \item $(e,e,123)$
        \item $(e,123,123)^*$
        \item $(e,e,132)^*$
    \end{itemize}
\end{multicols}
The starred permutations represent classes with 6 maps, while the unstarred ones represent classes with 12 maps. We explicitly list all the conjugates and twins in the appendix.



\section{Statements on Complexity for All TRIP Maps} \label{Statements on Complexity for all TRIP maps}

The remainder of this paper is dedicated to exploring classes of $S$-adic sequences corresponding to other TRIP maps with the ultimate vision of characterizing all such classes with an upper bound on complexity of at most $3n$.
We provide examples of $S$-adic sequences with complexity larger than $3n$ for some $n$ to rule out many cases. We will see that only one case remains, which we conjecture to have upper bound $3n$. We provide some insight into possible strategies for proving this in Section \ref{sec:e-23-e}.

By the results of the previous section, it suffices to consider only 21 cases, explicitly listed in Subsection \ref{sec:21}.  While there are morally only 21 cases as far as factor complexity is concerned, each of the 216 TRIP maps has its own distinctness. Thus in the following theorems we will be listing the factor complexities (at least conjecturally) for all possible TRIP maps.

\begin{theorem}
    Let $ (\sigma,\tau_0,\tau_1) $ be one of the following:
    $$ \begin{array}{cccccc}
         (e,e,e)  &  (12,12,12)   &(13, 13, 13)    &  (23,23,23) & (123, 132, 132  &   (132, 123, 123)  \\ 
    
  (13, 12, 12)  & (132, e,e)     & (e, 132, 132)   &  (123, 123, 123)  &  (23, 13, 13) & (12, 23, 23)
\\
    \end{array} $$
    Then, for any $S(\sigma,\tau_0,\tau_1)$-adic language $\lang$ with coding sequence associated to a rationally independent point in $\triangle$, $ 2n+1 \leq p_\lang(n) \leq 3n $.
\end{theorem}

\begin{proof}
    From Theorem \ref{Main theorem for $eee$}, we know the result is true for the $(e,e,e)$ TRIP map.
    The top row in the above are all of the TRIP maps conjugate to $(e,e,e)$, while the second row are their twins.
\end{proof}

Our second class of TRIP maps are all of the conjugates and twins of the TRIP map $T(e,23,23)$, which is the Cassaigne map.

\begin{theorem} \label{Cassaigne}
    Let $(\sigma,\tau_0,\tau_1)$ be one of the following:
    $$ \begin{array}{cccccc}
        (e, 23, 23)       & (  12  , 132   ,  132    )       & (   13 ,   123 , 123     )   & ( 23   ,  e  ,   e   )   & (  123  ,  12  ,  12    )     & (  132  ,  13  ,   13   )  
    \end{array} $$
    Then, the complexity function of any $S(\sigma,\tau_0,\tau_1)$-adic language $\lang$ with coding sequence associated to a rationally independent point in $\triangle$ satisfies $ p_\lang(n) = 2n+1 $.
\end{theorem}

\begin{proof}
    In \cite{Cassaigne}, Cassaigne, Labb\'{e} and Leroy showed this for $(e,23,23)$ (although, as mentioned earlier, they did not use the rhetoric of TRIP maps). The other maps are the conjugates of $(e,23,23)$.  In this case, the operation of twinning will produce no more TRIP maps.
\end{proof}

We now turn to the maps that are called ``degenerate'' in \cite{Garrity-Mccdonald}, as they are essentially maps in $2$ dimensions (as described in Subsection \ref{R2TRIP Maps}) instead of $3$.
As we might expect, the $S$-adic languages corresponding to these maps are essentially reduced to the two-dimensional, Sturmian case. More precisely, we have the following:

\begin{theorem} \label{th:degenerate}
    Let $(\sigma, \tau_0, \tau_1)$ be $(e,12,e)$, $(e,12,13)$, $(e,132,e)$, or any of their conjugates or twins.
    Then, any $S(\sigma,\tau_0,\tau_1)$-adic language $\lang$ with a coding sequence associated to a rationally independent point in $\triangle$ is the union of a single character and the language of a Sturmian sequence over two characters.
\end{theorem}

This implies that the complexity of such $S$-adic languages $\lang$ is
$$ p_\lang(n) = \begin{cases} n+1, & n \neq 1 \\ 3, & n = 1. \end{cases} $$

\begin{proof}[Proof of Theorem \ref{th:degenerate}]

By direct calculation, we have that

    \begin{multicols}{2}
    \begin{itemize}
        \item $ S_0(e,12,e): \quad 1 \mapsto 3, \quad 2 \mapsto 2, \quad 3 \mapsto 13 $
        \item $ S_0(e,12,13): \quad 1 \mapsto 3, \quad 2 \mapsto 2, \quad 3 \mapsto 13 $
        \item $ S_0(e,132,e): \quad 1 \mapsto 13, \quad 2 \mapsto 2, \quad 3 \mapsto 3 $
        \item $ S_1(e,12,e): \quad 1 \mapsto 1, \quad 2 \mapsto 2, \quad 3 \mapsto 13 $
        \item $ S_1(e,12,13): \quad 1 \mapsto 13, \quad 2 \mapsto 2, \quad 3 \mapsto 1 $
        \item $ S_1(e,132,e): \quad 1 \mapsto 1, \quad 2 \mapsto 2, \quad 3 \mapsto 13 $
    \end{itemize}
    \end{multicols}
    All of these substitutions have $2$ as a fixed point and map $1$ and $3$ to words containing only $1$s and $3$s. Therefore, any de-substitution of a finite word $w$ containing the character $2$ with $ \abs{w} \geq 2 $ must also contain the character $2$ and have length at least $2$.  This implies that every word $ w \in \lang $ other than $ w = 2 $ consists only of $1$s and $3$s.

    All of the above pairs of substitutions, when restricted to $1$ and $3$, produce Sturmian sequences. This can be proven using the same methods commonly used to show the equivalence between Sturmian sequences and $S$-adic sequences where $S$ is the pair of traditional Sturmian substitutions,
    $$ \sigma_0: \quad 1 \mapsto 1, \quad 2 \mapsto 21, \qquad \text{and} \qquad \sigma_1: \quad 1 \mapsto 12, \quad 2 \mapsto 2 $$
    (see, for example, Section 1.2 of \cite{Thuswaldner} and Section 6.3 of \cite{Fogg}).
    Hence, $\mathcal{L}$ is the union of a single character, $\set{2}$, and the language of Sturmian words over the alphabet $\set{1,3}$.

\end{proof}

We have analyzed computationally the complexities of $S$-adic sequences for other TRIP maps. Based on our computational results, we have the following:

\begin{theorem} \label{th:large-complexity}
    For any TRIP map $T(\sigma,\tau_0,\tau_1)$ other than $T(e,e,e)$, $T(e,23,23)$, $T(e,23,e)$, $T(e,13,e)$, $T(e,12,e)$, $T(e,12,13)$, $T(e,e,123)$, $T(e,132,e)$, or any of their conjugates or twins, there is an $S(\sigma,\tau_0,\tau_1)$-adic language $\lang$ with $ p_\lang(n) > 3n $ for some $ n \geq 1 $.
\end{theorem}

\begin{proof}
    Let $\sigma_0$ and $\sigma_1$ be the Farey substitutions associated with the TRIP map $T(e,\tau_0,\tau_1)$.
    For any finite sequence $ (i_0, i_1, \hdots, i_{m-1}) $, we have
    $$ w := \parens{\sigma_{i_0} \circ \sigma_{i_1} \circ \hdots \circ \sigma_{i_{m-1}}}(1) \in \lang. $$
    Therefore, to show the existence of an $S(e,\tau_0,\tau_1)$-adic language with complexity greater than $3n$ for some $n$, it suffices to show that there is some coding sequence $ \parens{\sigma_{i_0}, \sigma_{i_1}, \hdots, \sigma_{i_{m-1}}} $ and some $n$ such that the number of distinct factors of length $n$ of $w$, $p_w(n)$, is strictly greater than $3n$.
    Via a computer search, we found the following 14 cases:
    $$ \begin{array}{c|c|c|c}
        (e,\tau_0,\tau_1) & i_0 i_1 i_2 \hdots i_{m-1} & n & p_{\mathcal{L}}(n) \\ \hline
        (e, 123, e) & 000001010 & 2 & 7 \\
        (e, e, 12) & 01010101010000 & 4 & 13 \\
        (e, 12, 12) & 1110110100 & 3 & 10 \\
        (e, 13, 12) & 0010101010 & 7 & 22 \\
        (e, 23, 12) & 01011100100 & 3 & 10 \\
        (e, 123, 12) & 000000101 & 2 & 7 \\
        (e, e, 13) & 010110111 & 3 & 10 \\
        (e, 23, 13) & 11010110101 & 4 & 13 \\
        (e, 123, 13) & 00000011 & 2 & 7 \\
        (e, e, 23) & 101010100010 & 3 & 10 \\
        (e, 123, 23) & 00000000000 & 2 & 7 \\
        (e, e, 123) & 11110010 & 2 & 7 \\
        (e, 123, 123) & 000001101 & 2 & 7 \\
        (e, e, 132) & 0110011001001 & 6 & 19 \\
    \end{array} $$
    
    As discussed in Section \ref{More general TRIP maps}, the problem of determining complexities for all classes of $S$-adic sequences corresponding to TRIP maps can be reduced to just $21$ cases. The cases given in the theorem are precisely those excluded from the above list of counter-examples.
\end{proof}

By our above theorems, we have completely characterized the TRIP maps where the complexity of the corresponding $S$-adic sequences are bounded above by $3n$, with the exception of the maps $(e,23,e)$ and $(e,13,e)$ and their twins and conjugates.

In Section \ref{sec:e-13-e}, we turn to a class of TRIP maps exhibiting somewhat surprising dynamical properties, which we call ``hidden $\R^2$ behavior.''  These maps have regions where the dynamics mirror that of two dimensional continued fractions and regions where the dynamics do not.  We will prove the following:

\begin{theorem} \label{conj:e-13-e}
    For the $S(e,13,e)$-adic language $\lang$ and any of its five conjugates or six twins, one of the following is true:
    \begin{enumerate}[(a)]
        \item $ p_\lang(n) = 2n+1 $.
        \item $ p_\lang(n) = \min\set{2n+1, n+c} $ for some $ c \in \N $.
        \item $ p_\lang(n) = \min\set{2n+1, n+c_1, c_2} $ for some $ c_1,c_2 \in \N $.
    \end{enumerate}
\end{theorem}

We also pose the following conjecture:

\begin{conjecture} \label{conj:e-23-e}
    The complexity function $p_\lang$ for  $S(e,23,e)$-adic language $\lang$ (and any of its five conjugates or six twins) satisfies the bound $ p_\lang(n) \leq 3n $.
\end{conjecture}

We ran computational experiments generating and computing complexities of $S(e,23,e)$-adic sequences.
Specifically, we started with $ w = 1 $, then replaced $w$ with a random choice of either $S_0(w)$ or $S_1(w)$ until the length of $w$ was at least $10000$. We then truncated $w$ to $10000$ characters and computed its complexity function $p_w(n)$ for $ n \leq 500 $.
We computed $500$ such words $w$ and found that all of their complexities satisfied $ p_w(n) \leq p_w(n-1) + 3 $ for $ 2 \leq n \leq 500 $, consistent with Conjecture \ref{conj:e-23-e}. We suggest a strategy of proof for Conjecture \ref{conj:e-23-e} in Section \ref{sec:e-23-e}.

We include one final remark about the complexity of TRIP map sequences as defined in this paper.  Note that our original choice of abelianizations $S_0$ and $S_1$ in Section \ref{sec:sub-cont-frac} is arbitrary, as we could have chosen $ 3 \mapsto 31 $ instead of $ 3 \mapsto 13 $, but swapping $13$ with $31$ in both $S_0$ and $S_1$ will not affect the complexity by the symmetry noted in Section \ref{sec:36}. However, replacing $13$ with $31$ in one, but not both, of the substitutions $S_0$ and $S_1$ will yield different complexity results.

We now show that this does not produce low complexities for any TRIP maps except the degenerate cases in Theorem \ref{th:degenerate}.

\begin{theorem} \label{th:13-vs-31}
    Suppose we were to define the substitutions associated to the TRIP map $T(e,e,e)$ as
    \begin{itemize}
        \item $ S_0(e,e,e): \quad 1 \mapsto 2, \quad 2 \mapsto 3, \quad 3 \mapsto 13, $
        \item $ S_1(e,e,e): \quad 1 \mapsto 1, \quad 2 \mapsto 2, \quad 3 \mapsto 31, $
    \end{itemize}
    and modify the definition of each $S_i(\sigma,\tau_0,\tau_1)$ accordingly. Then, for all TRIP maps $T(\sigma,\tau_0,\tau_1)$ except $T(e,12,e)$, $T(e,12,13)$, $T(e,132,e)$, and their conjugates and twins, there is some $S(\sigma,\tau_0,\tau_1)$-language $\lang$ and some $n$ such that $ p_\lang(n) > 3n $.
\end{theorem}

\begin{proof}
    Just as in the proof of Theorem \ref{th:large-complexity}, it suffices to find some coding sequence $ \sigma_{i_0} \circ \sigma_{i_1} \circ \hdots \circ \sigma_{i_{m-1}} $ and some $n$ for each TRIP map such that the number of distinct factors of length $n$ of the word
    $$ w := \parens{\sigma_{i_0} \circ \sigma_{i_1} \circ \hdots \circ \sigma_{i_{m-1}}}(1), $$
    $p_w(n)$, is strictly greater than $3n$.  Up to conjugacy and twinning, there are 21 possible cases. As the theorem explicitly rules out considering the  $T(e,12,e)$, $T(e,12,13)$ and  $T(e,132,e)$ cases, we simply have to check the following 18 cases:
    
    $$ \begin{array}{c|c|c|c}
        (e,\tau_0,\tau_1) & i_0 i_1 i_2 \hdots i_{m-1} & n & p_{\mathcal{L}}(n) \\ \hline
        (e, e, e) & 110110100 & 2 & 7 \\
        (e, 13, e) & 1001010 & 2 & 7 \\
        (e, 23, e) & 1011011000 & 2 & 7 \\
        (e, 123, e) & 000000010 & 2 & 7 \\
        (e, e, 12) & 01101000000 & 3 & 10 \\
        (e, 12, 12) & 01101000 & 2 & 7 \\
        (e, 13, 12) & 000110110 & 3 & 10 \\
        (e, 23, 12) & 0111010100 & 3 & 10 \\
        (e, 123, 12) & 000000010 & 2 & 7 \\
        (e, e, 13) & 1010001 & 2 & 7 \\
        (e, 23, 13) & 011110111 & 2 & 7 \\
        (e, 123, 13) & 00000001 & 2 & 7 \\
        (e, e, 23) & 01001000000 & 3 & 10 \\
        (e, 23, 23) & 010101000 & 2 & 7 \\
        (e, 123, 23) & 0000000100 & 2 & 7 \\
        (e, e, 123) & 0111110000 & 2 & 7 \\
        (e, 123, 123) & 000000011 & 2 & 7 \\
        (e, e, 132) & 010011111 & 3 & 10 \\
    \end{array} $$
\end{proof}

\section{Complexity for the $T(e,13,e)$ TRIP Map} \label{sec:e-13-e}

Before stating our results on the complexity of $S(e,13,e)$-adic sequences (Theorem \ref{th:e-13-e-complexity}), we first explore some interesting dynamical properties of the TRIP map $T(e,13,e)$. We generalize this to describe a phenomenon of some TRIP maps we call ``hidden $\R^2$ behavior.'' As we shall see in the remainder of this section, this hidden $\R^2$ property seems to have direct implications on the complexities of $S(e,13,e)$-adic languages.

\subsection{The Dynamics of the TRIP map $T(e,13,e)$ and Hidden $\R^2$ Behavior} \label{sec:e-13-e-dyn}

Partition the region $ \triangle = \set{(x,y,z): x \geq 0, \; y \geq 0, \; z \geq 0, \; x+y+z=1} $ into subtriangles
$$ \triangle_k^G = \set{(x,y,z) \in \triangle: k z \leq x < (k+1)z}. $$
Note that this partition excludes the set $ \set{(x,y,z) \in \triangle: z = 0} $, which has measure zero.

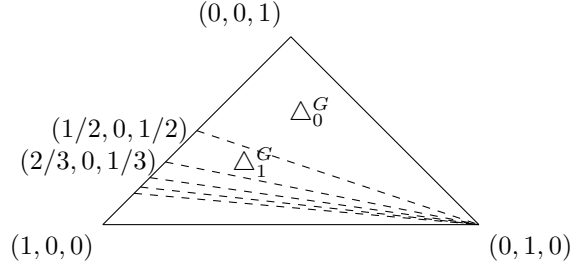
\begin{center}
\begin{tikzpicture}[scale=5]
\draw(0,0)--(1,0);
\draw(0,0)--(1/2,1/2);
\draw(1,0)--(1/2,1/2);
\draw[dashed] (1/4,1/4)--(1,0);
\draw[dashed] (1/6,1/6)--(1,0);
\draw[dashed] (1/8,1/8)--(1,0);
\draw[dashed] (1/10,1/10)--(1,0);
\draw[dashed] (1/12,1/12)--(1,0);

\node[] at (.55,.3){$\triangle_0^G$};
\node[] at (2/5,1/6){$\triangle_1^G$};

\node[below left]at(0,0){$(1,0,0)$};
\node[below right]at(1,0){$(0,1,0)$};

\node[above left] at (1/2,1/2){$(0,0,1)$};
\node[ left]at(1/4,1/4){$(1/2,0, 1/2)$};
\node[ left]at(1/6,1/6){$(2/3,0, 1/3)$};
\end{tikzpicture}
\end{center}Then, the Gauss map $T_G$ associated to $T(e,13,e)$ is
\begin{equation} \label{e-13-e-Gauss}
    \eval{T_G}_{\triangle_k^G} = T_k^G(x,y,z) := \parens{\frac{x-kz}{y+z}, \, \frac{(k+1)z-x}{y+z}, \, \frac{y}{y+z}}.
\end{equation}
The following conditions for a coding sequence $\parens{k_0,k_1,k_2,\hdots}$ will be important in this section:
\begin{itemize}
    \item[(I)] $ k_{2m} = 0 $ for all $ m \geq 0 $
    \item[(II)] $ k_{2m+1} = 0 $ for all $ m \geq 0 $
    \item[(I')] $ k_{2m} = 0 $ for all except finitely many $m$
    \item[(II')] $ k_{2m+1} = 0 $ for all except finitely many $m$
\end{itemize}

From the point of view of random sequences of integers, the above conditions are exceptional, as they require only a finite number of nonzero entries in either the subsequence $\set{k_{2m}}$ or the subsequence $\set{k_{2m+1}}$. However, the following implies that this is not the case from the point of view of the Lebesgue measure on $\triangle$:

\begin{proposition}
    Define the sets 
    \begin{equation*}
        A = \set{(x,y,z) \in \triangle: z \geq x+y }, \qquad B = \set{(x,y,z) \in \triangle: y \geq z}.
    \end{equation*}
    Conditions (I) and (II) are satisfied for all points in
    the sets  $A$ and $B$, respectively, and conditions (I') and (II') are satisfied for all points in $ \bigcup_{m=0}^\infty T_G^{-2m}(A) $ and $ \bigcup_{m=0}^\infty T_G^{-2m}(B) $, respectively.
\end{proposition}

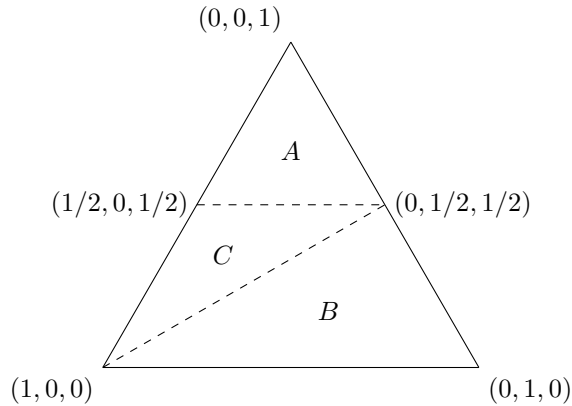
\begin{center}
    \begin{tikzpicture}[scale=5]
        \draw(0,0)--(1,0);
        \draw(0,0)--(.5,.866);
        \draw(1,0)--(.5,.866);
        \draw[dashed] (.25,.433) -- (.75,.433);
        \draw[dashed] (0,0) -- (.75,.433);
        
        \node[] at (.5,.5773) {$A$};
        \node[] at (.6,.15) {$B$};
        \node[] at (.32,.3) {$C$};
        
        \node[below left]at(0,0){$(1,0,0)$};
        \node[below right]at(1,0){$(0,1,0)$};
        
        \node[above left] at (.5,.866){$(0,0,1)$};
        \node[ right]at(.75,.433){$(0,1/2, 1/2)$};
        \node[ left]at(.25,.433){$(1/2, 0,1/2)$};
    \end{tikzpicture}
\end{center}

\begin{proof}
    We may compute directly from (\ref{e-13-e-Gauss}) that $ T_G(A) = B $ and $ T_G(B) \subseteq A $. From the above diagrams we can see that any  point in the set $A$ is in $\triangle_0^G$. Thus if we start with a point in $A$, its image under every second iteration of $T_G$ will also be in the set $A$, giving us $k_{2m} =0$.  The rest of the  proposition follows.
\end{proof}

Define
$$ C := \triangle \setminus (A \cup B) = \set{(x,y,z) \in \triangle: \, y < z, \, z < \frac{1}{2}}. $$
Then, $ T_G^{-1}(C) \subseteq C $, and the set of points which satisfy neither (I') nor (II') are contained in $ \bigcap_{m=0}^\infty T_G^{-m}(C) $.
We conjecture the following:

\begin{conjecture} \label{conj:e-13-e-full-measure}
    There is a subset of $\triangle$ of full Lebesgue measure on which condition (I') or condition (II') is satisfied.
\end{conjecture}

We performed computational experiments on random points in $\triangle$, the results of which support this conjecture. Specifically, we took 1 billion random positive integer points $(x,y,z)$ for which $ x + y + z \leq 2^{64}-1 $ and iterated the unnormalized Gauss map
$$ (x,y,z) \mapsto \parens{x - kz, (k+1)z - x, y} $$
from each point until either $ y > z $ or one of the coordinate variables is zero.
In all 1 billion cases, the condition $ y > z $ was attained before one of the coordinate variables reached zero, which implies that the normalized point $ \parens{\frac{x}{x+y+z}, \frac{y}{x+y+z}, \frac{z}{x+y+z}} $ lies in the interior of some triangular region $W$ on which condition (I') or (II') is satisfied. The verticies of $W$ must be rational points with a common denominator less than $x+y+z$.

We now show that uniqueness of the point in $\triangle$ for a given Gauss sequence, as described in Proposition \ref{freqvector}, does not hold for this TRIP map when condition (I') or (II') is satisfied.

\begin{proposition} \label{prop:e-13-e-R2}
    Define the map $ F: (0,1) \rightarrow [0,1) $ by
    $$ F(\gamma) = \ceil{\frac{1}{\gamma}} - \frac{1}{\gamma}, $$
    or equivalently, on each interval $ I_k := \clop{\frac{1}{k+2}, \frac{1}{k+1}} $, $ F|_{I_k} = F_k $, where
    $$ F_k(\gamma) = k + 2 - \gamma. $$
    If $ F^m(\gamma) \in I_{k_m} $, then we define the $F$-\keyterm{coding sequence} of $ \gamma \in (0,1) $ to be $ \set{k_0,k_1,k_2,\hdots} $.

    Define the projection $ \pi_A: A \setminus \set{(0,0,1)} \rightarrow [0,1] $ by
    \begin{equation} \label{pi-A-def}
        \pi_A(x,y,z) = \frac{y}{x+y}.
    \end{equation}
    Then, for any $ p \in A \setminus \set{(0,0,1)} $, the $(e,13,e)$ Gauss sequence of $p$ is $ \set{0, k_0, 0, k_1, 0, k_2, \hdots} $, where $\set{k_0,k_1,k_2,\hdots}$ is the $F$-coding sequence of $\pi_A(p)$.

    Similarly, if we define the projection $ \pi_B: B \setminus \set{(0,1,0)} \rightarrow [0,1] $ by
    \begin{equation} \label{pi-B-def}
        \pi_{B}(x,y,z) = \frac{z}{x+z},
    \end{equation}
    then the $(e,13,e)$ Gauss sequence of any $ p \in B \setminus \set{(0,1,0)} $ is $ \set{k_0, 0, k_1, 0, k_2, 0, \hdots} $, where $\set{k_0,k_1,k_2,\hdots}$ is the $F$-coding sequence of $\pi_B(p)$.
\end{proposition}

\begin{proof}[Proof of Proposition \ref{prop:e-13-e-R2}]
    We prove the statement for $ p \in B $, as the proof for $ p \in A $ is similar.
    By direct computation, $ \pi_A = \pi_{B} \circ T_0^G $ and $ F_k \circ \eval{\pi_{B}}_{\triangle_k^G} = \pi_A \circ T_k^G $.
    Therefore, $ F_k \circ \pi_{B} = \pi_{B} \circ T_0^G \circ T_k^G $ on $\triangle_k^G$. Thus, $\pi_B$ maps the entire orbit of $T_G$ to the entire orbit of $F$, as shown in the commutative diagram in Figure 1. This reduces the problem of determining the Gauss sequence of a point $ p \in B $ under the TRIP map $T(e,13,e)$ to that of determining the coding sequence of $\pi_{B}(p)$ under $F$.
\end{proof}

This means that on the sets $A$ and $B$, the TRIP map $T(e,13,e)$ is essentially reduced to a  two-dimensional case, as discussed in Subsection \ref{R2TRIP Maps}. The $(e,13,e)$ TRIP map is among an entire class of TRIP maps that exhibit this property. We formalize the phenomenon in a definition:
\begin{definition} \label{hiddenr2def} A \textbf{hidden $\R^2$ TRIP map} is a TRIP map which  has factor complexity of   an $\R^2$ TRIP map (defined in Subsection \ref{R2TRIP Maps})  on a region of positive measure, but not on all of $\Tri$. 
\end{definition}
Note the distinction in the definition between hidden $\R^2$ behavior and the degenerate behavior described in the previous section. For some points on $\Tri$, hidden $\R^2$ maps have genuine higher-dimensional behavior, but on some regions, their dynamics are only as complicated as $\R^2$ TRIP maps, as shown in Figure 1. In the case of the $T(e,13,e)$ TRIP map, the set $C$ also contains regions from which $A$ or $B$ can be reached after finitely many iterations of $T(e,13,e)$, giving it a fractal structure.

\begin{figure}[ht] \label{fig:e-13-e-line-segments}
    \begin{center}
        \begin{tikzpicture}[scale=5]
            \draw(0,0)--(1,0);
            \draw(0,0)--(.5,.866);
            \draw(1,0)--(.5,.866);
            \draw[dashed] (.25,.433) -- (.75,.433);
            \draw[dashed] (0,0) -- (.75,.433);

            \draw (.3,.433) -- (.5,.866);
            \draw (.35,.433) -- (.5,.866);
            \draw (.4,.433) -- (.5,.866);
            \draw (.45,.433) -- (.5,.866);
            \draw (.5,.433) -- (.5,.866);
            \draw (.55,.433) -- (.5,.866);
            \draw (.6,.433) -- (.5,.866);
            \draw (.65,.433) -- (.5,.866);
            \draw (.7,.433) -- (.5,.866);

            \draw (.075,.0433) -- (1,0);
            \draw (.15,.0866) -- (1,0);
            \draw (.225,.1299) -- (1,0);
            \draw (.3,.1732) -- (1,0);
            \draw (.375,.2165) -- (1,0);
            \draw (.45,.2598) -- (1,0);
            \draw (.525,.3031) -- (1,0);
            \draw (.6,.3464) -- (1,0);
            \draw (.675,.3897) -- (1,0);
            
            \node[below left]at(0,0){$(1,0,0)$};
            \node[below right]at(1,0){$(0,1,0)$};
            \node[above left] at (.5,.866){$(0,0,1)$};
        \end{tikzpicture}
    \end{center}
    \caption{The sets $A$ and $B$ are partitioned into line segments, except for the fact that all line segments include the vertex $(0,0,1)$ or $(0,1,0)$. On each line segment, all points have the same $(e,13,e)$ Farey sequence.}
\end{figure}
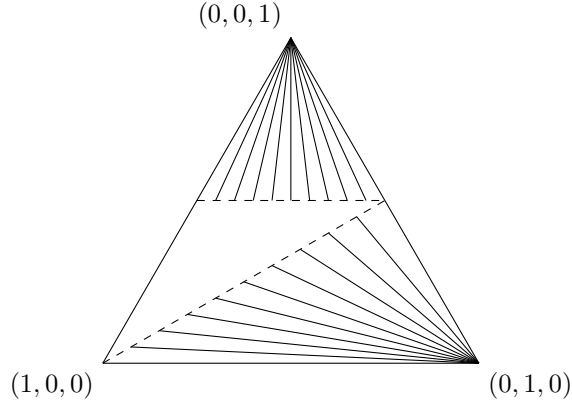

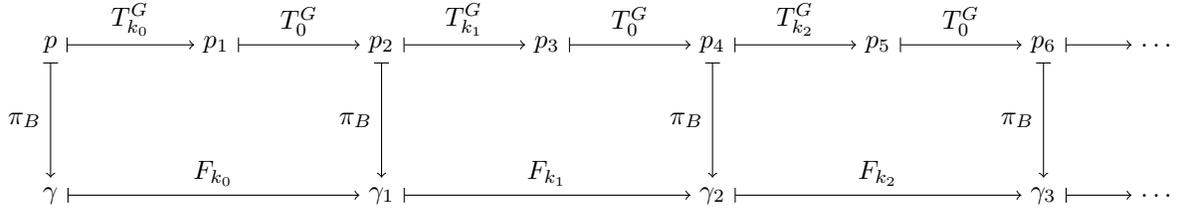
\begin{figure}[ht] \label{fig:e-13-e-commutative-diagram}
    \centering
    \begin{tikzpicture}
        \node(A1) at (0,2) {$p$};
        \node(A2) at (2.2,2) {$p_1$};
        \node(A3) at (4.4,2) {$p_2$};
        \node(A4) at (6.6,2) {$p_3$};
        \node(A5) at (8.8,2) {$p_4$};
        \node(A6) at (11,2) {$p_5$};
        \node(A7) at (13.2,2) {$p_6$};
        \node(A8) at (14.7,2) {$\hdots$};
        \node(B1) at (0,0) {$\gamma$};
        \node(B3) at (4.4,0) {$\gamma_1$};
        \node(B5) at (8.8,0) {$\gamma_2$};
        \node(B7) at (13.2,0) {$\gamma_3$};
        \node(B8) at (14.7,0) {$\hdots$};

        \path[|->] (A1) edge node[above]{$T_{k_0}^G$} (A2);
        \path[|->] (A2) edge node[above]{$T_0^G$} (A3);
        \path[|->] (A3) edge node[above]{$T_{k_1}^G$} (A4);
        \path[|->] (A4) edge node[above]{$T_0^G$} (A5);
        \path[|->] (A5) edge node[above]{$T_{k_2}^G$} (A6);
        \path[|->] (A6) edge node[above]{$T_0^G$} (A7);
        \path[|->] (A7) edge (A8);
        \path[|->] (B1) edge node[above]{$F_{k_0}$} (B3);
        \path[|->] (B3) edge node[above]{$F_{k_1}$} (B5);
        \path[|->] (B5) edge node[above]{$F_{k_2}$} (B7);
        \path[|->] (B7) edge (B8);
        \path[|->] (A1) edge node[left]{$\pi_B$} (B1);
        \path[|->] (A3) edge node[left]{$\pi_B$} (B3);
        \path[|->] (A5) edge node[left]{$\pi_B$} (B5);
        \path[|->] (A7) edge node[left]{$\pi_B$} (B7);
    \end{tikzpicture}
    \caption{The mapping of the Farey sequence of a point $ p \in B $ under $T(e,13,e)$ to the coding sequence of $\pi_B(p)$.}
\end{figure}

It would be good to classify all the TRIP maps with hidden-$\R^2$ behavior; based on preliminary work in \cite{Hidden} we make the following conjecture:
\begin{conjecture} \label{conj:hidden-list}
    The $\R^3$ TRIP maps that have hidden $\R^2$ structure are precisely the maps
    $$ \begin{array}{cccc}
        T(e,e,12)& T(e,13,123)& T(e,132,12) & T(e,13,12)\\
        T(e,12,12)& T(e,e,123)& T(e,e,13) & T(e,12,123)\\
        T(e,13,13)& T(e,13,132)& T(e,13,e) & T(e,123,12)\\
    \end{array} $$
    and all of their conjugates and twins. 
\end{conjecture}

It is reasonable to expect that in these regions where the dynamics of the TRIP map are much simpler the corresponding sequences have lower complexity. In fact, we confirm this in Theorem \ref{th:e-13-e-complexity} for the $(e,13,e)$ map.

\subsection{Factors of Length 2 and Other Basic Results}
We now focus solely on the $(e,13,e)$ case for the remainder of the section.
The substitutions for the TRIP map $T(e,13,e)$ are
\begin{itemize}
    \item $ S_0: \quad 1 \mapsto 13, \quad 2 \mapsto 3, \quad 3 \mapsto 2 $,
    \item $ S_1: \quad 1 \mapsto 1, \quad 2 \mapsto 2, \quad 3 \mapsto 13 $.
\end{itemize}
The Gauss substitutions are
\begin{equation}\label{Gfor  e13e}
    S_k^G: \quad 1 \mapsto 1^{k+1}3, \quad 2 \mapsto 1^k 3, \quad 3 \mapsto 2.
\end{equation}
We express the coding sequence as $ \set{S_{k_0}^G, S_{k_1}^G, S_{k_2}^G, \hdots} $ and let $\lang^{(m)}$ be the $m$-fold Gauss de-substitution of $\lang$, that is, the $S(e,13,e)$-adic language with coding sequence $ \set{S_{k_m}^G, S_{k_{m+1}}^G, S_{k_{m+2}}^G, \hdots} $.
We say that $\mathcal{L}$ satisfies conditions (I), (II), (I'), or (II') if and only if the sequence $\parens{k_0,k_1,k_2,\hdots}$ satisfies the corresponding conditions in Section \ref{sec:e-13-e-dyn}.

Notice that condition (I') or (II') hold for an $S(e,13,e)$-adic language $\mathcal{L}$ if and only if condition (I) or (II), respectively, holds for $\mathcal{L}^{(2m)}$ for some $m$. In particular, (I) implies (I') and (II) implies (II').
Also, $\mathcal{L}$ satisfies (I') or (II') if and only if $\mathcal{L}^{(1)}$ satisfies (II') or (I'), respectively.

We devote the remainder of this section to proving the following, from which Theorem \ref{conj:e-13-e} follows immediately:

\begin{theorem} \label{th:e-13-e-complexity}
    For any $S(e,13,e)$-adic language $\mathcal{L}$:
    \begin{itemize}
        \item If $\mathcal{L}$ satisfies both (I') and (II'), then $ p_\mathcal{L}(n) = \min\set{2n+1, n+c_1, c_2} $ for some $ c_1, c_2 \in \N $.
        \item If $\mathcal{L}$ satisfies (I') or (II'), but not both, then $ p_\mathcal{L}(n) = \min\set{2n+1, n+c} $ for some $ c \in \N $.
        \item If $\mathcal{L}$ satisfies neither (I') nor (II'), then $ p_\mathcal{L}(n) = 2n+1 $.
    \end{itemize}
\end{theorem}

Just as in the previous subsections, we begin by examining the possible factors of length $2$.

\begin{proposition} \label{faclen2e13e}
    For any $S(e,13,e)$-adic language $\lang$:
    \begin{enumerate}[(a)]
        \item $ 13, 32 \in \lang $
        \item $ 12, 22, 33 \notin \lang $
        \item $ 11 \in \lang $ if and only if $ k_0 \geq 1 $
        \item $ 23 \in \lang $ if and only if $ k_0 = 0 $
        \item $ 21 \in \lang $ if and only if $\mathcal{L}$ does not satisfy condition (I).
        \item $ 31 \in \lang $ if and only if $\mathcal{L}$ does not satisfy condition (II).
    \end{enumerate}
\end{proposition}

\begin{proof}
    Statement (a) follows immediately from the fact that $13$ is a subword of $G_k(1)$ and $32$ is a subword of $G_k(13)$.

    To prove (b), first notice that the character $2$ only occurs in (\ref{Gfor  e13e}) as a substitution of $3$. Therefore, $ 12 \in \lang $ is only possible if $1$ is the last character of $S_k^G(c)$ for some character $c$, but it is not, so $ 12 \notin \lang $.

    To show that $ 22, 33 \notin \lang $, it suffices to prove the following by an inductive argument applying Definition \ref{def:lang-coding-seq}:
    \begin{enumerate}[(i)]
        \item For all $ c \in \set{1,2,3} $ and all nonnegative integers $k$, neither $22$ nor $33$ is a subword of $S_k^G(c)$.
        \item For all letters $ c $ in our alphabet, all $ m \geq 2 $, and all non-negative integers $ j_1,j_2,\hdots,j_m $, $22$ is a subword of $ \parens{S_{j_1}^G \circ \hdots \circ S_{j_m}^G}(c) $ if and only if $33$ is a subword of $ \parens{S_{j_1}^G \circ \hdots \circ S_{j_{m-1}}^G}(c) $.
        \item For all letters $ c $ in our alphabet, all $ m \geq 2 $, and all non-negative integers $ j_1,j_2,\hdots,j_m $, $33$ is a subword of $ \parens{S_{j_1}^G \circ \hdots \circ S_{j_m}^G}(c) $ if and only if $22$ is a subword of $ \parens{S_{j_1}^G \circ \hdots \circ S_{j_m}^G}(c) $.
    \end{enumerate}
    Statement (i) follows directly from (\ref{Gfor  e13e}).
    Statement (ii) follows from the fact that $2$ only appears in (\ref{Gfor  e13e}) as a substitution of $3$.
    Since $3$ only appears in (\ref{Gfor  e13e}) as the last character of a substitution of $1$ or $2$, and $3$ is only the first character of a substitution of $2$, $33$ is a subword of $ \parens{S_{j_1}^G \circ \hdots \circ S_{j_m}^G}(c) $ if and only if either $12$ or $22$ is a subword of $ \parens{S_{j_1}^G \circ \hdots \circ S_{j_{m-1}}^G}(c) $.
    However, we have already showed that $12$ is forbidden, so (iii) follows.
    This completes the proof of (b).

    To prove (c), notice that the last character of $S_k^G(c)$ for all $k$ and any single character $c$ is never $1$, so $ 11 \in \lang $ if and only if $11$ is a subword of $S_{k_0}^G(c)$ for some single character $c$. In particular, if $ k_0 \geq 1 $, then this is true for $ c = 1 $, but if $ k_0 = 0 $, then this is not true.

    To prove (d), notice that since $2$ only occurs in (\ref{Gfor  e13e}) as the substitution of $3$, a necessary condition for $ 23 \in \lang $ is that there is some single character $c$ such that $3$ is the first character of $G_{k_0}(c)$. If $ k_0 \geq 1 $, then this is not true, but if $ k_0 = 0 $, then $23$ is a subword of $S_{k_0}^G(32)$ and it has already been established that $ 32 \in \lang^{(1)} $, which implies $ 23 \in \lang $. This proves (d).

    Statements (e) and (f) can be derived by an inductive argument via the following:
    \begin{enumerate}[(i)]
        \item If $ k_0 \geq 1 $, then $ 21 \in \lang $.
        \item If $ k_1 \geq 1 $, then $ 31 \in \lang $.
        \item For all $ c \in \A $, $ m \geq 1 $, and nonnegative integers $ j_1,j_2,\hdots,j_m $, $21$ is a subword of $ \parens{S_0^G \circ S_{j_1}^G \circ \hdots \circ S_{j_m}^G}(c) $ if and only if $31$ is a subword of $ \parens{S_{j_1}^G \circ \hdots \circ S_{j_m}^G}(c) $.
        \item For all $ c \in \A $, $ m \geq 1 $, and nonnegative integers $ j_1,j_2,\hdots,j_m $, $31$ is a subword of $ \parens{S_0^G \circ S_{j_1}^G \circ \hdots \circ S_{j_m}^G}(c) $ if and only if either $11$ or $21$ is a subword of $ \parens{S_{j_1}^G \circ \hdots \circ S_{j_m}^G}(c) $.
    \end{enumerate}
    Statement (i) is a consequence of the fact that $21$ is a subword of $ S_{k_0}^G(32) $ when $ k_0 \geq 1 $.
    Statement (ii) is a consequence of the fact that $31$ is a subword of $ \parens{S_{k_0}^G \circ S_{k_1}^G}(32) $ when $ k_1 \geq 1 $. 

    To prove statements (iii) and (iv), notice that in the substitution $S_0^G$, $21$ can only arise from a substitution of $31$, and $31$ can only arise from a substitution of $21$.

    Now to prove the ``if'' version of (e) and (f). Suppose as inductive hypothesis that for all $ m \le M $, $21$ is a subword of $ \parens{S_{k_0} \circ \hdots \circ S_{k_m}}(c) $ if not all $k_i$, for $i \le m$ and $i$ even, are zero. Likewise for all $n \le M$, $31$ is a subword of $ \parens{S_{k_0} \circ \hdots \circ S_{k_n}}(c) $ if not all $k_i$, for $ i \le n $ and $i$ odd, are zero.

    Now suppose $ \parens{S_{k_0} \circ \hdots \circ S_{k_{M+1}}}(c) $ has some $i \le m+1$ with $k_i \neq 0$, with $i$ even or odd depending on whether we are assuming (I) or (II) holds. Pick the smallest $i$ such that this is the case. If $i <m+1$, then $ \parens{S_{k_{i+1}} \circ \hdots \circ S_{k_{m+1}}}(c) $ contains all characters $\set{1,2,3}$, so by the inductive assumption, $21$ and $31$ are subwords of $ \parens{S_{k_0} \circ \hdots \circ S_{k_{m+1}}}(c) $. If $i = m+1$, then by $(iii)$ and $(iv)$, 21 (respectively 31) will be a subword of $ \parens{S_{k_0} \circ \hdots \circ S_{k_{m+1}}}(c) $ if $m+1$ is even (respectively odd).

    Now, to prove the ``only if'' direction of (e) and (f), suppose for the sake of induction that for all $m,n < M$, $21$ is a subword of $S_{k_0} \hdots S_{k_m} (c)$ only if there is some $k_i>0$ with $i$ even, and $31$ is a subword of $ \parens{S_{k_0} \hdots S_{k_n}}(c) $ only if there is some $k_i >0$ with $i $ odd. Suppose $21$ is a subword of $ \parens{S_{k_0} \circ \hdots \circ S_{k_{M+1}}}(c) $ and all $k_i$, for $i$ even, are zero. This means that $31$ is a subword of $ \parens{S_{k_1} \circ \hdots \circ S_{k_{m+1}}}(c) $, which implies that there is some $k_i >0$ for $i +1 \ge 1$ odd, or $i $ even. This is a contradiction, and so completes the inductive step for $21$.
    The inductive step for $31$ is similar.
\end{proof}

In the next subsection, in order to  sketch the proof of  Theorem \ref{th:e-13-e-complexity}, we will need to characterize the right special factors of an $S(e,13,e)$-adic word and apply Proposition \ref{prop:l-r-complex}. This differs from the standard methods in \cite{ARP,Cassaigne} and the methods we use in Section \ref{sec:e-e-e} and \ref{sec:e-23-e}. We proceed to characterize notions of age, antecedent, and extended image of right special factors.

\subsection{Characterization of Right Special Factors}

\begin{definition} \label{def:e-13-e-right}
    We say that a nonempty right special factor $w$ of an $S(e,13,e)$-adic language $\mathcal{L}$ is:
    \begin{itemize}
        \item of \keyterm{Type 1} if the last character of $w$ is $1$ or $2$
        \item of \keyterm{Type 2} if the last character of $w$ is $3$
    \end{itemize}
\end{definition}

The following follows easily from the results of the previous subsection:

\begin{proposition}
    For any right special factor $w$ of an $S(e,13,e)$-adic language $\mathcal{L}$:
    \begin{itemize}
        \item If $w$ is of Type 1, then $ E^+(w) = \set{1,3} $.
        \item If $w$ is of Type 2, then $ E^+(w) = \set{1,2} $.
    \end{itemize}
\end{proposition}

Given any $ w \in \mathcal{L} $ containing at least one occurrence of the character $2$ or $3$, we may ``de-substitute'' $w$ into an element of $\mathcal{L}^{(1)}$ by inserting a break after every $2$ or $3$ using methods similar to those in Section \ref{sec:eee-length2}.
The following two statements can be proven using the same methods as those used in Sections \ref{sec:eee-ante} and \ref{sec:eee-extensions}:

\begin{proposition} \label{prop:e-13-e-ante}
    Let $w$ be a right special factor of $\lang$. Then, there is some $ c \in \set{\eps, 1, \hdots, 1^{k_0+1}} $ and $ v \in \lang^{(1)} $ such that:
    \begin{itemize}
        \item If $w$ is of Type 1, then $ cw = S_{k_0}^G(v) 1^{k_0} $ and $v$ is a right special factor of $\lang^{(1)}$ of Type 2.
        \item If $w$ is of Type 2, then $ cw = S_{k_0}^G(v) $ and $v$ is a right special factor of $\lang^{(1)}$ of Type 1.
    \end{itemize}
\end{proposition}

\begin{proposition} \label{prop:e-13-e-extimg}
    Let $v$ be a right special factor of $\mathcal{L}^{(1)}$. If $v$ is of Type 1, then $ w = S_{k_0}^G(v) $ is a right special factor of $\mathcal{L}$ of Type 2, while if $v$ is of Type 2, then $ w = S_{k_0}^G(v) 1^{k_0} $ is a right special factor of $\mathcal{L}$ of Type 1.
\end{proposition}

For any $S(e,13,e)$-adic language $\mathcal{L}$, we define sequences of right special factors $\set{v_m(\mathcal{L})}_{m=0}^\infty$ and $\set{w_m(\mathcal{L})}_{m=0}^\infty$ via the following:
\begin{itemize}
    \item $ v_0(\mathcal{L}) = \emptyword $
    \item $ w_0(\mathcal{L}) = \emptyword $
    \item For $ m \geq 1 $, $ v_m(\mathcal{L}) = S_{k_0}^G\parens{w_{m-1}\parens{\mathcal{L}^{(1)}}} 1^{k_0} $
    \item For $ m \geq 1 $, $ w_m(\mathcal{L}) = S_{k_0}^G\parens{v_m\parens{\mathcal{L}^{(1)}}} $
\end{itemize}
We often omit the parameter $\mathcal{L}$ when it is clear from context.

\begin{lemma} \label{lem:e-13-e-ordered}
    The sequences $\set{v_m(\mathcal{L})}$ and $\set{w_m(\mathcal{L})}$ satisfy the following properties:
    \begin{enumerate}
        \item If $ m \leq n $, then $v_m(\mathcal{L})$ is a suffix of $v_n(\mathcal{L})$.
        \item If $ m \leq n $, then $w_m(\mathcal{L})$ is a suffix of $w_n(\mathcal{L})$.
        \item For all $m$, $ v_m = \eps $ or $v_m$ is a right special factor of Type 1.
        \item For all $m$, $ w_m = \eps $ or $w_m$ is a right special factor of Type 2.
    \end{enumerate}
\end{lemma}

Let us first look at an example. If $k_0=1$ and $k_m=0$ for all $ m \geq 1 $, then
$$ v_0(\mathcal{L}^{(   n  ) }) = w_0(\mathcal{L}^{(   n  ) }) = \emptyword $$
for all $n$.
We may compute $ v_1(\mathcal{L}) = S_{k_0}^G (w_{0}(\mathcal{L}^{(1)}) 1^{k_0} = S_{1}^G (\emptyword)) 1^{1} = 1 $ and $ w_1(\mathcal{L}) = S_{1}^G (v_{0}(\mathcal{L}^{(1)}) = \eps $.
Continuing, $ v_2(\mathcal{L}) = S_{k_0}^G (w_{1}(\mathcal{L}^{(2)}) 1^{k_0} = S_{1}^G (\emptyword)) 1^{1} = 1 $ and $ w_2(\mathcal{L}) = S_{1}^G (v_{1}(\mathcal{L}^{(2)}) = \eps $.
This pattern continues, giving us $ v_m(\mathcal{L}) = 1 $ and $ w_m(\mathcal{L}) = \eps $ for all $m$.

\begin{proof}[Proof of Lemma \ref{lem:e-13-e-ordered}]
    To prove statement (1), it suffices to show that $v_m$ is a suffix of $v_{m+1}$. This is trivially true for $m=0$. We may compute
    \begin{equation} \label{v-recur}
        v_{m+1}(\mathcal{L}) = \parens{S_{k_0}^G \circ S_{k_1}^G}\parens{v_m(\mathcal{L}^{(2)})} 1^{k_0},
    \end{equation}
    which implies that if $v_m$ is a suffix of $v_{m+1}$, then $v_{m+1}$ is a suffix of $v_{m+2}$. The statement (a) follows by induction.
    Similarly, (2) can be proven inductively using the recursive relation
    \begin{equation} \label{w-recur}
        w_{m+1}(\mathcal{L}) = \parens{S_{k_0}^G \circ S_{k_1}^G}\parens{w_m(\mathcal{L}^{(2)})} (1^{k_0+1} 3)^{k_1}
    \end{equation}
    Statements (3) and (4) follow inductively from Proposition \ref{prop:e-13-e-extimg}.
\end{proof}

Statements (1) and (2) of Lemma \ref{lem:e-13-e-ordered} allow us to define the limits
\begin{equation}
    v_\infty(\mathcal{L}) := \lim_{m \rightarrow \infty} v_m(\mathcal{L}), \qquad w_\infty (\mathcal{L}):= \lim_{m \rightarrow \infty} w_m(\mathcal{L})
\end{equation}
as either finite or left-infinite words, according to whether or not the respective sequences $\set{v_m}(\mathcal{L})$ and $\set{w_m}(\mathcal{L})$ are bounded in length.

 Theorem \ref{th:e-13-e-complexity} will then follow directly from Proposition \ref{prop:l-r-complex} and the following two statements:

\begin{lemma} \label{lem:e-13-e-complete}
    All Type 1 right special factors of an $S(e,13,e)$-adic language $\mathcal{L}$ are suffixes of $v_\infty(\mathcal{L})$, and all Type 2 right special factors are suffixes of $w_\infty(\mathcal{L})$.
\end{lemma}

\begin{lemma} \label{lem:e-13-e-cases}
   If an $S(e,13,e)$-adic language $\mathcal{L}$ satisfies condition (I') (respectively (II')), then $v_\infty$ (respectively $w_\infty$) is finite. Otherwise, $v_\infty$ (respectively $w_\infty$) is left-infinite.
   Furthermore, if $\mathcal{L}$ satisfies (I) or (II), then $ v_\infty = \eps $ or $ w_\infty = \eps $, respectively.
\end{lemma}

\begin{proof}[Proof of Lemma \ref{lem:e-13-e-complete}]
    It suffices to prove the following for all $ m \geq 1 $:
    \begin{enumerate}
        \item All Type 1 right special factors $w$ of $\mathcal{L}$ for which $w1$ is a subword of $ \parens{S_{k_0}^G \circ S_{k_1}^G \circ \hdots \circ S_{k_{2m-1}}^G}(c) $ for some $ c \in \A $ are suffixes of $v_m$ for some $m$.
        \item All Type 2 right special factors $w$ of $\mathcal{L}$ for which $w1$ is a subword of $ \parens{S_{k_0}^G \circ S_{k_1}^G \circ \hdots \circ S_{k_{2m}}^G}(c) $ for some $ c \in \A $ are suffixes of $w_m$ for some $m$.
    \end{enumerate}

    First consider the case where $m=1$. The only possible Type 1 right special factors $w$ for which $w1$ is a subword of some $S_{k_0}^G(c)$ are $1^j$ for $ 1 \leq j \leq k_0 $. We may compute $ v_1 = 1^{k_0} $. Therefore, (1) holds for $m=1$.

    We complete the proof by induction. We first show that if (1) holds for some $m$, then (2) also holds for $m$.
    Suppose (a) holds for some $m$, and let $w$ be a Type 2 right special factor of $\mathcal{L}$ for which $w1$ is a subword of $ \parens{S_{k_0}^G \circ S_{k_1}^G \circ \hdots \circ S_{k_{2m}}^G}(c) $ for some $ c \in \A $.
    Let $v$ be the right special factor of $\mathcal{L}^{(1)}$ in Proposition \ref{prop:e-13-e-ante}, so that $w$ is a suffix of $S_{k_0}^G(v)$.
    We may ``de-substitute'' $w1$ into $v1$ or $v2$, but since the last character of $v$ must be either $1$ or $2$, $v2$ is forbidden by Proposition \ref{faclen2e13e}. Therefore, $v1$ must be a subword of $ \parens{S_{k_1}^G \circ S_{k_2}^G \circ \hdots \circ S_{k_{2m}}^G}(c) $.
    By the inductive hypothesis, $v$ is a suffix of $v_m(\mathcal{L}^{(1)})$, which implies that $S_{k_0}^G(v)$, and consequently $w$ is a suffix of $w_m(\mathcal{L})$.

    Now suppose that (2) holds for some $m$ and let $w$ be a Type 1 right special factor of $\mathcal{L}$ for which $w1$ is a subword of $ \parens{S_{k_0}^G \circ S_{k_1}^G \circ \hdots \circ S_{k_{2m+1}}^G}(c) $ for some $ c \in \A $.
    Let $v$ be the right special factor of $\mathcal{L}^{(1)}$ in Proposition \ref{prop:e-13-e-ante}, so that $w$ is a suffix of $ S_{k_0}^G(v) 1^{k_0} $.
    We assume that $ v \neq \eps $, so that $v$ is a Type 2 right special factor, as the only case where $ v = \eps $ is $ w = 1^j $ which we already considered.
    Since $w1$ must ``de-substitute'' into $v1$, we have that $v1$ is a subword of $ \parens{S_{k_1}^G \circ S_{k_2}^G \circ \hdots \circ S_{k_{2m+1}}^G}(c) $, so $w1$ is a subword of $ \parens{S_{k_0}^G \circ S_{k_1}^G \circ \hdots \circ S_{k_{2m+1}}^G}(c) $.
    Thus, (1) is true for $m+1$.
\end{proof}

\begin{proof}[Proof of Lemma \ref{lem:e-13-e-cases}]
    We prove the statement relating conditions (I) and (I') to the sequence $\set{v_m}$, as the case for the sequence $\set{w_m}$ is similar.

    One can easily show inductively by (\ref{v-recur}) that if $ k_{2n} = 0 $ for all $ 0 \leq n \leq m-1 $, then $ v_m = \eps $. Thus, if condition (I) is satisfied, then $ v_m = \eps $ for all $m$.
    Also, it can be shown by inductively applying (\ref{v-recur}) that
    \begin{equation} \label{vm-vn}
        v_m(\mathcal{L}) = \parens{S_{k_0}^G \circ S_{k_1}^G \circ \hdots \circ S_{k_{2m-1}}^G}\parens{v_{m-n}(\mathcal{L}^{(2n)})} v_n(\mathcal{L}).
    \end{equation}
    for all $ m \geq n $. Therefore, if condition (I') is satisfied, then the sequence $\set{v_m}$ is eventually constant.

    Now suppose that (I') is not satisfied.
    We may compute $ v_1(\mathcal{L}) = 1^{k_0} $, so $ \abs{v_1(\mathcal{L}^{2n})} = k_{2n} $. By (\ref{vm-vn}), this implies $ \abs{v_{n+1}} \geq \abs{v_n} + k_{2n} $. If (I') is not satisfied, then $ k_{2n} \geq 1 $ for infinitely many $n$, so the sequence $\set{v_m}$ is unbounded in length.
\end{proof}

\begin{proof}[Proof of Theorem \ref{th:e-13-e-complexity}]
    Let $\lang$ be an $S(e, 13,e)$-adic language. If $\lang$ satisfies both $(I')$ and $(II')$, by Lemma \ref{lem:e-13-e-cases}, $v_\infty$ and $w_\infty$ are finite. Since all right special factors are contained as suffixes of $v_\infty$ and $w_\infty$ by Lemma \ref{lem:e-13-e-complete}, there are two, one, or zero right special factors of length $n$ corresponding to whether $n$ is less than the length of both $v_\infty$ and $w_\infty$, just one of them, or neither. Using Proposition \ref{prop:l-r-complex} we see that $p_\lang(n+1) - p_\lang(n) = \sum_{w \in \lang_n} (|E^+(w)| - 1)$. Since $v_\infty$ and $w_\infty$ each have $2$ right special factors, this gives the result. 
    The other two cases follow nearly identically.
\end{proof}

\section{Complexity for the $T(e,23,e)$ TRIP Map} \label{sec:e-23-e}

In this section, we set up a possible approach to proving Conjecture \ref{conj:e-23-e} by a similar method as used in Section \ref{sec:e-e-e}.
However, we do not go through all of the details required to verify the result, as we expect the computations will be more involved than those in Section \ref{sec:e-e-e}.
The Farey substitutions for the $(e,23,e)$ TRIP map are
\begin{itemize}
    \item $ S_0(e,23,e): \qquad 1 \mapsto 2, \qquad 2 \mapsto 13, \qquad 3 \mapsto 3 $, 
    \item $ S_1(e,23,e): \qquad 1 \mapsto 1, \qquad 2 \mapsto 2, \qquad 3 \mapsto 13 $.
\end{itemize}
and the Gauss substitutions are
\begin{equation} \label{e-23-e-Gauss}
    G_k(e,23,e) = S_1^k(e,23,e) S_0(e,23,e): \qquad 1 \mapsto 2, \qquad 2 \mapsto 1^{k+1} 3, \qquad 3 \mapsto 1^k 3.
\end{equation}
In this section, we denote these substitutions simply by $S_0$, $S_1$, and $G_k$, respectively.

A natural place to start is to derive an analog of Proposition \ref{prop:extension}, which, in the $(e,e,e)$ case, relates the extension set of a word $w$ with the extension set of its antecedent $v$. However, it turns out that it is impossible to relate the factors of an $S(e,23,e)$-adic language $\lang$ and their extensions to those of a de-substitution of $\lang$ by a single Gauss substitution, but we can relate them to those of a de-substitution of $\lang$ by two Gauss substitutions. Therefore, we consider the coding sequence of $\lang$ to consist of substitutions of the form
\begin{equation} \label{e-23-e-two-Gauss}
    G_j \circ G_k: \qquad 1 \mapsto 1^{j+1}3, \qquad 2 \mapsto 2^{k+1} 1^j 3, \qquad 3 \mapsto 2^k 1^j 3.
\end{equation}
We write the coding sequence of an $S(e,23,e)$-adic language $\lang$ as $ \set{\sigma_0, \sigma_1, \sigma_2, \ldots} $, where $ \sigma_m = G_{j_m} \circ G_{k_m} $.
We define the $\lang^{(m)}$ to be the language with coding sequence $ \set{\sigma_m, \sigma_{m+1}, \sigma_{m+2}, \ldots} $ so that $\lang$ is the factorial closure of $ (\sigma_0 \circ \sigma_1 \circ \ldots \circ \sigma_{m-1})(\lang^{(m)}) $.

Notice that the last character of $(G_j \circ G_k)(c)$ for $ c \in \set{1,2,3} $ is always $3$, and all preceding characters, if they exist, are $1$s or $2$s.
Using methods similar to those in Subsection \ref{sec:eee-length2}, the following can be proven:

\begin{proposition} \label{prop:e-23-e-ante}
    For any $ w \in \lang $ containing at least one $3$, there is a unique $ v \in \lang^{(1)} $ and $ a,b \in \A^* $ such that $ w = a \sigma_0(v) b $, $a$ is a nonempty suffix of $\sigma_0(c)$ for some $ c \in \set{1,2,3} $, and $b$ is a proper prefix of $\sigma_0(d)$ for some $ d \in \set{1,2,3} $.
\end{proposition}

This means that $a$ must take one of the following forms:
\begin{itemize}
    \item $ a = 1^p 3 $ for some $ 0 \leq p \leq j_0 + 1 $
    \item $ a = 2^p 1^{j_0} 3 $ for some $ 1 \leq p \leq k_0 + 1 $
\end{itemize}
Furthermore, $b$ must take one of the following forms:
\begin{itemize}
    \item $ b = 1^p $ for some $ 0 \leq p \leq j_0 + 1 $
    \item $ b = 2^p $ for some $ 1 \leq p \leq k_0 + 1 $
    \item $ b = 2^{k_0} 1^p $ for some $ 1 \leq p \leq j_0 $
    \item $ b = 2^{k_0+1} 1^p $ for some $ 1 \leq p \leq j_0 $
\end{itemize}
Note that unlike in Section \ref{sec:e-e-e}, we never take $ a = \emptyword $. For all $w$, $a$ will simply be the prefix of $w$ consisting of all characters before and including the first occurrence of the character $3$, while $b$ will be the suffix consisting of all characters after the last occurrence of $3$.

By considering all possibilities for $a$ and $b$ and different cases based on whether or not $j_0$ and $k_0$ are zero, we may formulate an analog of Proposition \ref{prop:eee-ante} giving conditions for when $w$ is contained in an $S(e,23,e)$-adic language $\lang$.
We have the following analog of Proposition \ref{prop:bispecial-a-b}:

\begin{proposition} \label{prop:e-23-e-bispecial-a-b}
    Let $ w \in \lang $ contain at least one occurrence of the character $3$, and $a$ and $b$ be as in Proposition \ref{prop:e-23-e-ante}. If $w$ is left special, then $ a \in \set{1^{j_0}3, 2^{k_0}1^{j_0}3} $, and if $w$ is right special, then $ b \in \set{\emptyword, 1^{j_0}, 2^{k_0}} $.
\end{proposition}

We then have an analog of Proposition \ref{prop:extension}:

\begin{proposition} \label{prop:e-23-e-extension}
    Define the left extension function $ \alpha_{a,k_0}^L: \set{1,2,3} \rightarrow \set{1,2,3,\emptyword} $ and right extension function $ \alpha_{b,j_0,k_0}^R: \set{1,2,3} \rightarrow \set{1,2,3,\emptyword} $ as follows:
    $$ \begin{array}{c|c|c|c}
        \alpha_{a,k_0}^L & \makecell{k_0 = 0 \\ a = 1^{j_0}3} & \makecell{k_0 \geq 1 \\ a = 1^{j_0}3} & \makecell{k_0 \geq 1 \\ a = 2^{k_0}1^{j_0}3} \\ \hline
        1 & 1 & 1 & \emptyword \\
        2 & 2 & 2 & 2 \\
        3 & 3 & 2 & 3
    \end{array} $$
    $$ \begin{array}{c|c|c|c|c|c|c|c}
        \alpha_{b,j_0,k_0}^R & \makecell{j_0 = 0 \\ k_0 = 0 \\ b = \eps} & \makecell{j_0 = 0 \\ k_0 \geq 1 \\ b = \eps} & \makecell{j_0 = 0 \\ k_0 \geq 1 \\ b = 2^{k_0}} & \makecell{j_0 \geq 1 \\ k_0 = 0 \\ b = \eps} & \makecell{j_0 \geq 1 \\ k_0 = 0 \\ b = 1^{j_0}} & \makecell{j_0 \geq 1 \\ k_0 \geq 1 \\ b = \eps} & \makecell{j_0 \geq 1 \\ k_0 \geq 1 \\ b = 2^{k_0}} \\ \hline
        1 & 1 & 1 & \emptyword & 1 & 1 & 1 & \emptyword \\
        2 & 2 & 2 & 2 & 2 & \emptyword & 2 & 2 \\
        3 & 3 & 2 & 3 & 1 & 3 & 2 & 1
    \end{array} $$
    Then, for any $ v \in \lang^{(1)} $, $ a \in \set{1^{j_0}3, 2^{k_0}1^{j_0}3} $, and $ b \in \set{\eps, 1^{j_0}, 2^{k_0}} $, the extension set of $ w = a\sigma_0(v)b $ in $\lang$ is given by
    $$ E(w) = \set{\parens{\alpha_{a,k_0}^L(c), \alpha_{b,j_0,k_0}^R(d)}: (c,d) \in E(v), \, \alpha^L(c) \neq \emptyword, \, \alpha^R(c) \neq \emptyword}. $$
\end{proposition}

By a similar procedure as that in Section \ref{sec:e-e-e}, we expect that there is a function $A$ mapping bispecial words in $\lang$ containing at least one occurrence of the character $3$ to its antecedent bispecial word in $\lang^{(1)}$. We then define the \keyterm{age} of a bispecial word $w$ to be the non-negative integer $m$ such that $A^m(w)$ is defined and does not contain the character $3$.

Just as in Section \ref{sec:e-e-e}, we may formulate a list of all possible bispecial factors of age $0$ and $1$ and their extension sets, then use Proposition \ref{prop:e-23-e-extension} to characterize all possible bispecial factors of all ages and their extension sets.
It remains true that in the case where a bispecial $ v \in \lang^{(1)} $ has a unique bispecial extended image $ w \in \lang $, $E(w)$ can be derived by applying a left permutation and a right permutation to $E(v)$.
However, unlike for the $(e,e,e)$ TRIP map, it is possible for bispecial factors of any age to have three left extensions and three right extensions, so there are more situations to consider where the extended image is not unique and a more detailed analysis is required. In particular, the following situations can occur:
\begin{itemize}
    \item If $ \abs{E^-(v)} = \abs{E^+(v)} = 3 $ and $ k_0 \geq 1 $, then $v$ has up to four bispecial extended images $w$, each with $ \abs{E^-(w)} = \abs{E^+(w)} = 2 $.
    \item If $ \abs{E^-(v)} = \abs{E^+(v)} = 3 $, $ j_0 \geq 1 $, and $ k_0 = 0 $, then $v$ has up to two bispecial extended images $w$, each with $ \abs{E^-(w)} = 3 $ and $ \abs{E^+(w)} = 2 $.
    \item If $ \abs{E^-(v)} = 3 $, $ \abs{E^+(v)} = 2 $, and $ k_0 \geq 1 $, then $v$ has up to two bispecial extended images $w$, each with $ \abs{E^-(w)} = \abs{E^+(w)} = 2 $.
\end{itemize}

To solve Conjecture \ref{conj:e-23-e}, it remains to determine bispecial words in $\lang$ and their extension sets from bispecial words in $\lang^{(1)}$ and their extension sets, characterize all non-neutral bispecial words in $\lang$, and control their lengths.

\begin{remark}
    Just as for Proposition \ref{prop:extension}, we need very delicate properties of the substitutions (\ref{e-23-e-two-Gauss}) and the set of factors that are possible in order for Proposition \ref{prop:e-23-e-extension} to be possible. In this case, we rely heavily on the fact that the character $3$ always occurs as the last character and nowhere else in a double Gauss substitution of a single character.
    
    If particular, an analog of Proposition \ref{prop:extension} for de-substitutions under a single Gauss substitution does not hold. To see this, consider the case where $ j_m = 0 $ and $ k_m \geq 1 $ for all $m$. Then, the above results imply that any bispecial $ w \in \lang $ containing the character $3$ must have exactly two left extensions and two right extensions. However, the factorial closure $\lang'$ of $G_1(\lang)$ will have a coding sequence $ j_m \geq 0 $ and $ k_m = 0 $ for all $m$, so $\lang'$ contains bispecial factors of any age with three left extensions. This cannot happen if a statement analogous to (\ref{e-e-e-extension}) holds.
\end{remark}

\section{Conclusion and Future Directions} \label{sec:conclusion}

Factor complexity is of course only one of many ways to understand $S$-adic sequences. Others are seen, for example, in the recent survey by Thuswaldner \cite{Thuswaldner}. The extent to which other properties, for example, balance, hold for TRIP maps is mostly unknown. There are many different multi-dimensional continued fraction algorithms that are not in the family of 216 TRIP maps. As shown in \cite{SMALLTRIP}, almost all of them are combination triangle partition maps, and hence various combinations of the maps considered in this paper.  How to take known complexity bounds  for, say, two multidimensional continued fraction algorithms and then determine the complexity bounds for some combination of the two  would be important to know.  Possibly the recent techniques developed by Berth\'{e}, Steiner and Thuswaldner \cite{Berthe-Steiner-Thuswaldner} could be used.

\section{Appendix: Conjugates and Twins}

In Section \ref{More general TRIP maps}, we had the notion of equivalence via either conjugacy or twinning:

\begin{enumerate}
\item Conjugacy
$$ (e, \tau_0, \tau_1) \sim_C  (\rho, \tau_0 \rho^{-1}, \tau_1 \rho^{-1})$$

\item Twinning
$$(e , \tau_0, \tau_1) \sim_{\mathcal{T}} ( (13), (12) \tau_1, (12) \tau_0) $$
\end{enumerate}

While all 216 maps are different, as far as factor complexity is concerned, as mentioned earlier, two TRIP maps that are conjugate or twins will have the same factor complexity.  This allowed us to reduce our work to simply 21 cases.

We will now set up 21 cases.  In the first position of the first  row will be one of the above 21 cases.  The rest of the first row will be all the conjugacies of the first term.  In the second row will be the corresponding twins.  We should the 216 distinct terms.  Note that for six of the cases (the ones with a $^*$, the twins do not give us any new TRIP maps.

\begin{enumerate}

\item  
$$  \begin{array}{cccccc}
  (e,e,e)  &  (12,12,12)   &(13, 13, 13)    &  (23,23,23) & (123, 132, 132  &   (132, 123, 123)  \\ 
    
  (13, 12, 12)  & (132, e,e)     & (e, 132, 132)   &  (123, 123, 123)  &  (23, 13, 13) & (12, 23, 23)

 \end{array}$$

\item
$$\begin{array}{cccccc}
(e, 12, e) & (12, e, 12)       & (13, 132, 13) & (23, 123, 23)     & (123, 13  , 132     ) & (132, 23 , 123   ) \\
(13, 12, e) &  ( 132, e, 12) & ( e, 132, 13)  & (123 ,  123, 23   )          & ( 23  , 13  , 132 )              & ( 12  ,  23  ,   123   )

\end{array} $$

\item
$$  \begin{array}{cccccc}
 (e, 13, e)   &  (12, 123, 12 )   &  (13, e, 13)  &  (23, 132, 23 )  & (123, 23, 132)   &  (132, 12, 123)   \\ 
    
  (13, 12, 132)  &   (132, e, 23)  &  (e, 132, 12)  & (123, 123, 13)  & (23, 13, 123)   & (12, 23, e)

 \end{array}$$
 
\item
 $$  \begin{array}{cccccc}
  (e, 23, e)  &   (12, 132, 12) &  (13, 123, 13)  & (23, e, 23)   & (123, 12, 132)  &   (132, 13, 123)  \\ 
    
   (13,12, 123)      &     (132,e, 13)                   &         (e,132, 23)                 &     (123,123,12)              &             (23,13, e)              & (12,23, 132)

 \end{array}$$
 
\item $$  \begin{array}{cccccc}
          (e,123,e)         &   (12,  13,12) & (13,23,13)   & (23,12,23)   &  (123,e,132) &   (132,132,123)   \\ 
    
            (13,12,23)     &         (132,e,132)               &        (  e, 132, 123)                &       (123, 123, e)            &         (23,13,12)                  & (12,23,13)
    
 \end{array}$$
 
 \item
 
 $$  \begin{array}{cccccc}
    (e,132,e)^*    &   (12,23,12)  &  (13,12,13)  &   (23,13,23)  &  (123,123,132) &  (132,e,123)   \\ 
    
  (13,12,13)  &  (132,e,123)   & (e,132,e)   & (123,123,132)  &   (23,123,132)  & (12,23,12)

 \end{array}$$
 
 \item
 
 $$  \begin{array}{cccccc}
  (e,e,12)  &  (12,12,e)   &   (13,13,132) &  (23,23,123) & (123,132,13)  &  (132,123,23)   \\ 
    
  (13,e,12)  &   (132,12,e)  & (e,13,132)   &  (123,23,123) &  (23,132,13) & (12,123,23)

 \end{array}$$
 
 \item
 
 $$  \begin{array}{cccccc}
  (e,12,12)  &  (12,e,e)   &  (13,132,132)  & (23,123,123)   &  (123,13,13) & (132, 23,23)     \\ 
    
   (13,e,e) &   (132,12,12)  &  (e,13,13)  &  (123,23,23) & (23,132,132)  & (12,123,123)

 \end{array}$$
 
 \item
 
 $$  \begin{array}{cccccc}
  (e,13,12)^*  &  (12,123,e)   & (13,e,132)   &  (23,132,123) &  (123,23,13) &  (132,12,23)   \\ 
    
   (13,e,132) &  (132,12,23)   & (e,13,12)   &   (123,23,13)& (23,132,123)  & (12,123,e)

 \end{array}$$
 
 \item  
 
 $$  \begin{array}{cccccc}
 (e, 23, 12)   & (12,132,e)    &   (13,123,132) &  (23,e,123) & (123,12,13)  &  (132, 13, 23)   \\ 
    
   (13, e, 123) &   (132,12,13)  &  (e,13,23)  &  (123,23,12) &  (23,132,e) & (12,123,132)

 \end{array}$$

 \item $$ \begin{array}{cccccc}
   (e, 123, 12) &  (12,13,e)   &   (13,23,132) &  (23,12,123) & (123,e,13)  &  (132,132,23)   \\ 
    
    (13,e,23)&   (132,12,132)  & (e,13,123)   &  (123,23,e) & (23,132,12)  & (12,123,13)

 \end{array}$$
 
 \item
 
 $$  \begin{array}{cccccc}
  (e,e,13)  &  (12,12,123)   &  (13,13,e)  & (23,23,132)  & (123,132,23)  &  (132,123,12)   \\ 
    
   (13,132,12) & (132,23,e)    &  (e,12,132)  & (123,13,123)  & (23,123,13)  & (12,e,23)

 \end{array}$$
 
 \item
 $$  \begin{array}{cccccc}
 (e,12,13)^*   & (12,e,123)    & (13,132,e)   &  (23,123,132) &  (123,13,23) &  (132,23,12)   \\ 
    
   (13,132,e)& (132,23,12)    & (e,12,13)   & (123,13,23)  &  (23,123,132) & (12,e,123)

 \end{array}$$
 
 \item
 $$  \begin{array}{cccccc}
 (e,23,13)   & (12,132,123)    &  (13,123,e)  &(23,e,132)   &  (123,12,23) & (132,13,12)    \\ 
    
   (13,132,123) &  (132,23,13)   &  (e,12,23)  &  (123,13,12) &  (23,123,e) & (12,e,132)

 \end{array}$$
 
 \item
 $$  \begin{array}{cccccc}
 (e,123,13)   & (12,13,123)    &  (13,23,e)  & (23,12,132)  & (123,e,23)  & (132,132,12)    \\ 
    
  (13,132,23)  &   (132,23,132)  &  (e,12,123)  &  (123,13,e) & (23,123,12)  & (12,e,13)

 \end{array}$$
 
 \item
 $$  \begin{array}{cccccc}
   (e,e,23) & (12,12,132)    & (13,13,123)   &  (23,23,e) & (123,132,12)  &  (132,123,13)   \\ 
    
    (13,123,12)&    (132,13,e) & (e,23,132)   &  (123,12,123) &  (23,e,13) & (12,132,23)

 \end{array}$$
 
\item $$\begin{array}{cccccc}
(e, 23, 23)^*       & (  12  , 132   ,  132    )       & (   13 ,   123 , 123     )   & ( 23   ,  e  ,   e   )   & (  123  ,  12  ,  12    )     & (  132  ,  13  ,   13   )    \\
 (  13  ,  123  ,  123    )    & ( 132   , 13   ,  13    )       & (  e  ,  23  , 23     )   & (  123  ,  12  ,  12    )   & (  23  , e   ,  e    )     & (  12  ,  132  ,  132    )   \end{array} $$
 
 \item
 
 $$  \begin{array}{cccccc}
  (e,123,23)  &  (12,13,132)   &  (13,23,123)  & (23,12,e)  &  (123,e,12) & (132,132,13)    \\ 
    
    (13,123,23)&   (132,13,132)  &   (e,23,123) &  (123,12,e) & (23,e,12)  & (12,132,13)

 \end{array}$$
 
 \item
 
 $$  \begin{array}{cccccc}
  (e,e,123) &    (12,12,13) &   (13,13,23) & (23,23,12)  & (123,132, e)  &  (132,123,132)   \\ 
    
   (13,23,12) &  (132,132,e)   &  (e,123,132)  & (123,e,123)  &(23,12,13)  & (12,13,23)

 \end{array}$$
 
 \item
 $$  \begin{array}{cccccc}
 (e,123,123)^*   &  (12,13,13)   &  (13,23,23)  &  (23,12,12) &(123,e,e)   &  (132,132,132)   \\ 
    
   (13,23,23) &   (132,132,132)  & (e,123,123)   & (123,e,e)  & (23,12,12)   & (12,13,13)

 \end{array}$$
 
 \item
 $$  \begin{array}{cccccc}
  (e,e,132)^*  &  (12,12,23)   &  (13,13,12)  &  (23,23,13) & (123,132,123)  &(132,123,e) \\
   (13,13,12) &  (132,123,e)   & (e,e,132)   &  (123,132,123) &  (23,23,13) & (12,12,23)

 \end{array}$$

\end{enumerate}


\begin{thebibliography}{9}

\bibitem{Hidden}  D. Alvarez, A. Bradford, D. Dong, K. Herbst, A. Koltun-Fromm, B. Mintz, O. V. Osterman, and M. Stelow.
On Hidden Continued Fractions in Some Multidimensional Continued Fraction Algorithms, in preparation. 



\bibitem{Allouche-Shallit} J. Allouche and J. Shallit, {\it Automatic sequences:
Theory, applications, generalizations}
Cambridge University Press, Cambridge, 2003. 

\bibitem{stern-trip}  I. Amburg, K.  Dasaratha,  L.  Flapan, T. Garrity, C. Lee, C.  Mihaila, N.  Neumann-Chun, S.  Peluse, and M.  Stoffregen.
Stern Sequences for a Family of Multidimensional Continued Fractions: TRIP-Stern Sequences,
 {\it Journal of Integer Sequences},
Vol. 20
(2017),
Article 17.1.7.


\bibitem{Andrieu-Vivion} M. Andrieux and L Vivion, Minimal complexities for infinite words written with $d$ letters, {\it Lecture Notes in Computer Science}, Vol. 13899, pp 3-13.

\bibitem{AR}
P. Arnoux and G. Rauzy.
Repr\'{e}sentation g\'{e}om\'{e}trique de suites de complexit\'{e} $2n+1$,
\textit{Bulletin de la S.M.F.}, Vol. 119, no. 2 (1991), pp. 199-215.

\bibitem{AD}
A. Avila and V. Delecroix.
Some monoids of Pisot matrices, \textit{Springer Proc. Math. Stat.}, Vol. 285, (2019), pp. 21-30.



\bibitem{GarrityT05}
S. Assaf, L. Chen, T. Cheslack-Postava, B. Cooper, A. Diesl, T. Garrity, M. Lepinski and A. Schuyler.
Dual Approach to Triangle Sequences: A Multidimensional Continued Fraction Algorithm, 
{\it Integers}, Vol. 5 (2005), A08.

\bibitem{BCDGI}
W. Baalbaki, C. Bonanno, A. Del Vigna, T. Garrity, and S. Isola.
On Partition Numbers and Continued Fraction Type Algorithms,
{\it The Ramanujan Journal}, Vol. 63, no. 3 (2024), pp. 873-915.

\bibitem{Baalbaki-Garrity}
W. Baalbaki and T. Garrity.
Generating New Partition Identities via a Generalized Continued Fraction Algorithm,
{\it Electronic Journal of Combinatorics}, Vol. 31, no. 1 (2024).


\bibitem{Berthe-Delecroix}
V. Berth\'{e} and V. Delecroix.
Beyond Substitutive Dynamical Systems: $S$-adic Systems,
{\it RIMS K\^{o}ky\^{u}roku Bessatsu} , B42, [Series of Lecture Notes from RIMS]
Research Institute for Mathematical Sciences (RIMS), Kyoto, 2014(2014), pp. 81-123.

\bibitem{ARP}
V. Berth\'{e} and S. Labb\'{e}.
Convergence and Factor Complexity for the Arnoux-Rauzy-Poincar\'{e} Algorithm,
{\it Combinatorics on Words}, {\it Lecture Notes in Computer Science}, Vol. 8079, Springer (2013), pp. 71-82.

\bibitem{ARP2}
V. Berth\'{e} and S. Labb\'{e}.
Factor complexity of S-adic sequences generated by the Arnoux-Rauzy-Poincar\'{e} algorithm,
{\it Adv. in Appl. Math}, Vol. 63 (2015), pp. 90-130.


\bibitem{Berthe-Rigo}
V. Berth\'{e} and M. Rigo.
{\it Combinatorics, Automata, and Number Theory}.
Encyclopedia of Mathematics and its Applications 135, Cambridge (2010).

 \bibitem{Berthe-Steiner-Thuswaldner}
 V. Berth\'{e}, W. Steiner, J. M. Thuswaldner.
 On the second Lyapunov exponent of some multidimensional continued fraction algorithms,
 {\it Mathematics of Computation}, Vol. 90 (2021), no. 328, pp. 883-905. 


\bibitem{Bonanno-Del Vigna} C. Bonnano and  A. Del Vigna.
Representation and coding of rational pairs on a triangular tree and Diophantine approximation in $\R^2$,
{\it Acta Arithmetica}, Vol. 200, no. 4 (2021), pp, 389-427.

\bibitem {Bonanno-Del Vigna-Munday}  
 C. Bonnano, A. Del Vigna, and S. Munday.
 A slow triangle map with a segment of indifferent fixed points and a complete tree of rational pairs,
 {\it Monatshefte f\"{u}r Mathematik}, Vol. 194 (2021), no. 1, pp. 1-40.



\bibitem{Cassaigne}
J. Cassaigne, S. Labb\'{e}, and J. Leroy.
A Set of Sequences of Complexity $2n+1$,
{\it Combinatorics on Words}, Lecture Notes in Computer Science, Vol. 10432, Springer (2017), pp. 144-156.


\bibitem{Cassaigne2} J. Cassaigne, S. Labb\'{e}, and J. Leroy. Almost everywhere balanced sequences of complexity $2n+1$, {\it Moscow Journal of Combinatorics and Number Theory}, Vol. 11 (2022), no. 4, pp. 287-333.

\bibitem{Coven-Hedlund} E. Coven and G. A. Hedlund. Sequences with minimal block growth,
{\it Mathematical Systems Theory}, Vol.  7 (1973), pp. 138-153.

\bibitem{SMALLTRIP}
K.  Dasaratha,  L.  Flapan, T. Garrity, C. Lee, C.  Mihaila, N.  Neumann-Chun, S.  Peluse, M. Stoffregen.
A Generalized Family of Multidimensional Continued Fractions: Triangle Partition Maps,
{\it International Journal of Number Theory}, Vol 10, no. 8 (2014), pp. 2151-2186.

\bibitem{Cubic}
K. Dasaratha, L. Flapan, T. Garrity C. Lee, C. Mihaila, N. Neumann-Chun, S. Peluse, and M. Stoffregen.
Cubic irrationals and periodicity via a family of multi-dimensional continued fraction algorithms,
{\it Monatshefte f\"{ur} Mathematik}, Vol. 174 (2014), pp. 549-566.


\bibitem{Espinoza} B. Espinoza, The structure of low complexity subshifts, https://arxiv.org/abs/2305.03096.

\bibitem{Ferenczi}
S. Ferenczi.
Rank and symbolic complexity,
\textit{Ergod. Th. \& Dynam. Sys.}, Vol. 16 (1996), no. 4, pp. 663-682.

\bibitem{Fogg}
N.P. Fogg.
{\it Substitutions in dynamics, arithmetics and combinatorics}, Lecture Notes in Mathematics, Vol. 1794, Springer-Verlag, Berlin (2002).

\bibitem{Fougeron-Skripchenko}
C. Fougeron and A. Skripchenko.
Simplicity of spectra of certain mulidimensional fraction algorithms,
{\it Monatshefte f\"{ur} Mathematik}, Vol. 194, no. 4 (2021), pp. 767-787.

\bibitem{Triangle} 
T. Garrity.
On Periodic Sequences for Algebraic Numbers,
{\it Journal of Number Theory}, Vol. 88 (2001), pp. 86-103.

\bibitem{Garrity-Mccdonald}
T. Garrity and P. Mcdonald.
Generalizing the Minkowski question mark function to a family of multidimensional continued fractions,
{\it International Journal of Number Theory}, Vol. 14, no. 9 (2018), pp. 2473-2516.

\bibitem{Karpenkov}
O. Karpenkov.
{\it Geometry of Continued Fractions:
Algorithms and Computations in Mathematics},
Vol 26 (2013), Springer-Verlag.

\bibitem{Goodson}  G. Goodson, {\it Chaotic Dynamics: Fractals, Tilings, and Substitutions}, Cambdridge University Press (2017).


\bibitem{Ito} 
H. Ito.
Self-duality of multidimensional continued fractions,
Arxiv:2203.07887 (2022).

\bibitem{Leroy}
J. Leroy.
$S$-adic characterization of minimal subshifts with first difference of complexity $ 1 \leq p(n+1) - p(n) \leq 2 $,
{\it Discrete Math. Theor. Comput. Sci.}, Vol. 16, no. 1 (2014), pp. 233-286.

\bibitem{Leroy1}
J. Leroy, 
Some improvements in the $S$-adic conjecture,
{\it Advancements in Applied Mathematics}, Vol. 48, no. 1 (2012), pp. 79-98.

\bibitem{Lothaire}
M. Lothaire, {\it Algebraic combinatorics on words}, 
Encyclopedia of Mathematics and its Applications,  Vol. 90
Cambridge University Press, Cambridge, 2002.

\bibitem{SchweigerF08}
A. Messaoudi, A. Nogueira, and F. Schweiger.
Ergodic properties of triangle partitions,
{\it Monatsh. Math.}, Vol. 157 (2009), no. 3, pp. 283-299.

\bibitem{Morse-Hedlund1}
M. Morse and G. Hedlund.
Symbolic Dynamics,
{\it American Journal of Mathematics}, Vol. 60 (1938), pp. 815-866.

\bibitem{Morse-Hedlund2}
M. Morse and G. Hedlund.
Symbolic Dynamics II: Sturmian Trajectories,
{\it American Journal of Mathematics}, Vol. 62, no. 1 (1940), pp. 1-42.



\bibitem{PantiG00} G. Panti, Multidimensional continued fractions and a Minkowski function, {\it Monatshefte f\"{u}r Mathematick}, Vol. 154 (2008), pp. 247-264.

\bibitem{contfrac}
F. Schweiger.
 {\it Multidimensional Continued Fractions},
Oxford University Press (2000).

\bibitem{Schweiger05}
F. Schweiger.
Periodic multiplicative algorithms of Selmer type
{\it Integers}, Vol. 5, no. 1 (2005).

\bibitem{Thuswaldner}
J. Thuswaldner, 
$S$-adic sequences. A bridge between dynamics, arithmetics, and geometry,
{\it Substitution and tiling dynamics: introduction to self-inducing structures}, Lecture Notes in Math., Vol. 2273, Springer,(2020), pp. 97-191.

\bibitem{Tijdeman}  R. Tijdeman, 
On the minimal complexity of infinite words,
{\it Indagationes Mathematicae},  (N.S.) Vol. 10 (1999), no. 1, 123-129.

\end{thebibliography}
\end{document}